\documentclass[12pt]{amsart}

\author{Paul \textsc{Poncet}}
\address{CMAP, \'{E}cole Polytechnique, Route de Saclay, 91128 Palaiseau Cedex, France}

\email{poncet@cmap.polytechnique.fr}

\usepackage[colorinlistoftodos,bordercolor=orange,backgroundcolor=orange!20,linecolor=orange,textsize=scriptsize]{todonotes}
\usepackage{stmaryrd}
\usepackage{amssymb}
\usepackage{amsmath}%
\usepackage{t1enc}
\usepackage{color}
\usepackage[latin1]{inputenc}

\usepackage{calrsfs} 
\usepackage{times} 
\usepackage{bbm}
\usepackage{tikz}
\usetikzlibrary{arrows}

\def\twoheaddownarrow{\rlap{$\downarrow$}\raise-.5ex\hbox{$\downarrow$}}
\def\twoheaduparrow{\rlap{$\uparrow$}\raise.5ex\hbox{$\uparrow$}}

\DeclareMathOperator*{\Max}{Max}
\DeclareMathOperator*{\Min}{Min}
\DeclareMathOperator*{\Irr}{Irr}
\DeclareMathOperator*{\rMax}{rMax}

\DeclareMathOperator*{\ex}{ex}
\DeclareMathOperator*{\cp}{cp}

\DeclareMathOperator*{\car}{car}
\DeclareMathOperator*{\Cop}{Cop}
\DeclareMathOperator*{\kit}{kit}

\newcommand{\cc}{\mathfrak{c}}
\newcommand{\sss}{\mathfrak{s}}
\newcommand{\ttt}{\mathfrak{t}}

\newcommand{\doublesetminus}{\!\setminus\!\!\setminus}

\newtheorem*{theorem*}{Theorem}

\newtheorem{theorem}{Theorem}[section]
\newtheorem{corollary}[theorem]{Corollary}
\newtheorem{proposition}[theorem]{Proposition}
\newtheorem{lemma}[theorem]{Lemma}

\theoremstyle{definition}
\newtheorem{definition}[theorem]{Definition}
\newtheorem{example}[theorem]{Example}

\newtheorem{remark}[theorem]{Remark}
\newenvironment{acknowledgements}[1][]{\par\noindent\textbf{Acknowledgements#1.} }{\par}

\begin{document}

\title{Order-generation in posets and \\ convolution of closure operators}
\date{\today}

\subjclass[2020]{06A06, 
                 06B35, 
                 52A01, 
                 54A05} 

\keywords{poset, semilattice, irreducible element, prime element, order-generation, continuous poset, continuous lattice, preclosure operator, closure operator, abstract convexity, extreme point, Krein--Milman property}

\begin{abstract}
Motivated by the Hofmann--Lawson theorem, which states that every continuous lattice is inf-generated by its irreducible elements, we explore how to represent posets by extreme points with respect to a closure operator. 
For this purpose, we introduce the convolution product of closure operators, and prove that the Krein--Milman property can be transferred from one collection of subsets to another by convolution. 
This result underpins two new representation theorems of topological flavor, which generalize existing ones, even in posets lacking lattice or semilattice structures. 
We also prove a third representation theorem: given a poset equipped with a closure operator $\cc$ with adequate properties, we show that the set of kit points, defined as an extension of compact points, has the Krein--Milman property with respect to the convolution product of $\cc$ with the dual Alexandrov operator $\uparrow\!\! \cdot$; moreover, every kit point is sup-generated by a unique antichain of compact points, finite if $\cc$ is finitary. 
\end{abstract}

\maketitle

\section{Introduction}

The Hofmann--Lawson theorem \cite[Proposition~2.7]{HofmannLawson76} states that, in every continuous lattice $L$, the set $G := { \Irr L }$ of irreducible elements forms an \textit{inf-generating} subset, in the sense that 
\begin{equation}\label{eq:repr}
{ x } = { \inf { \uparrow\!\! x } \cap { G } },
\end{equation}
for all $x \in L$. 
We also say that we have a \textit{representation} of elements of $L$ by $G$. 
The first representation theorem of this kind is actually due to Birkhoff and Frink \cite[Theorem~7]{BirkhoffFrink48}, who proved that algebraic lattices are inf-generated by the subset of their completely irreducible elements. 
In Equation~\eqref{eq:repr}, the elements of the inf-generating subset $G$ act as the ``building blocks'' of $L$, from which all elements of $L$ can be reconstructed. 

These representation theorems and their extensions have been mostly proved in the case where nonempty finite infs exist, with the notable exception of Martinez \cite[Theorem~1.2]{Martinez72}. 
However, the ambient poset may fail to be a semilattice, while a representation theorem may still occur. 
An emblematic case is that of the set $\mathrsfs{H}_n$ of $n \times n$ Hermitian matrices, and its cone $\mathrsfs{H}_n^+$ of Hermitian positive semidefinite matrices, both partially ordered by the \textit{dual Loewner order} defined by $A \leqslant B$ if $A - B$ is positive semidefinite. 
By Kadison \cite[Theorem~6]{Kadison51}, the posets $(\mathrsfs{H}_n, \leqslant)$ and $(\mathrsfs{H}_n^+, \leqslant)$ are \textit{antilattices}, in the sense that two Hermitian matrices $A$ and $B$ have an inf (or a sup) with respect to the dual Loewner order if and only if these two matrices are comparable, i.e.\ $A \leqslant B$ or $B \leqslant A$. 
Despite this hurdle, we do have a representation theorem in $\mathrsfs{H}_n^+$, for we have 
\[
{ A } = { \inf { \uparrow\!\! A } \cap { e(\mathrsfs{H}_n^+) } }, 
\]
for all $A \in \mathrsfs{H}_n^+$, where $e(\mathrsfs{H}_n^+) := \{ u u^* : u \in \mathbb{C}^n \}$ is the set of extremal elements of the cone $\mathrsfs{H}_n^+$. 
This compelling example has been a source of motivation for this paper, whose main goal is to generalize these representation theorems in several directions. 
 
A first direction of generalization is to do without the semilattice property. 
In fact, a close look at the proof of \cite[Corollary~I-3.10]{Gierz03} shows that the Hofmann--Lawson representation theorem mentioned above still holds in directed-complete continuous \textit{posets}. 
Moreover, one can replace irreducible elements by elements that are maximal in the complement of a filter. 
We call such elements \textit{relatively-maximal}. 
While every relatively-maximal element is irreducible, the converse does not hold in general, although the two notions coincide in a semilattice. 

Our first two representation theorems can then be formulated as follows. 
We say that a poset equipped with a topology \textit{has small open (principal) filters} if, whenever $x \in U$ for some open upper subset $U$, there exists an open (principal) filter $V$ such that $x \in { V } \subseteq { U }$. 
A subset $A$ is said to be \textit{strongly chain-complete} if every nonempty chain of $A$ has a sup that belongs to $A$ to which it converges. 

\begin{theorem*}[Representation Theorems I and II]
Let $P$ be a poset equipped with a lower semiclosed topology. 
Assume that $P$ has small open filters (resp.\ small open principal filters), and let $S$ be a strongly chain-complete subset of $P$. 
Then $S$ is inf-generated by the set of its relatively-maximal elements (resp.\ its completely relatively-maximal elements). 
\end{theorem*}

Despite their level of abstraction, these results do not encompass Martinez' representation theorem \cite[Theorem~1.2]{Martinez72}, which echoes the classical representation (called \textit{factorization}) of non-unit elements by irreducible elements in rings that are principal ideal domains. 
Led by the idea that a representation of the form of Equation~\eqref{eq:repr} can be interpreted as a Krein--Milman theorem after Poncet \cite{Poncet12c}, we investigate a second direction of generalization, where we replace irreducible elements by \textit{extreme points} with respect to a specific closure operator. 
Closure operators have been used as a central tool in various axiomizations of classical convexity theory and set-theoretic generalizations of the Krein--Milman theorem, see notably Fan \cite{Fan63}, Jamison \cite{Jamison74}, Wieczorek \cite{Wieczorek89}, van de Vel \cite{vanDeVel93}, Bachir \cite{Bachir18}. 
Recall that a \textit{closure operator} on a set $E$ is a map $\cc : 2^E \to 2^E$ such that $A \subseteq { \cc(A) } \subseteq { \cc(B) }$ and $\cc(\cc(A)) = \cc(A)$, for all subsets $A$, $B$ with $A \subseteq B$. 
Then, a \textit{$\cc$-extreme point} of $E$ is an element $x \in E$ such that $x \notin { \cc(E \setminus \{ x \}) }$, or equivalently ${ x \in \cc(A) } \Rightarrow { x \in A }$, for all subsets $A$ of $E$. 
We also define a \textit{$\cc$-compact point} of a poset $(E, \leqslant)$ as an element $x \in E$ such that $x \notin { \cc(E \setminus { \uparrow\!\! x }) }$

Interestingly, relatively-maximal elements in a poset $P$ can also be seen as the extreme points of $P$ with respect to an appropriate closure operator, and this closure operator can be expressed as the \textit{convolution product} $\cc^{\downarrow}$ of the closure operator $A \mapsto { \downarrow\!\! A }$ with the closure operator $\cc$ generated by the collection of filters of $P$. 
For this reason, we shall define the convolution product $\cc * \sss$ between two closure operators $\cc$ and $\sss$ in its full generality and develop some of its properties along the road. 

Now, given a closure operator $\cc$ and a collection of subsets $\mathrsfs{K}$, we say that $\mathrsfs{K}$ has the Krein--Milman property if 
\[
\cc(K) = \cc(\textstyle{\ex_{\cc}} K),
\]
for all $K \in \mathrsfs{K}$, where $\ex_{\cc} K$ denotes the set of $\cc$-extreme points of $K$. 
As a key intermediate result, we show how this property can be transferred from a collection $\mathrsfs{U}$ with respect to $\cc$, to a collection $\mathrsfs{V}$ with respect to $\cc * \sss$, where $\sss$ is the closure operator generated by $\mathrsfs{V}$. 
The application of this transfer theorem becomes especially interesting in the specific case of the convolution products $\cc^{\downarrow}$ and $\cc^{\uparrow}$, and is the main ingredient of our third representation theorem. 
This result is concerned with the set of \textit{kit points} $\kit_{\cc} E$ of $E$, a kit point requiring that every copoint admits a compact attaching point. 

\begin{theorem*}[Representation Theorem III]
Let $(E, \leqslant)$ be a poset and $\cc$ be a preinductive closure operator on $E$ such that ${ \downarrow\!\! \cc(A) } = { \cc(A) }$ for all subsets $A$. 
Then $K := \kit_{\cc} E$ is $\cc^{\uparrow}$-closed and has the Krein--Milman property with respect to $\cc^{\uparrow}$, i.e.\ 
\[
{ K } = { \cc^{\uparrow}(\textstyle{\ex_{\cc^{\uparrow}}} K) }. 
\]
Moreover, for every $x \in { \kit_{\cc} E }$, there exists a unique antichain $Y_x$ made of $\cc$-compact points such that 
$
x \in { { \cc(Y_x) } \cap { Y_x^{\uparrow} } }
$; if $\cc$ is finitary, this antichain is finite.
In addition, if $\cc$ separates points, then $\kit_{\cc} E$ is sup-generated by the subset of $\cc$-compact points, and $x = { \bigvee Y_x }$, for all $x \in { \kit_{\cc} E }$. 
\end{theorem*}



The paper is organized as follows. 
In Section~\ref{sec:posets}, we revisit known concepts from poset theory, extended to quasi-ordered sets (qosets). 
We focus on the notion of \textit{preinductivity}, a property of subsets of a qoset that will play a central role in later representation theorems. 
We also provide several conditions - both classical and novel - under which a subset of a qoset has this property. 
In Section~\ref{sec:closures}, we begin by recalling the notions of preclosure and closure operators and elementary properties. 
We show how to associate a \textit{finitary} preclosure operator $\cc^{\circ}$ and an \textit{idempotent} preclosure operator (or closure operator) $\overline{\cc}$ to every preclosure operator $\cc$. 
In Section~\ref{sec:convolution}, we define the convolution product $\cc * \sss$ between two preclosure operators $\cc$ and $\sss$. 
We show that $*$ is a commutative, associative binary relation on the set of preclosure operators on a set $E$, and we prove various additional properties, including some results of categorical flavor. 
In Section~\ref{sec:specialcase}, we deepen our analysis of the convolution product by introducing \textit{enriched qosets} - qosets equipped with a compatible preclosure operator. 
We examine the significant case of the convolution $\cc^{\uparrow}$ of a preclosure operator $\cc$ with the dual Alexandrov operator $\uparrow\!\! \cdot$ induced by the qoset, and provide explicit formulas for this convolution. 
In the case where $\cc$ \textit{separates points}, sups of subsets are involved in these formulas, which lays the ground to our representation theorems (in their sup-generating formulations). 
In Section~\ref{sec:copoints}, we introduce the concepts of \textit{preinductive} preclosure operator and \textit{copoint} of an element. 
Given an enriched qoset $(E, \leqslant, \cc)$, we also define the notions of \textit{$\cc$-compact point} and \textit{$\cc$-extreme point}, and give various characterizations. 
In particular, we characterize $\cc$-extreme points when $\cc$ is the convolution of a preclosure operator with a preinductive closure operator (resp.\ with $\uparrow\!\! \cdot$). 
In Section~\ref{sec:extr}, we examine some notions of extremality for elements of a qoset, and highlight links between them and their underlying (pre)closure operators. 
We emphasize specifically one notion of extremality, namely that of \textit{relatively-maximal element}. 
We show that every relatively-maximal element is irreducible, while the converse does not hold in general. 
In Section~\ref{sec:km}, we consider the \textit{Krein--Milman property}, formulated in terms of extreme points in enriched qosets. 
We show how this property can be transferred from a collection of subsets $\mathrsfs{Q}$ to another, by convolution. 
We apply this result to a variety of examples, where $\mathrsfs{Q}$ is typically the collection of preinductive subsets of a qoset. 
The section culminates in two new representation theorems (Theorems \ref{thm:rep1} and \ref{thm:rep2}) that generalize the Hofmann--Lawson theorem and the Birkhoff--Frink theorem. 
In Section~\ref{sec:local}, we delve deeper into the properties of preclosure operators to establish our third and last representation theorem (Theorem~\ref{thm:rep3}). 
For this purpose, we introduce the notion of \textit{kit point}. 
Given an enriched qoset $(E, \leqslant, \cc)$ with adequate properties, we show that the set of kit points is closed and has the Krein--Milman property with respect to $\cc^{\uparrow}$, and that every kit point is sup-generated by an antichain of compact points, unique up to order-equivalence, and finite if $\cc$ is finitary. 
In passing, we explore the interplay between kitness, compactness, continuity, distributivity of $\cc$, and the Carath\'eodory number of $\cc$.

\section{Quasi-ordered sets and preinductivity}\label{sec:posets}

This section recalls concepts from poset theory, extended to qosets (quasi-ordered sets). 
We focus on the notion of \textit{preinductivity}, a property of subsets of a qoset that will play a central role in later representation theorems. 
We also provide several conditions - both classical and novel - under which a subset of a qoset has this property. 

\subsection{Main definitions and notations}

A \textit{quasi-ordered set} or \textit{qoset} $(P,\leqslant)$ is a set $P$ together with a reflexive and transitive binary relation $\leqslant$, called a \textit{quasiorder}. 
If in addition $\leqslant$ is antisymmetric, then $(P, \leqslant)$ is a \textit{partially ordered set} or \textit{poset}, and $\leqslant$ is a \textit{partial order}. 
We write $<$ for the binary relation defined by $x < y$ if $x \leqslant y$ and $y \not\leqslant x$, for all $x, y \in P$, and $\sim$ for the binary relation defined by $x \sim y$ if $x \leqslant y$ and $y \leqslant x$, for all $x, y \in P$. 
If $x \sim y$, we say that $x$ and $y$ are \textit{order-equivalent}. 
We also denote by $[x]$ the $\sim$-equivalence class of $x \in P$, i.e.\ ${ y \in [x] } \Leftrightarrow { y \sim x }$, and by $[A]$ the subset $\bigcup_{a \in A} [a]$, for all $A \subseteq P$. 
A subset $A$ is then called \textit{order-saturated} if $[A] = A$. 

Let $P$ be a qoset and $A \subseteq P$. 
We denote by $\downarrow\!\! A$ the \textit{lower} subset generated by $A$, i.e.\ $\downarrow\!\! A := \{ x \in P : \exists a \in A, x \leqslant a\}$, and we write $\downarrow\!\! x$ for the \textit{principal order-ideal} $\downarrow\!\!\{x\}$. \textit{Upper} subsets $\uparrow\!\! A$ and \textit{principal filters} $\uparrow\!\! x$ are defined dually. 

An \textit{upper bound} of $A$ is an element $x \in P$ such that $a \leqslant x$ for all $a \in A$. 
\textit{Lower bounds} are defined dually. 
We write $A^{\uparrow}$ (resp.\ $A^{\downarrow}$) for the set of upper bounds (resp.\ lower bounds) of $A$, and we say that $A$ is \textit{bounded above} if $A^{\uparrow} \neq \emptyset$ (resp.\ \textit{bounded below} if $A^{\downarrow} \neq \emptyset$). 
Given subsets $A$ and $B$, we denote by $A \leqslant B$ the assertion $a \leqslant b$ for all $a \in A$, $b \in B$. 
We write $x \leqslant A$ as a shorthand for $\{ x \} \leqslant A$, so that ${ x \leqslant A } \Leftrightarrow { x \in { A^{\downarrow} } } \Leftrightarrow { A \subseteq { \uparrow\!\! x } }$. 
Similarly, ${ A \leqslant x } \Leftrightarrow { x \in { A^{\uparrow} } } \Leftrightarrow { A \subseteq { \downarrow\!\! x } }$.

A \textit{supremum} (or \textit{sup}) of $A$ is any element of $A^{\vee} := { (A^{\uparrow})^{\downarrow} \cap A^{\uparrow} }$. 
Dually, we define an \textit{infimum} (or \textit{inf}) of $A$ and the subset $A^{\wedge}$. 
Any two sups (resp.\ infs) of $A$ are order-equivalent, and if the quasiorder is a partial order, then a sup (resp.\ an inf), if it exists, is necessarily unique, and then denoted by $\bigvee A$ (resp.\ $\bigwedge A$). 
A \textit{semilattice} is a poset in which every nonempty finite subset has an inf. 
 
\begin{example}\label{ex:ring1}
Let $R$ be a commutative ring with a multiplicative identity $1$. 
We assume that $R$ is an \textit{integral domain}, in the sense that $x y = 0$ implies $0 \in \{ x, y \}$, for all $x, y \in R$. 
Then $(R, \leqslant)$ is a qoset, where $\leqslant$ is the \textit{dual divisibility relation} defined by $x \leqslant z$ if $x = y z$, for some $y \in R$. 
In this case, we say that $z$ \textit{divides} $x$. 
The zero $0$ of $R$ is the (unique) inf of $R$ with respect to $\leqslant$. 
Recall that a \textit{unit} in $R$ is an element $u \in R$ such that $u v = 1$, for some $v \in R$. 
Then every unit $u$ is a sup of $R$ with respect to $\leqslant$, since $u$ divides $x$, for all $x \in R$. 
Moreover, $x \sim y$ if and only if $x$ and $y$ are \textit{associate}, in the sense that $x = u y$, for some unit $u$. 
In particular, $\leqslant$ is a partial order if and only if $1$ is the unique unit in $R$. 
\end{example}

A subset $D$ of a qoset $P$ is \textit{directed} if it is nonempty and, for all $x, y \in D$, one can find $z \in D$ such that $x \leqslant z$ and $y \leqslant z$. 
The qoset $P$ is \textit{directed-complete} if every directed subset has a sup in $P$. 
An \textit{order-ideal} is a directed lower subset. 
A subset is \textit{Scott-open} if it is an upper subset such that ${ D \cap U } \neq \emptyset$ whenever $D$ is a directed subset of $P$ such that ${ D^{\vee} \cap U } \neq \emptyset$. 
The collection of Scott-open subsets is a topology called the \textit{Scott topology}. 
A subset $T$ is \textit{filtered} if it is nonempty and, for all $x, y \in T$, one can find $z \in T$ such that $x \geqslant z$ and $y \geqslant z$. 
A \textit{filter} is a filtered upper subset. 

Let $P$ be a poset. 
We say that $x \in P$ is \textit{way-below} $y \in P$, written $x \ll y$, if, for every directed subset $D$ with a sup $\bigvee D$, $y \leqslant \bigvee D$ implies $x \leqslant d$ for some $d \in D$. 
We then say that the poset $P$ is \textit{continuous} if $\twoheaddownarrow x := \{ y \in P : y \ll x \}$ is directed and $x = \bigvee \twoheaddownarrow x$, for all $x \in P$. 

\begin{lemma}\label{lem:openfilter}
Let $P$ be a continuous directed-complete poset, and let $x, y \in P$. 
If $x \not\leqslant y$, then there exists a Scott-open filter $V$ such that $x \in V$ and $y \notin V$. 
\end{lemma}

\begin{proof}
See e.g.\ \cite[Proposition~I-3.3(ii)]{Gierz03}. 
The proof crucially uses the fact that, in a continuous poset, $u \ll w$ implies $u \ll v \ll w$ for some $v$ (a property called the \textit{interpolation property}). 
\end{proof}

For more background on continuous posets, see the monograph by Gierz et al.\ \cite{Gierz03}.  

\subsection{Maximal elements and preinductive subsets}\label{sub:max}

A \textit{chain} in a qoset is a subset $C$ such that $x \leqslant y$ or $y \leqslant x$ for all $x, y \in C$. 
A subset $A$ of a qoset $P$ is called \textit{chain-complete} if every nonempty chain $C$ of $A$ has a sup in $P$ that belongs to $A$, i.e.\ ${ C^{\vee} \cap A } \neq \emptyset$.  \footnote{Be aware that, in the literature, the empty set is sometimes included in the notion of chain-complete poset.}

Recall that, for a qoset, chain-completeness and directed-completeness are actually equivalent properties. 
The proof, which relies on the axiom of choice, is not trivial, see Cohn \cite[Proposition~5.9]{Cohn81} or Markowsky \cite[Corollary~2]{Markowsky76}. 

Let $P$ be a qoset and $A$ be a subset of $P$. 
An element $x$ is \textit{maximal in $A$} if $x \in A$ and there is no element $y \in A$ with $y > x$, equivalently if $x \in A$ and $x \leqslant y \Rightarrow x \sim y$, for all $y \in A$.  
The set of maximal elements of $A$ is denoted by $\Max A$. 
Dually, we define \textit{minimal} elements of $A$, whose set is denoted by $\Min A$. 
The subset $A$ is \textit{preinductive} if every element of $A$ has some maximal element of $A$ above it, and \textit{inductive} if every nonempty chain in $A$ has an upper bound in $A$. 
The subset $A$ satisfies the \textit{ascending chain condition} if, for every sequence $(a_n)_{n \in \mathbb{N}}$ of elements of $A$ such that $a_n \leqslant a_{n+1}$ for all $n \in \mathbb{N}$, there is some $n_0 \in \mathbb{N}$ with $a_{n_0} \sim a_{n}$ for all $n \geqslant n_0$. 

\begin{lemma}\label{lem:preinductive}
Let $P$ be a qoset and $A \subseteq P$. 
Consider the following assertions: 
\begin{enumerate}

  \item\label{lem:preinductive4} $A$ is finite;
  \item\label{lem:preinductive5} ${ \uparrow\!\! x } \cap { A }$ is finite, for all $x \in A$;
  \item\label{lem:preinductive0} $A$ satisfies the ascending chain condition;
  \item\label{lem:preinductive1} $A$ is chain-complete;
  \item\label{lem:preinductive2} $A$ is inductive;
  \item\label{lem:preinductive3} $A$ is preinductive.
\end{enumerate}
Then \eqref{lem:preinductive4} $\Rightarrow$ \eqref{lem:preinductive5} $\Rightarrow$ \eqref{lem:preinductive0} $\Rightarrow$ \eqref{lem:preinductive1} $\Rightarrow$ \eqref{lem:preinductive2} $\Rightarrow$ \eqref{lem:preinductive3}. 
\end{lemma}

\begin{proof}
\eqref{lem:preinductive4} $\Rightarrow$ \eqref{lem:preinductive5} $\Rightarrow$ \eqref{lem:preinductive0} is straightforward. 

\eqref{lem:preinductive0} $\Rightarrow$ \eqref{lem:preinductive1}. 
Let $C$ be a nonempty chain in $A$. 
Suppose that $C$ has no greatest element, i.e.\ for every $c \in C$ there exists $c' \in C$ such that $c' \not\leqslant c$. 
One can construct a sequence $(c_n)_{n \in \mathbb{N}}$ of elements of $A$ such that $c_n < c_{n+1}$ for all $n \in \mathbb{N}$, which contradicts the ascending chain condition. 
So $C$ has a greatest element $c_0$, and $c_0 \in { C^{\vee} \cap A }$. 
This proves that $A$ is chain-complete. 

%

\eqref{lem:preinductive1} $\Rightarrow$ \eqref{lem:preinductive2} is obvious. 

\eqref{lem:preinductive2} $\Rightarrow$ \eqref{lem:preinductive3} is Zorn's lemma applied to the qoset ${ \uparrow\!\! x } \cap A$, $x \in A$. 
\end{proof}

\begin{remark}\label{rk:noetherian}
A qoset satisfying the ascending chain condition is called a \textit{Noetherian qoset}. 
The proof of the previous lemma shows that a qoset is Noetherian if and only if every nonempty chain has a greatest element. 
It is even true that a qoset is Noetherian if and only if every order-ideal is principal. 
This implies that every Noetherian poset is a continuous directed-complete poset, in which the way-below relation coincides with the partial order. 
See e.g.\ Lawson and Xu \cite[Proposition~11]{LawsonXu24} for a series of equivalent conditions and references to original proofs. 
\end{remark}

\begin{example}[Example~\ref{ex:ring1} continued]\label{ex:ring2}
Let $R$ be an integral domain with a multiplicative identity $1$. 
Recall that a \textit{ring-ideal} in $R$ is a nonempty subset $I$ of $R$ such that $i + i' \in I$ and $x i \in I$, for all $i, i' \in I$ and $x \in R$. 
The ring-ideal $I$ is \textit{principal} if it is generated by one element, i.e.\ if it is of the form $I = (x_0) := \{ x x_0 : x \in R \}$, for some $x_0 \in R$. 
Note that principal ring-ideals coincide with principal order-ideals in $(R, \leqslant)$, where $\leqslant$ is the dual divisibility relation. 
Also, $R$ is called a \textit{principal ideal domain} or \textit{PID} if every ring-ideal of $R$ is principal.  
Now, if $R$ is a PID, then $(R, \leqslant)$ is Noetherian, in the sense of the previous remark. 
Indeed, in this case, ring-ideals in $R$ coincide with order-ideals in $(R, \leqslant)$. 
Let us prove this assertion. 

So let $I$ be an order-ideal in $(R, \leqslant)$, 
and let us show that $I$ is a ring-ideal. 
Let $i, i' \in I$. 
There is some $i'' \in I$ with $i'' \geqslant i$ and $i'' \geqslant i'$, so $i = x i''$ and $i' = x' i''$, for some $x, x' \in R$. 
Thus, $i + i' = x i'' + x' i'' = (x + x') i''$. 
This shows that $i'' \geqslant i + i'$. 
Since $I$ is a lower subset in $(R, \leqslant)$ and $i'' \in I$, we obtain $i + i' \in I$. 
Now, let $i \in I$ and $x \in R$. 
Then $x i \leqslant i$, so $x i \in I$. 
We have proved that $I$ is a ring-ideal. 

Conversely, let $I$ be a ring-ideal in $R$, and let us show that $I$ is an order-ideal in $(R, \leqslant)$. 
So let $x \in R$ and $i \in I$ with $x \leqslant i$. 
Then $x = y i$, for some $y \in R$. 
So $x \in I$. 
This proves that $I$ is a lower subset. 
Now, let $i, i' \in I$. 
Since $R$ is a PID, the ring-ideal $(i) + (i')$ is principal, i.e.\ of the form $(j)$. 
Thus, $j = x i + x' i'$, for some $x, x' \in R$, so that $j \in I$. 
Moreover, $i, i' \in { (i) + (i') } = { (j) }$, so $j \geqslant i$ and $j \geqslant i'$. 
This proves that $I$ is directed, hence is an order-ideal in $(R, \leqslant)$. 
\end{example}

\begin{lemma}
Let $P$ be a qoset and $A \subseteq P$ be preinductive. 
If $C$ is a nonempty chain of $A$, then the collection of chains $\{ { \downarrow\!\! a } \cap { C } \}_{a \in A_C}$, ordered by inclusion, is a chain, and we have
\begin{equation}\label{eq:union}
{ C } = { \bigcup_{a \in A_C} { \downarrow\!\! a } \cap { C } },
\end{equation}
where the index $a$ runs over ${ A_C } := { \uparrow\!\! C } \cap { \Max A }$. 
\end{lemma}

\begin{proof}
Take $C_a := { { \downarrow\!\! a } \cap C }$, for all $a \in A_C$. 
Now consider some $a, b \in A_C$ such that $C_a$ is not included in $C_b$. 
There is some $c \in { C_a \setminus C_b }$. 
Thus, $c \leqslant a$ and $c \not\leqslant b$. 
If $z \in C_b$, then $z \leqslant c$ or $c \leqslant z$, since $C$ is a chain. 
The latter case would yield $c \leqslant b$ (contradiction), so $z \leqslant c$, hence $z \leqslant a$. 
This proves that $z \in C_a$. 
So $C_b \subseteq C_a$. 
We have just shown that $\{ C_a \}_{a \in A_C}$, ordered by inclusion, is a chain. 

Now, let us show that the union of $\{ C_a \}_{a \in A_C}$ equals $C$. 
So let $x \in C$. 
There exists some $a \in { { \uparrow\!\! x } \cap { \Max A} }$, since $A$ is preinductive. 
We obtain $x \in C_a$ with $a \in A_C$. 
So Equation~\eqref{eq:union} is proved. 
\end{proof}

We call a subset $A$ of a qoset \textit{local} if ${ \uparrow\!\! x } \cap { \Max A }$, the subset of maximal elements of $A$ above $x$, is nonempty finite for all $x \in A$. 
Obviously, a local subset is preinductive, but there is more, as the following result asserts. 

\begin{lemma}\label{lem:local0}
Let $P$ be a qoset and $A \subseteq P$. 
If $A$ is local, then $A$ is inductive. 
\end{lemma}

\begin{proof}
Let $K$ be a nonempty chain in $A$, and let $x \in K$. 
Take $C := { \uparrow\!\! x \cap K}$. 
Since $A$ is preinductive, the previous lemma applies. 
Yet, note that the index set ${ A_C } := { \uparrow\!\! C } \cap { \Max A } = { \uparrow\!\! x } \cap { \Max A }$ in Equation~\eqref{eq:union} is finite, since $A$ is local. 
So there is some $a \in A_C$ such that $C = { { \downarrow\!\! a } \cap C }$. 
This proves that $C$ and $K$ are bounded above by $a \in A$. 
So $A$ is inductive. 
\end{proof}


A topology on a qoset is \textit{upper semiclosed} if each principal filter is a closed subset.  

\begin{lemma}[Wallace]\label{lem:wallace}
Let $P$ be a qoset equipped with an upper semiclosed topology, and let $A \subseteq P$. 
If $A$ is compact, then $A$ is inductive. 
\end{lemma}

\begin{proof}
See Wallace \cite[Paragraph~2]{Wallace45}; see also Gierz et al.\ \cite[Proposition~VI-5.3]{Gierz03}. 
We show that $A$ is inductive. 
So let $C$ be a nonempty chain in $A$. 
If $C$ has no upper bound in $A$, then ${ A } \subseteq { \bigcup_{x \in C} P \setminus { \uparrow\!\! x } }$. 
By hypothesis, the topology on $P$ is upper semiclosed, so each subset $P \setminus { \uparrow\!\! x }$ is open. 
Since $A$ is compact, we deduce that ${ A } \subseteq { P \setminus \uparrow\!\! x_0 }$, for some $x_0 \in C$, which contradicts $x_0 \in A$. 
\end{proof}

In a topological space, the \textit{specialization order} is the quasiorder $\leqslant$ defined by $x \leqslant y$ if $x \in G \Rightarrow y \in G$, for all open subsets $G$. 
A \textit{monotone convergence space} (or \textit{$d$-space}) is a topological space that, when equipped with its specialization order, is directed-complete and such that every open subset is Scott-open. 
For instance, every continuous directed-complete poset with its Scott topology is a monotone convergence space. 
On this topic, see e.g.\ Ershov \cite{Ershov99} and Keimel and Lawson \cite{Keimel09b}. 

\begin{lemma}\label{lem:mcs}
Let $P$ be a monotone convergence space and $A \subseteq P$. 
If $A$ is closed, then $A$ is chain-complete, hence inductive. 
\end{lemma}

\begin{proof}
Let $C$ be a nonempty chain in $A$. 
Then $C$ has a sup $m$ in $P$. 
Moreover, $C$ converges to $m$. 
Indeed, let $U$ be an open subset containing $m$; since $U$ is Scott-open, there exists some $y \in C \cap U$, and every element of the net $\uparrow\!\! y \cap C$ is in $U$. 
Since $A$ is closed and $C \subseteq A$ converges to $m$, we deduce that $m \in A$. 
\end{proof}

\begin{remark}
If the ambient qoset $P$ is chain-complete, other results of the same flavor may be formulated. 
For instance, if $A$ is an upper subset, or if it is the image of a kernel map, then $A$ is itself chain-complete, hence inductive; see Poncet \cite{Poncet22}. 
\end{remark}

Preinductive subsets are interesting in that the property of preinductivity can be interpreted as a Krein--Milman property, as we shall see in Section~\ref{sec:km} with Example~\ref{ex:maximal}. 

\subsection{Strongly chain-complete subsets}

We mention here another property that implies inductivity, which will reveal its usefulness in Section~\ref{sec:km}. 
A subset $A$ of a qoset $P$ equipped with a topology is said to be \textit{strongly chain-complete} if every nonempty chain of $A$ has a sup in $P$ that belongs to $A$ to which it converges. 
Obviously, every subset satisfying the ascending chain condition is strongly chain-complete, and every strongly chain-complete subset is chain-complete, hence inductive. 

\begin{example}
The following subsets are strongly chain-complete (we do not provide a definition for all the notions involved; the reader is referred to the related publications and monographs):
\begin{itemize}
  \item closed subsets of a compact pospace (see Nachbin \cite{Nachbin65});
  \item principal order-ideals of a pospace with compact intervals (this is easily deduced from the previous case); 
  \item closed subsets of a monotone convergence space (see the proof of Lemma~\ref{lem:mcs});
  \item sober spaces equipped with their specialization order \cite[Proposition~O-5.15(vii)]{Gierz03}; more generally, saturated subsets of a sober space (use the fact that these subsets are themselves sober);  
  \item continuous (or more generally, quasi-continuous) directed-complete posets equipped with their Scott topology (use the fact that they are sober \cite[Exercise~8.2.15]{Goubault13});
  \item compact subsets of a Hausdorff topological semilattice (use Propositions VI-1.3 and VI-1.14 in \cite{Gierz03}). 
  \item In a Hausdorff topological semilattice, let $A$ be a subset containing no upward rays, in the sense that every chain in $A$ with a least element is relatively compact; then $A$ is strongly chain-complete. 
\end{itemize}
\end{example}



\section{Preclosure operators}\label{sec:closures}

In this section, we begin by recalling the notion of preclosure operator and elementary properties. 
We show how to associate a finitary preclosure operator $\cc^{\circ}$ and an idempotent preclosure operator (or closure operator) $\overline{\cc}$ to every preclosure operator $\cc$. 

\subsection{Definition and first properties}

A \textit{preclosure operator} on a set $E$ is a map $\cc : 2^E \to 2^E$ such that $A \subseteq { \cc(A) } \subseteq { \cc(B) }$, for all subsets $A$ and $B$ such that $A \subseteq B$. 
The pair $(E, \cc)$ is then called a \textit{preclosure space}. 
A subset $A$ of $E$ is \textit{$\cc$-closed} if $\cc(A) = A$, \textit{$\cc$-open} if its complement is $\cc$-closed, and \textit{relatively $\cc$-closed} if $\cc(\cc(A)) = \cc(A)$. 
For a singleton $\{ x \}$, we usually write $\cc(x)$ instead of $\cc(\{ x \})$. 

Note that the empty set $\emptyset$ is not assumed to be closed in general; we say that $\cc$ is \textit{untied} if this is the case, i.e.\ if $\cc(\emptyset) = \emptyset$. 

\begin{proposition}\label{prop:moore}
The collection of closed subsets of a preclosure operator is stable under arbitrary intersections. 
\end{proposition}

\begin{proof}
Let $\cc$ be a preclosure operator, let $\{ C_j \}_{j\in J}$ be a collection of $\cc$-closed subsets, and take $C := \bigcap_{j\in J} C_j$. 
For all $j \in J$, we have ${ \cc(C) } \subseteq { \cc(C_j) } = { C_j }$, so that $C \subseteq { \cc(C) } \subseteq { \bigcap_{j\in J} C_j } = { C }$, which proves that $C$ is closed. 
\end{proof}

A \textit{closure operator} $\cc$ on a set $E$ is a preclosure operator on $E$ that is also \textit{idempotent}, in the sense that $\cc(\cc(A)) = \cc(A)$ for all subsets $A$ of $E$. 
The pair $(E, \cc)$ is then called a \textit{closure space}. 

\begin{example}\label{ex:clogene}
Let $\mathrsfs{V}$ be a collection of subsets of a set $E$. 
Then the map $\langle \cdot \rangle_{\mathrsfs{V}}$ defined by 
\[
\langle A \rangle_{\mathrsfs{V}} = \bigcap_{V \supseteq A} V, 
\]
for all subsets $A$ of $E$, where the intersection runs over the subsets $V$ in $\mathrsfs{V}$ containing $A$, is a closure operator, called the \textit{closure operator generated by $\mathrsfs{V}$}. 
The elements of $\mathrsfs{V}$ are closed with respect to $\langle \cdot \rangle_{\mathrsfs{V}}$, yet some other closed subsets exist unless $\mathrsfs{V}$ is stable under arbitrary intersections. 
\end{example}

\begin{example}\label{ex:alex}
On a qoset, some natural closure operators arise, such as the \textit{Alexandrov closure operator} 
\[
A \mapsto { \downarrow\!\! A } := { \{ x : \exists a \in A, x \leqslant a \} }, 
\] 
and the \textit{Dedekind--MacNeille closure operator} $\mathfrak{d}$
\[
A \mapsto { \mathfrak{d}(A) } := { (A^{\uparrow})^{\downarrow} } = { \{ x : \forall u \in A^{\uparrow}, x \leqslant u \} }.
\] 
Note that $\downarrow\!\! A$ coincides with $\langle A \rangle_{\mathrsfs{L}}$ and $\mathfrak{d}(A)$ coincides with $\langle A \rangle_{\mathrsfs{P}}$,  where $\mathrsfs{L}$ and $\mathrsfs{P}$ denote respectively the collection of lower subsets and the collection of principal order-ideals of the qoset. 
Moreover, if $A$ has a sup $s$, then ${ \mathfrak{d}(A) } = { \downarrow\!\! s }$. 
\end{example}

\begin{example}\label{ex:ranzato0}
Given a qoset $P$, Ranzato \cite{Ranzato02} considered the preclosure operator $\mathfrak{p}^{\downarrow}$ defined by
\begin{equation}\label{eq:pdown}
{ \mathfrak{p}^{\downarrow}(A) } = { \bigcup_{B \subseteq A} \Max(B^{\downarrow}) },
\end{equation}
for all subsets $A$ of $P$. 
We shall also consider its dual $\mathfrak{q}^{\uparrow}$ defined by
\begin{equation}\label{eq:pup}
{ \mathfrak{q}^{\uparrow}(A) } = { \bigcup_{B \subseteq A} \Min(B^{\uparrow}) },
\end{equation}
for all subsets $A$ of $P$. 
\end{example}

It is useful to recall the following property met by closure operators. 

\begin{proposition}\label{prop:distr}
Let $(E, \cc)$ be a closure space. 
Then
\[
\cc(\bigcup_{j \in J} A_j) = \cc(\bigcup_{j \in J} \cc(A_j)),
\]
for every collection $\{ A_j \}_{j \in J}$ of subsets of $E$. 
\end{proposition}

\begin{proof}
The inclusion $\subseteq$ is clear. 
For the reverse inclusion, note first that ${ \cc(A_i) } \subseteq { \cc(\bigcup_{j \in J} A_j) }$ for every $i \in J$, hence ${ \bigcup_{j \in J} \cc(A_j) } \subseteq { \cc(\bigcup_{j \in J} A_j) }$. 
This implies ${ \cc(\bigcup_{j \in J} \cc(A_j)) } \subseteq { \cc(\cc(\bigcup_{j \in J} A_j)) } = { \cc(\bigcup_{j \in J} A_j) }$, the last equality using the idempotency of $\cc$. 
\end{proof}

\subsection{Special constructions}

We partially order the set of preclosure operators on a set $E$ by $\cc \leqslant \sss$ if ${ \cc(A) } \subseteq { \sss(A) }$, for all subsets $A$. 


%
%

If $\cc$ is a preclosure operator, and if $\mathrsfs{C}$ denotes the collection of $\cc$-closed subsets, then there exists a least closure operator $\overline{\cc}$ such that $\cc \leqslant \overline{\cc}$, defined by $\overline{\cc} = \langle \cdot \rangle_{\mathrsfs{C}}$. 
It follows from the definition that ${ \cc(A) } \subseteq { \overline{\cc}(A) }$ and 
\[
{ \cc(\overline{\cc}(A)) } = { \overline{\cc}(A) } = { \overline{\cc}(\cc(A)) }, 
\]
for all subsets $A$ of $E$. 
Note also that $\cc$ and $\overline{\cc}$ share the same collection of closed subsets. 
As per the Knaster--Tarski fixpoint theorem and the arguments of Ranzato \cite[Section~3]{Ranzato02}, we can build $\overline{\cc}$ as 
\begin{equation}\label{eq:cbar}
\overline{\cc}(A) = \bigcup_{\alpha \in \mathbb{O}} \cc^{\alpha}(A),
\end{equation}
where $\mathbb{O}$ denotes the class of ordinals and $\cc^{\alpha}$ is defined by $\cc^{\alpha}(A) = A$ if $\alpha = 0$, $\cc^{\alpha}(A) = \cc(\cc^{\alpha - 1}(A))$ if $\alpha$ is a successor ordinal, and $\cc^{\alpha}(A) = \bigcup_{\beta < \alpha} \cc^{\beta}(A)$ if $\alpha$ is a limit ordinal.

A preclosure operator $\cc$ is called \textit{finitary} if, for all subsets $A$, 
\begin{equation}\label{eq:finitary}
{ \cc(A) } = { \bigcup_{F \subseteq A} \cc(F) }, 
\end{equation}
where the union runs over the finite subsets $F$ of $A$. 

\begin{lemma}\label{lem:fin}
Let $(E, \cc)$ be a preclosure space. 
Then $\cc$ is finitary if and only if it commutes with directed unions of subsets, in the sense that 
\[
\cc(\bigcup_{D \in \mathrsfs{D}} D) = \bigcup_{D \in \mathrsfs{D}} \cc(D), 
\]
for every directed collection $\mathrsfs{D}$ of subsets of $E$. 
In this case, the collections of $\cc$-closed subsets and relatively $\cc$-closed subsets are both stable under directed unions. 
\end{lemma}

\begin{proof}
Assume that $\cc$ commutes with directed unions of subsets. 
Using the fact that $A \subseteq E$ is the directed union of its finite subsets, we obtain Equation~\eqref{eq:finitary}. 

Conversely, assume that Equation~\eqref{eq:finitary} holds, and let $\mathrsfs{D}$ be a directed collection of subsets. 
Take $A := { \bigcup \mathrsfs{D} }$, and let $x \in { \cc(A) }$. 
There exists some finite subset $F$ of $A$ with $x \in \cc(F)$. 
If $f \in F$, there is some $D_f \in { \mathrsfs{D} }$ such that $f \in { D_f }$. 
Since $\mathrsfs{D}$ is directed, there is some $D \in { \mathrsfs{D} }$ such that $D_f \subseteq D$ for all $f \in F$. 
Thus, $F \subseteq D$, so $x \in { \cc(D) }$. 
This shows that $x \in { \bigcup \cc(\mathrsfs{D}) }$. 
So $\cc$ commutes with directed unions of subsets.

Now assume that $\cc$ is finitary, and let $\{ A_j \}_{j \in J}$ be a directed collection of relatively $\cc$-closed subsets. 
Take $A := { \bigcup_{j \in J} A_j }$. 
Then 
\begin{align*}
{ \cc(\cc(A)) } &= { \cc(\cc(\bigcup_{j \in J} A_j )) } = { \cc(\bigcup_{j \in J} \cc(A_j)) } \\
&= { \bigcup_{j \in J} \cc(\cc(A_j)) } = { \bigcup_{j \in J} \cc(A_j) } \\
&= { \cc(\bigcup_{j \in J} A_j) } = { \cc(A) }, 
\end{align*}
where we have used that $\cc(\cc(A_j)) = \cc(A_j)$ for all $j \in J$, and the fact that $\{ \cc(A_j) \}_{j \in J}$ is also a directed collection of subsets. 
This proves that $A$ is relatively $\cc$-closed. 
A similar proof shows that, if $\{ A_j \}_{j \in J}$ is a directed collection of $\cc$-closed subsets, then its union is $\cc$-closed. 
\end{proof}

\begin{example}[Example \ref{ex:clogene} continued]\label{ex:clogene0}
Let $\mathrsfs{V}$ be a collection of subsets of a set $E$. 
If $\mathrsfs{V}$ is stable under arbitrary intersections and directed unions, then $\langle \cdot \rangle_{\mathrsfs{V}}$ is finitary. 
Indeed, first note that $E \in \mathrsfs{V}$ can be assumed without loss of generality. 
Then $\mathrsfs{V}$ coincides with the collection of $\langle \cdot \rangle_{\mathrsfs{V}}$-closed subsets. 
Now, if $\mathrsfs{D}$ is a directed collection of subsets, then $\{ \langle D \rangle_{\mathrsfs{V}} \}_{D \in \mathrsfs{D}}$ is a directed collection of closed subsets. 
By assumption, its union is closed and contains $\bigcup \mathrsfs{D}$. 
Thus, 
\[
{ \langle \bigcup \mathrsfs{D} \rangle_{\mathrsfs{V}}  } = { \bigcup_{D \in \mathrsfs{D}} { \langle D \rangle_{\mathrsfs{V}} } }. 
\]
By the previous lemma, this proves that $\langle \cdot \rangle_{\mathrsfs{V}}$ is finitary. 
\end{example}

\begin{example}[Example \ref{ex:clogene0} continued]\label{ex:infconvex1}
Let $(P, \leqslant)$ be a qoset. 
A subset $H$ of $P$ is called \textit{inf-closed} if $F^{\wedge} \subseteq H$, for all finite subsets $F$ of $H$. 
As an example, every filter in $P$ is inf-closed. 
The collection $\mathrsfs{H}$ of inf-closed subsets is stable under arbitrary intersections and directed unions, hence the closure operator $\langle \cdot \rangle_{\mathrsfs{H}}$ is finitary. 
\end{example}

\begin{remark}
Finitary untied closure operators are called \textit{convexity operators}, and are at the core of an axiomatization of classical convexity theory called \textit{abstract convexity}. 
The data of a convexity operator on a set $E$ is equivalent to that of a collection $\mathrsfs{C}$ of subsets of $E$, called a \textit{convexity}, which satisfies the following axioms: 
\begin{itemize}
	\item $\emptyset, E \in \mathrsfs{C}$, 
	\item $\mathrsfs{C}$ is stable under arbitrary intersections, 
	\item $\mathrsfs{C}$ is stable under directed unions. 
\end{itemize}
Then $\langle \cdot \rangle_{\mathrsfs{C}}$ is a convexity operator, the elements of $\mathrsfs{C}$, i.e.\ the $\langle \cdot \rangle_{\mathrsfs{C}}$-closed subsets, are the \textit{convex} subsets, and $\langle A \rangle_{\mathrsfs{C}}$ is the \textit{convex hull} of the subset $A$. 
Abstract convexity theory was widely developed in van de Vel's monograph \cite{vanDeVel93}. 
For abstract convexity concepts applied to posets and semilattices, see Poncet \cite{Poncet11}, \cite{Poncet12c} and references therein. 
\end{remark}

\begin{proposition}\label{prop:ccirc}
Let $(E, \cc)$ be a preclosure space. 
Then the map $\cc^{\circ}$ defined by 
\[
{ \cc^{\circ}(A) } = { \bigcup_{F \subseteq A} \cc(F) },
\]
for all subsets $A$ of $E$, where the union runs over the finite subsets $F$ of $A$, 
is a finitary preclosure operator on $E$ such that $\cc^{\circ} \leqslant \cc$. 
Moreover, 
\begin{enumerate}
  \item\label{prop:ccirc1} if $\cc$ is a closure operator, then $\cc^{\circ}$ is a finitary closure operator, and ${ \cc^{\circ}(\cc(A)) } = { \cc(A) } = { \cc(\cc^{\circ}(A)) }$, for all $A \subseteq E$;
  \item\label{prop:ccirc2} if $\cc$ is an untied closure operator, then $\cc^{\circ}$ is a convexity operator;
  \item\label{prop:ccirc3} if $\cc$ is untied, then $\cc$ is a convexity operator if and only if $\overline{\cc} = \cc = \cc^{\circ}$. 
\end{enumerate}
\end{proposition}

\begin{proof}
The map $\cc^{\circ}$ clearly satisfies $A \subseteq { \cc^{\circ}(A) } \subseteq { \cc^{\circ}(B) }$ for all subsets $A \subseteq B$, hence is a preclosure operator. 
Moreover, $\cc^{\circ}(F) = \cc(F)$ for all finite subsets $F$ of $E$, so $\cc^{\circ}$ is finitary. 
The assertion $\cc^{\circ} \leqslant \cc$ is evident. 

\eqref{prop:ccirc1}. 
Let $x \in \cc^{\circ}(\cc^{\circ}(A))$. 
Then $x \in \cc(F)$ for some finite subset $F$ of $\cc^{\circ}(A)$. 
For every $y \in F$, there exists some finite subset $F_y$ of $A$ such that $y \in \cc(F_y)$. 
Let $F'$ be the finite subset $\bigcup_{y \in F} F_y$. 
Then $F' \subseteq A$, $F \subseteq { \cc(F') }$, and $x \in { \cc(F) } \subseteq { \cc(\cc(F')) } = { \cc(F') }$. 
This shows that $x \in { \cc^{\circ}(A) }$. 
So $\cc^{\circ}$ is indeed a closure operator. 
The remaining part of \eqref{prop:ccirc1}, i.e.\ the fact that $\cc^{\circ} \circ \cc = \cc = \cc \circ \cc^{\circ}$, is an easy consequence of Proposition~\ref{prop:distr}. 

\eqref{prop:ccirc2}. 
Since $\cc$ is untied, $\cc^{\circ}$ is also untied. 
By \eqref{prop:ccirc1}, $\cc^{\circ}$ is thus a convexity operator. 

\eqref{prop:ccirc3}. 
Since $\cc$ is untied, the fact that $\cc$ is a convexity operator if and only if $\overline{\cc} = \cc = \cc^{\circ}$ is clear from the definitions. 
\end{proof}


\begin{example}[Example \ref{ex:clogene0} continued]\label{ex:clogene1}
Let $\mathrsfs{V}$ be a collection of subsets of a set $E$. 
Let $\mathrsfs{V}^{\circ}$ denote the collection of subsets $C$ of $E$ satisfying
\[
{ F \subseteq C } \Rightarrow { { \langle F \rangle_{\mathrsfs{V}} } \subseteq C  },
\]
for all finite subsets $F$. 
Then $\mathrsfs{V}^{\circ}$ is stable under arbitrary intersections and directed unions, hence $ \langle \cdot \rangle_{\mathrsfs{V}^{\circ}}$ is finitary. 
Moreover, it is easily seen that 
\[
{ \langle A \rangle_{\mathrsfs{V}}^{\circ} } = { \langle A \rangle_{\mathrsfs{V}^{\circ}} },
\]
for all subsets $A$. 
\end{example}

With the next result we show that one can replace the class of ordinals by the set of natural numbers $\mathbb{N}$ in Equation~\eqref{eq:cbar} when the preclosure operator at stake is finitary. 
Given a preclosure space $(E, \cc)$, we define $\cc^{n}$ by $\cc^{n}(A) = A$ if $n = 0$, and $\cc^{n}(A) = \cc(\cc^{n-1}(A))$ if $n > 0$, for all subsets $A$ and $n \in \mathbb{N}$. 

\begin{proposition}\label{prop:progag}
Let $(E, \cc)$ be a preclosure space. 
If $\cc$ is finitary (resp.\ finitary and untied), then $\overline{\cc}$ is a finitary closure operator (resp.\ a convexity operator) satisfying
\[
\overline{\cc}(A) = \bigcup_{n \in \mathbb{N}} \cc^{n}(A),
\]
for all subsets $A$ of $E$. 
\end{proposition}

\begin{proof}
Let $\sss$ be the preclosure operator $A \mapsto { \bigcup_{n \in \mathbb{N}} \cc^{n}(A) }$. 
Let $A$ be a subset of $E$. 
We first show that $\cc(\sss(A)) = \sss(A)$. 
So let $x \in \cc(\sss(A))$. 
Since $\cc$ is finitary, there is some finite subset $F$ of $\sss(A)$ with $x \in \cc(F)$. 
Then, for every $y \in F$, there is some $n_y \in \mathbb{N}$ such that $y \in \cc^{n_y}(A)$. 
Take $n := \max_{y \in F} n_y$. 
We deduce that ${ F } \subseteq { \cc^{n}(A) }$. 
Thus, $x \in { \cc(F) } \subseteq { \cc^{n + 1}(A) } \subseteq { \sss(A) }$. 
This proves that $\cc(\sss(A)) = \sss(A)$, hence $\cc^{m}(\sss(A)) = \sss(A)$ for all $m \in \mathbb{N}$, which yields $\sss(\sss(A)) = \sss(A)$. 
So $\sss$ is a closure operator, i.e.\ $\overline{\sss} = \sss$. 
Moreover, $\sss$ is finitary since each of the preclosure operators $\cc^{m}$ with $m \in \mathbb{N}$ is itself finitary, as a straightforward induction shows. 
Now we have $\cc \leqslant \sss \leqslant \overline{\cc} \leqslant \overline{\sss} = \sss$, so that $\sss = \overline{\cc}$. 
This proves that $\overline{\cc}$ is a finitary closure operator satisfying the equation of the proposition. 
If moreover $\cc$ is untied, then $\overline{\cc}(\emptyset) = \sss(\emptyset) = \emptyset$, so $\overline{\cc}$ is untied, hence is a convexity operator. 
\end{proof}


\begin{proposition}
Let $(E, \cc)$ be a preclosure space. 
Then $\overline{\cc^{\circ}}$ and $(\overline{\cc})^{\circ}$ are finitary closure operators such that $\overline{\cc^{\circ}} \leqslant (\overline{\cc})^{\circ}$. 
\end{proposition}

\begin{proof}
By Propositions~\ref{prop:ccirc} and \ref{prop:progag}, $(\overline{\cc})^{\circ}$ and $\overline{\cc^{\circ}}$ are indeed finitary closure operators. 
So $\overline{\cc^{\circ}} \leqslant (\overline{\cc})^{\circ}$ if and only if ${ \overline{\cc^{\circ}}(F) } \subseteq { (\overline{\cc})^{\circ}(F) }$, for all finite subsets $F$, which amounts to ${ \overline{\cc^{\circ}}(F) } \subseteq { \overline{\cc}(F) }$, for all finite subsets $F$. 
This inclusion is clear, so the result is proved. 
\end{proof}

\section{Convolution of preclosure operators}\label{sec:convolution}

In this section, we define the \textit{convolution product} $\cc * \sss$ between two preclosure operators $\cc$ and $\sss$. 
We show that $*$ is a commutative, associative binary relation on the set of preclosure operators on a set $E$, and we prove various additional properties. 
We also prove some results of categorical flavor. 
Notably, we show that the convolution product induces a functor from the category of bi-preclosure spaces and bi-continuous maps to the category of preclosure spaces and continuous maps. 

\subsection{Convolution product}

A (pre)closure space $(E, \cc)$ endowed with an additional (pre)closure operator $\sss$ is called a \textit{bi-(pre)closure space} and denoted by $(E, \cc, \sss)$. 

\begin{definition}
Let $(E, \cc, \sss)$ be a bi-preclosure space. 
We define the \textit{convolution product} $\cc * \sss$ of $\cc$ and $\sss$ by
\begin{equation}\label{eq:convol}
{ \cc * \sss(A) } = { \bigcap_{B \subseteq E} \cc(A \cap B) \cup \sss(A \setminus B) },
\end{equation}
for all subsets $A$ of $E$. 
\end{definition}

We denote by $\mathfrak{C}(E)$ the set of preclosure operators on a set $E$. 

\begin{theorem}\label{thm:basic}
Given a set $E$, the convolution product is a commutative, associative binary relation on $\mathfrak{C}(E)$. 
\end{theorem}

\begin{proof}
Let $\cc$ and $\sss$ be preclosure operators on $E$. 
That $\cc * \sss$ be a preclosure operator, is evident from the definition of the convolution product. 
Commutativity is straightforward. 
Let us show that $*$ is associative. 
So let $\cc$, $\sss$, and $\ttt$ be three preclosure operators, and let $A \subseteq E$. 
To prove that $(\cc * \sss) * \ttt (A) = \cc * (\sss * \ttt)(A)$, it suffices to notice that both terms are equal to the quantity
\[
\bigcap_{B_1, B_2, B_3} \cc(A \cap B_1) \cup \sss(A \cap B_2) \cup \ttt(A \cap B_3),
\]
where the intersection runs over the subsets $B_1$, $B_2$, $B_3$ that form a partition of $E$. 
\end{proof}

\begin{example}[Example~\ref{ex:clogene1} continued]\label{ex:clogene3}
Let $E$ be a set. 
Then 
\[
{ { \langle \cdot \rangle_{\mathrsfs{U}} } * { \langle \cdot \rangle_{\mathrsfs{V}} } } = { \langle \cdot \rangle_{\mathrsfs{U} * \mathrsfs{V}} }, 
\]
for all collections $\mathrsfs{U}$, $\mathrsfs{V}$ of subsets of $E$, 
where $\mathrsfs{U} * \mathrsfs{V}$ denotes the collection made of subsets $U \cup V$ with $U \in \mathrsfs{U}$, $V \in \mathrsfs{V}$. 
\end{example}

A preclosure operator $\cc$ is called \textit{\v{C}ech} if $\cc$ is untied and ${ \cc(A \cup A') } = { \cc(A) \cup \cc(A') }$ for all subsets $A$, $A'$, and \textit{topological} if it is both \v{C}ech and idempotent. 

\begin{proposition}\label{prop:properties}
The following properties hold for the convolution product of preclosure operators $\cc$, $\sss$ on a set $E$:
\begin{enumerate}
  \item\label{thm:basic1} $\cc * \top = \top * \cc = \top$, where $\top : A \mapsto E$;
  \item\label{thm:basic2} $\cc * \langle \cdot \rangle = \langle \cdot \rangle * \cc = \cc$, where $\langle A \rangle = E$ if $A \neq \emptyset$ and $\langle \emptyset \rangle = \emptyset$;
  \item\label{thm:basic3} $\cc * \sss(\emptyset) = \cc(\emptyset) \cup \sss(\emptyset)$;
  \item\label{thm:basic4} $\cc * \sss \leqslant \cc' * \sss'$ whenever $\cc \leqslant \cc'$ and $\sss \leqslant \sss'$;
  \item\label{thm:basic5} $\cc * \sss$ is untied if $\cc$ and $\sss$ are untied;
  \item\label{thm:basic6} $\cc * \sss$ is idempotent if $\cc$ and $\sss$ are idempotent;
  \item\label{thm:basic7} $\cc$ is \v{C}ech if and only if $\cc$ is untied and $\cc * \cc = \cc$;
  \item\label{thm:basic8} $\cc * \sss$ is \v{C}ech if $\cc$ and $\sss$ are \v{C}ech;
  \item\label{thm:basic9} $\cc$ is topological if and only if $\cc$ is untied and $\cc * \cc = \cc \circ \cc$;
  \item\label{thm:basic10} $\cc * \sss$ is topological if $\cc$ and $\sss$ are topological. 
\end{enumerate}
Moreover, taking $B := \cc(\emptyset)$, 
\begin{enumerate}
  \setcounter{enumi}{10}
  \item\label{thm:basic2b} $\cc * { \langle \cdot \rangle^B } = { \langle \cdot \rangle^B } * \cc = \cc$, where ${ \langle A \rangle^B } = E$ if $A \neq \emptyset$ and ${ \langle \emptyset \rangle^B } = B$;
  \item\label{thm:basic11} $\cc * \bot^B = \bot^B * \cc = \bot^B$, where $\bot^ B : A \mapsto { A \cup B }$;
  \item\label{thm:basic12} $\cc * \sss \leqslant \cc$, if $\sss(\emptyset) \subseteq B$;
  \item\label{thm:basic13} every $\cc$-closed subset is $\cc * \sss$-closed, if $\sss(\emptyset) \subseteq B$;
  \item\label{thm:basic14} $\cc * \sss(x) = \cc(x) \cap \sss(x)$, for all $x \in E$, if $\sss(\emptyset) = B$. 
\end{enumerate}
\end{proposition}

\begin{proof}
\eqref{thm:basic1}, \eqref{thm:basic2}, \eqref{thm:basic3}, \eqref{thm:basic4}, \eqref{thm:basic5}, \eqref{thm:basic2b}, and \eqref{thm:basic11} easily follow from the definition of the convolution product. 

\eqref{thm:basic6}. 
Assume that $\cc$ and $\sss$ are idempotent, i.e.\ are closure operators, and let us show that so is $\cc * \sss$. 
So let $A \subseteq E$, and let $x \in \cc * \sss(A')$ with ${ A' } := { \cc * \sss(A) }$. 
Let $B \subseteq E$. 
Take ${ B' } := { \cc(A \cap B) }$, and suppose that $x \notin B'$. 
Since $x \in \cc * \sss(A')$, we have $x \in { \cc(A' \cap B') \cup \sss(A' \setminus B') }$. 
If $x \in { \cc(A' \cap B') }$, then $x \in { \cc(B') } = { B' }$ since $\cc$ is idempotent, which contradicts $x \notin B'$. 
So $x \notin \cc(A' \cap B')$, hence $x \in \sss(A' \setminus B')$. 
Moreover, from the definitions of $A'$ and $B'$ we deduce that ${ A' \setminus B' } \subseteq { \sss(A \setminus B) }$, so that ${ \sss(A' \setminus B') } \subseteq { \sss(\sss(A \setminus B)) } = { \sss(A \setminus B) }$. 
This proves that $x \in \sss(A \setminus B)$. 
Thus, we have proved that $x \in { B' \cup \sss(A \setminus B) }$, i.e.\ $x \in { \cc(A \cap B) \cup \sss(A \setminus B) }$, for all $B \subseteq E$. 
This shows that $x \in \cc * \sss(A)$. 
So $\cc * \sss$ is a closure operator. 

\eqref{thm:basic7}. 
Assume that $\cc$ is \v{C}ech, and let $A \subseteq E$. 
Then $\cc$ is untied.  
Moreover, 
\begin{align*}
{ \cc * \cc(A) } &= { \bigcap_{B \subseteq E} \cc(A \cap B) \cup \cc(A \setminus B) } \\
&= { \bigcap_{B \subseteq E} \cc((A \cap B) \cup (A \setminus B)) } = { \bigcap_{B \subseteq E} \cc(A) } = { \cc(A) }. 
\end{align*}
So $\cc * \cc = \cc$. 
Conversely, assume that $\cc$ is untied and $\cc * \cc = \cc$. 
Let $A$, $A'$ be subsets of $E$. 
Then ${ \cc(A \cup A') } = { \cc * \cc(A \cup A') }$. 
Taking $B = A$ in Equation~\eqref{eq:convol}, this implies ${ \cc(A \cup A') } \subseteq { \cc((A \cup A') \cap A) \cup \cc((A \cup A') \setminus A) } \subseteq { \cc(A) \cup \cc(A') }$. 
This shows that $\cc$ is \v{C}ech. 

\eqref{thm:basic8}. 
Assume that $\cc$ and $\sss$ are \v{C}ech. 
Then $(\cc * \sss) * (\cc * \sss) = (\cc * \cc) * (\sss * \sss) = \cc * \sss$. 
So $\cc * \sss$ is \v{C}ech. 

\eqref{thm:basic9}. 
Assume that $\cc$ is topological. 
Then $\cc$ is \v{C}ech, so $\cc * \cc = \cc$. 
Moreover, $\cc$ is idempotent, so $\cc = \cc \circ \cc$. 
This shows that $\cc * \cc = \cc \circ \cc$. 
Conversely, assume that $\cc$ is untied and $\cc * \cc = \cc \circ \cc$. 
Since $\cc \leqslant \cc \circ \cc$ and $\cc * \cc \leqslant \cc$, we obtain $\cc = \cc \circ \cc$ and $\cc * \cc = \cc$, so $\cc$ is idempotent and \v{C}ech, i.e.\ topological.  

\eqref{thm:basic10}. 
Assume that $\cc$ and $\sss$ are topological. 
Then $\cc * \sss$ is idempotent and \v{C}ech, i.e.\ topological. 


\eqref{thm:basic12}. 
Let $A \subseteq E$. 
Then ${ \cc * \sss(A) } \subseteq { \cc(A \cap E) \cup \sss(A \setminus E) } = { \cc(A) \cup \sss(\emptyset) }$, so ${ \cc * \sss(A) } \subseteq { \cc(A) \cup \cc(\emptyset) } = { \cc(A) }$. 
This shows that $\cc * \sss \leqslant \cc$. 

\eqref{thm:basic13}. 
Let $A \subseteq E$ be $\cc$-closed. 
By \eqref{thm:basic12}, we have $\cc * \sss(A) \subseteq \cc(A) = A$, so $\cc * \sss(A) = A$. 
This proves that $A$ is $\cc * \sss$-closed. 

\eqref{thm:basic14}. 
Let $y \in { \cc(x) \cap \sss(x) }$, and let $B$ be a subset of $E$. 
If $x \in B$, then 
$
y \in { \cc(x) } 
  = { \cc(\{ x \} \cap B) } 
  \subseteq { \cc(\{ x \} \cap B) \cup \sss(\{ x \} \setminus B) }
$. 
If $x \notin B$, then 
$
y \in { \sss(x) } 
  = { \sss(\{ x \} \setminus B) } 
  \subseteq { \cc(\{ x \} \cap B) \cup \sss(\{ x \} \setminus B) }
$. 
This shows that $y \in { \cc * \sss(x) }$. 
Thus, ${ \cc(x) \cap \sss(x) } \subseteq { \cc * \sss(x) }$. 
By \eqref{thm:basic12} and the assumption $\cc(\emptyset) = \sss(\emptyset)$, the reverse inclusion holds. 
So ${ \cc(x) \cap \sss(x) } = { \cc * \sss(x) }$. 
\end{proof}

\begin{remark}
Let $\mathfrak{C}^{B}(E)$ denote the set of preclosure operators $\cc$ on $E$ such that $\cc(\emptyset) = B$. 
The previous result shows in particular that $(\mathfrak{C}^{B}(E), *, \leqslant)$ is a partially ordered monoid with neutral element $\langle \cdot \rangle^B$ and absorbing element $\bot^B$. 
Moreover, in $\mathfrak{C}^{\emptyset}(E)$, elements such that $\cc * \cc = \cc$ coincide with \v{C}ech preclosure operators, and elements such that $\cc * \cc = \cc \circ \cc$ coincide with topological closure operators. 
\end{remark}

\begin{remark}
Though we do not have a counterexample to give, it seems unlikely that the convolution product restrict to a binary relation on \textit{finitary} preclosure operators. 
Yet, we shall see with Theorem~\ref{thm:finitecopoints} that the convolution product $\cc * \sss$ is finitary in case $\sss$ is finitary and $\cc$ is a \textit{local} closure operator (see the definition in Section~\ref{subsec:local}).  
\end{remark}

\begin{remark}
More generally, one may consider the \textit{convolution product over $\mathrsfs{B}$}, defined as
\[
{ \cc *_{\mathrsfs{B}} \sss(A) } = { \bigcap_{B \in \mathrsfs{B}} \cc(A \cap B) \cup \sss(A \setminus B) },
\]
where $\mathrsfs{B}$ is a collection of subsets of $E$. 
Commutativity and associativity of $*_{\mathrsfs{B}}$ hold if $\mathrsfs{B}$ is a Boolean algebra. 
\end{remark}

\subsection{Continuous maps and functoriality}

Let $E$ be a set. 
Recall that $\mathfrak{C}(E)$ denotes the set of preclosure operators on $E$. 
This is a complete lattice: given a collection $(\cc_j)_{j \in J}$ of preclosure operators on $E$, its inf is given by
$A \mapsto \bigcap_{j \in J} \cc_j(A)$, 
and its sup by $A \mapsto \bigcup_{j \in J} \cc_j(A)$. 
The map $\mathfrak{C}$ extends to a contravariant functor from the category of sets to the category of complete lattices. 
Indeed, given a map $f : E \to E'$, we let $\mathfrak{C}(f) : \mathfrak{C}(E') \to \mathfrak{C}(E)$ be defined by 
\[
\mathfrak{C}(f)(\cc') := f^{-1} \cc' f : A \mapsto f^{-1}(\cc'(f(A))),
\] 
for every preclosure operator $\cc'$ on $E'$. 
Note that $f^{-1} \cc' f$ is untied (resp.\ finitary, idempotent) if $\cc'$ is. 

A map $f : E \to E'$ between two preclosure spaces $(E,\cc)$ and $(E',\cc')$ is \textit{$(\cc, \cc')$-continuous} if ${ f(\cc(A)) } \subseteq { \cc'(f(A)) }$, for all subsets $A$ of $E$, in other words if $\cc \leqslant f^{-1} \cc' f$. 

\begin{example}
Let $P$, $P'$ be qosets, and let $f : P \to P'$. 
Then $f$ is $(\downarrow\!\! \cdot, \downarrow\!\! \cdot)$-continuous if and only if it is \textit{order-preserving}, i.e.\ $f(x) \leqslant f(y)$ whenever $x \leqslant y$. 
\end{example}

\begin{example}[Example~\ref{ex:clogene3} continued]\label{ex:clogene4}
Let $E$ (resp.\ $E'$) be a set equipped with a collection $\mathrsfs{V}$ (resp.\ $\mathrsfs{V}'$) of subsets. 
Then a map $f : E \to E'$ is $(\langle \cdot \rangle_{\mathrsfs{V}}, \langle \cdot \rangle_{\mathrsfs{V}'})$-continuous if ${ f^{-1}(\mathrsfs{V}') } \subseteq { \mathrsfs{V} }$. (Note that the latter is sufficient but not necessary in general.)
This means that the map that associates to each pair $(E, \mathrsfs{V})$ the closure space $(E, \langle \cdot \rangle_{\mathrsfs{V}})$ extends to a functor. 
\end{example}

The constructions $\cc \mapsto \cc^{\circ}$ and $\cc \mapsto \overline{\cc}$ are functorial in the following sense. 

\begin{proposition}
Let $(E, \cc)$ and $(E', \cc')$ be preclosure spaces, and let $f : E \to E'$. 
If $f$ is $(\cc, \cc')$-continuous, then $f$ is $(\cc^{\circ}, \cc'^{\circ})$-continuous and $(\overline{\cc}, \overline{\cc}')$-continuous. 
\end{proposition}

\begin{proof}
The easy proof is left to the reader. 
\end{proof}

\begin{lemma}\label{lem:ineqf}
Let $(E', \cc', \sss')$ be a bi-preclosure space, and let $f : E \to E'$. 
Then 
\[
(f^{-1} \cc' f) * (f^{-1} \sss' f) \leqslant f^{-1} (\cc' * \sss') f.
\] 
Morever, if $f$ is injective, then the previous inequality is an equality. 
\end{lemma}

\begin{proof}
Take $\cc := f^{-1} \cc' f$ and $\sss := f^{-1} \sss' f$. 

Let $A \subseteq E$ and $x \in \cc * \sss(A)$. 
We want to show that $x \in f^{-1} (\cc' * \sss') f(A)$, i.e.\ $f(x) \in \cc' * \sss' (f(A))$. 
So let $B' \subseteq E'$. 
We take $B := f^{-1}(B')$. 
Then $x \in { \cc(A \cap B) \cup \sss(A \setminus B) }$, so $f(x) \in { \cc'(f(A \cap B)) \cup \sss'(f(A \setminus B)) } \subseteq { \cc'(f(A) \cap B') \cup \sss'(f(A) \setminus B') }$. 
This proves that $f(x) \in \cc' * \sss' (f(A))$, as required. 

Now, assume that $f$ is injective. 
Let $A \subseteq E$ and $x \notin  \cc * \sss(A)$. 
Then there is some $B_0 \subseteq E$ with $x \notin \cc(A \cap B_0)$ and $x \notin \sss(A \setminus B_0)$, equivalently $f(x) \notin \cc'(f(A \cap B_0))$ and $f(x) \notin \sss'(f(A \setminus B_0))$.  
Since $f$ is injective, we have ${ f(A \cap B_0) } = { f(A) \cap f(B_0) }$ and ${ f(A \setminus B_0) } = { f(A) \setminus f(B_0) }$, 
so we deduce that $f(x) \notin \cc' * \sss' (f(A))$, i.e.\ $x \notin f^{-1}(\cc' * \sss' (f(A)))$, as required. 
\end{proof}

The following results shows that the convolution product induces a functor from the category of bi-preclosure spaces to the category of preclosure spaces. 

\begin{proposition}\label{prop:functor}
Let $(E, \cc, \sss)$, $(E', \cc', \sss')$ be bi-preclosure spaces. 
If a map $f : E \to E'$ is both $(\cc, \cc')$-continuous and $(\sss, \sss')$-continuous, then $f$ is $(\cc * \sss, \cc' * \sss')$-continuous.
\end{proposition}

\begin{proof}
The assumptions of continuity rewrite as 
\[
\cc \leqslant f^{-1} \cc' f \mbox{ and } \sss \leqslant f^{-1} \sss' f. 
\]
By Theorem~\ref{thm:basic}, we have $\cc * \sss \leqslant (f^{-1} \cc' f) * (f^{-1} \sss' f)$. 
This yields 
\[
\cc * \sss \leqslant f^{-1} (\cc' * \sss') f
\]
by Lemma~\ref{lem:ineqf}, so $f$ is  $(\cc * \sss, \cc' * \sss')$-continuous.
\end{proof}

\section{Convolution of preclosure operators: special case}\label{sec:specialcase}

This section deepens our analysis of the convolution product by introducing \textit{enriched qosets} - qosets equipped with a compatible preclosure operator. 
We examine the significant case of the convolution $\cc^{\uparrow}$ of a preclosure operator $\cc$ with the dual Alexandrov operator $\uparrow\!\! \cdot$ induced by the qoset. 
Theorems \ref{thm:charac} and \ref{thm:convex} provide explicit formulas for this convolution. 
In the case where $\cc$ \textit{separates points}, sups of subsets are involved in these formulas, which lays the ground to our representation theorems (in their sup-generating formulations) of Sections \ref{sec:km} and \ref{sec:local}. 

\subsection{Enriched qosets}

Let $(E, \leqslant)$ be a qoset and $\cc$ be a preclosure operator on $E$. 
We say that $(E, \leqslant, \cc)$ (or merely $\cc$, if the context is clear) is
\begin{enumerate}
  \item\label{def:compat1} \textit{absorbing} if ${ \downarrow\!\! \cc(A) } = { \cc(A) } = { \cc(\downarrow\!\! A) }$, for all $A \subseteq E$; 
  \item\label{def:compat2} \textit{right-absorbing} if ${ \downarrow\!\! \cc(A) } = { \cc(A) }$, for all $A \subseteq E$;
  \item\label{def:compat3} \textit{compatible} if ${ \downarrow\!\! A } \subseteq { \cc(A) }$, for all $A \subseteq E$. 
\end{enumerate}
Note that \eqref{def:compat1} $\Rightarrow$ \eqref{def:compat2} $\Rightarrow$ \eqref{def:compat3}.  

\begin{proposition}\label{prop:topped}
Let $(E, \leqslant)$ be a qoset and $\cc$ be a preclosure operator on $E$. 
If $\cc$ is compatible and idempotent, then $\cc$ is absorbing. 
\end{proposition}
  
\begin{proof}
Let $A \subseteq E$. 
From the compatibility condition ${ \downarrow\!\! A } \subseteq { \cc(A) }$, we deduce ${ \cc(\downarrow\!\! A) } \subseteq { \cc(\cc(A)) } = { \cc(A) }$, since $\cc$ is idempotent. 
We also have ${ \downarrow\!\! \cc(A) } \subseteq { \cc(\cc(A)) } = { \cc(A) }$. 
This shows that $\cc$ is absorbing. 
\end{proof}

\begin{definition}
A triplet $(E, \leqslant, \cc)$ is an \textit{enriched qoset} if $(E, \leqslant)$ is a qoset and $\cc$ is a right-absorbing preclosure operator on $E$. 
\end{definition}


\begin{remark}
If $(E, \leqslant, \cc)$ is an enriched qoset, then so are $(E, \leqslant, \cc^{\circ})$ and $(E, \leqslant, \overline{\cc})$. 
\end{remark}

\subsection{Preclosure operators convoluted with a quasiorder}

Let $(E, \leqslant, \cc)$ be an enriched qoset. 
The convolution product of $\cc$ with $\downarrow\!\! \cdot$ (resp.\ $\uparrow\!\! \cdot$) is denoted by $\cc^{\downarrow}$ (resp.\ $\cc^{\uparrow}$), as a shorthand for ${ \cc } * { \downarrow\!\! \cdot }$ (resp.\ ${ \cc } * { \uparrow\!\! \cdot }$). 
Be aware of the inequalities $\cc^{\downarrow} \leqslant \cc$ and $\cc^{\uparrow} \leqslant \cc$. 
Note also that the mapping
\[
(E, \leqslant, \cc) \mapsto (E, \sim, \cc^{\uparrow})
\]
is an endofunctor of the category of enriched qosets (with order-preserving, continuous maps as morphisms), thanks to Proposition~\ref{prop:functor}. 

\begin{remark}
Given a collection of subsets $\mathrsfs{V}$, the notation $\langle A \rangle_{\mathrsfs{V}}^{\uparrow}$ will refer to the closure of $A$ with respect to the closure operator ${ \langle \cdot \rangle_{\mathrsfs{V}} } * { \uparrow\!\! \cdot }$, and should not be mixed up with the set of upper bounds of $\langle A \rangle_{\mathrsfs{V}}$. 
Similar remarks hold for the notations $\langle A \rangle_{\mathrsfs{V}}^{\downarrow}$ and $\langle A \rangle_{\mathrsfs{V}}^{\circ}$. 
\end{remark}

\begin{lemma}\label{lem:charac}
Let $(E, \leqslant, \cc)$ be an enriched qoset, and let $x \in E$. 
Then
\[
{ x \in \cc^{\uparrow}(A) } \Rightarrow { x \in \cc^{\uparrow}({ \downarrow\!\! x } \cap A) }, 
\]
for all $A \subseteq E$. 
\end{lemma}

\begin{proof}
Take ${ A_x } := { { \downarrow\!\! x } \cap A }$, and let $B$ be a subset of $E$. 
We need to show that $x \in { \cc(A_x \cap B) }$ or $x \in { \uparrow\!\! (A_x \setminus B) }$. 
Since $x \in { \cc^{\uparrow}(A) }$, we have $x \in { { \cc(A \cap B_x) } \cup { \uparrow\!\! (A \setminus B_x) } }$, where ${ B_x } := { { \downarrow\!\! x } \cap B }$.
If $x \in { \cc(A \cap B_x) }$, then $x \in { \cc(A \cap { \downarrow\!\! x } \cap B) } = { \cc(A_x \cap B) }$. 
If $x \in { \uparrow\!\! (A \setminus B_x) }$, there exists some $a \in { A \setminus B_x }$ with $a \leqslant x$; 
this implies that $a \in { A_x \setminus B }$, so $x \in { \uparrow\!\! (A_x \setminus B) }$. 
This proves that $x \in { \cc^{\uparrow}(A_x) }$.
\end{proof}

Let $(E, \leqslant)$ be a qoset. 
We say that a preclosure operator $\cc$ on $E$ \textit{separates points} (resp.\ \textit{dually separates points}) if, whenever $x \not\leqslant y$, there exists some $\cc$-closed subset $V$ such that $y \in V$ and $x \notin V$ (resp.\ $x \in V$ and $y \notin V$). 
Similarly, we say that a collection $\mathrsfs{V}$ of subsets of $E$ \textit{separates points} (resp.\ \textit{dually separates points}) if, whenever $x \not\leqslant y$, there exists some $V \in \mathrsfs{V}$ such that $y \in V$ and $x \notin V$ (resp.\ $x \in V$ and $y \notin V$). 
It is easily seen that $\mathrsfs{V}$ (dually) separates points if and only if $\langle \cdot \rangle_{\mathrsfs{V}}$ (dually) separates points. 

Recall that $\mathfrak{d}$ denotes the Dedekind--MacNeille closure operator $A \mapsto (A^{\uparrow})^{\downarrow}$, see Example~\ref{ex:alex}. 

\begin{proposition}\label{prop:sep}
Let $(E, \leqslant, \cc)$ be an enriched qoset. 
Then the following conditions are equivalent:
\begin{enumerate}
  \item\label{prop:sep1} $\cc$ separates points; 
  \item\label{prop:sep0} ${ \cc(\downarrow\!\! x) } = { \downarrow\!\! x }$, for all $x \in E$; 
  \item\label{prop:sep2} $x \notin { \cc(\downarrow\!\! y) }$, whenever $x \not \leqslant y$; 
  \item\label{prop:sep3} $\cc \leqslant \mathfrak{d}$. 
\end{enumerate}
Moreover, if these conditions are satisfied, then ${ x \leqslant y } \Leftrightarrow { x \in \cc(y) }$, for all $x, y \in E$. 
\end{proposition}

\begin{proof}
\eqref{prop:sep3} $\Rightarrow$ \eqref{prop:sep0}. 
Let $y \in \cc(\downarrow\!\! x)$. 
Take $A := { \downarrow\!\! x }$. 
Then $y \in \mathfrak{d}(A) = { (A^{\uparrow})^{\downarrow} } = { \downarrow\!\! x }$. 
This shows that ${ \cc(\downarrow\!\! x) } = { \downarrow\!\! x }$. 

\eqref{prop:sep0} $\Leftrightarrow$ \eqref{prop:sep2} is obvious. 

\eqref{prop:sep0} $\Rightarrow$ \eqref{prop:sep1}. 
If $x \not \leqslant y$, take $V := { \downarrow\!\! y }$, which is $\cc$-closed. 
Then $y \in V$ and $x \notin V$. 

\eqref{prop:sep1} $\Rightarrow$ \eqref{prop:sep3}. 
Let $x \in { \cc(A) }$ and $y \in { A^{\uparrow} }$. 
If $x \not\leqslant y$, there exists some $\cc$-closed subset $V$ with $y \in V$ and $x \notin V$. 
Then $A \subseteq { \downarrow\!\! y } \subseteq { V }$, so $x \in { \cc(A) } \subseteq { \cc(V) } = { V }$, a contradiction. 
So $x \leqslant y$. 
This shows that $x \in { (A^{\uparrow})^{\downarrow} } = { \mathfrak{d}(A) }$. 

Now, assume that $\cc$ separates points. 
We have ${ \cc(\downarrow\!\! x) } = { \downarrow\!\! x } \subseteq { \cc(x) } \subseteq { \cc(\downarrow\!\! x) }$, so that ${ \downarrow\!\! x } = { \cc(x) }$, for all $x \in E$. 
This amounts to $y \leqslant x$ if and only if $y \in \cc(x)$, for all $x, y \in E$. 
\end{proof}

\begin{remark}
If $\cc$ separates points, then $\cc^{\circ}$ and $\overline{\cc}$ separate points (use e.g.\ Condition~\eqref{prop:sep3} of the previous proposition). 
\end{remark}

\begin{example}
Let $P$ be a qoset. 
\begin{itemize}
  \item $\mathfrak{d}$ separates points, i.e.\ the collection of principal order-ideals of $P$ separate points. 
  \item The collection of principal filters of $P$ dually separates points, since $x \in { \uparrow\!\! x }$ and $y \notin { \uparrow\!\! x }$ whenever $x, y \in P$ with $x \not\leqslant y$. 
  \item If $P$ is a continuous directed-complete poset, then the collection of Scott-open filters of $P$ dually separates points, as per Lemma~\ref{lem:openfilter}. 
  \item If $P$ is a distributive semilattice, then the collection of prime ideals of $P$ separates points, see e.g.\ Ern\'e \cite[Theorem~4]{Erne06b}. 
\end{itemize}
\end{example}

Let $(E, \leqslant, \cc)$ be an enriched qoset. 
A subset $B \subseteq E$ \textit{has a $\cc$-sup} if $B$ has a sup in $E$ that belongs to $\cc(B)$ (in which case every sup of $B$ belongs to $\cc(B)$). 
Note that, if $\cc$ separates points and $B$ has a $\cc$-sup, then $B$ satisfies 
\[
{ \cc(B) } = { \mathfrak{d}(B) } = { \downarrow\!\! B^{\vee} }. 
\]

While the equation defining the convolution product as an intersection gives an ``outer'' representation, Equation~\eqref{eq:inner} below provides us with an ``inner'' representation. 

\begin{theorem}\label{thm:charac}
Let $(E, \leqslant, \cc)$ be an enriched qoset. 
Then 
\begin{equation}\label{eq:inner}
{ \cc^{\uparrow}(A) } = { \bigcup_{B \subseteq A} \cc(B) \cap B^{\uparrow} },
\end{equation}
for all $A \subseteq E$. 
Moreover, if $\cc$ separates points, then 
\begin{equation}\label{eq:inner2}
{ \cc^{\uparrow}(A) } = { \bigcup_{B \subseteq A} B^{\vee} }, 
\end{equation}
for all $A \subseteq E$, where the union runs over the subsets $B \subseteq A$ with a $\cc$-sup. 
\end{theorem}

\begin{proof}
We first prove the inclusion $\subseteq$ in Equation~\eqref{eq:inner}. 
So let $x \in { \cc^{\uparrow}(A) }$, and take ${ B } := { \downarrow\!\! x \cap A }$. 
Then $x$ is an upper bound of $B$, i.e.\ $x \in { B^{\uparrow} }$. 
Moreover, by Lemma~\ref{lem:charac}, $x \in { \cc^{\uparrow}(B) }$, so $x \in { \cc(B) }$. 
Thus, $x \in { \cc(B) \cap B^{\uparrow} }$. 

Now we prove the reverse inclusion $\supseteq$ in Equation~\eqref{eq:inner}. 
So let $x \in { \cc(B) \cap B^{\uparrow} }$, for some $B \subseteq A$, and let $B'$ be a subset of $E$. 
We need to show that $x \in { \cc(A \cap B') }$ or $x \in { \uparrow\!\! (A \setminus B') }$. 
Assume that $x \notin { \cc(A \cap B') }$. 
Then $B \not\subseteq { B' }$, so there exists some $b \in { B \setminus B' } \subseteq { A \setminus B' }$. 
Since $x \in { B^{\uparrow} }$, we have $b \leqslant x$, which yields $x \in { \uparrow\!\! (A \setminus B') }$. 
Thus, $x \in { \cc^{\uparrow}(A) }$. 

Now assume that $\cc$ separates points. 
Let us show that Equation~\eqref{eq:inner2} holds. 
So let $x \in { \cc^{\uparrow}(A) }$, and take ${ B } := { \downarrow\!\! x \cap A }$. 
We have just proved that $x \in { \cc(B) \cap B^{\uparrow} }$. 
Since $\cc \leqslant \mathfrak{d}$ by Proposition~\ref{prop:sep}, we get $x \in { \mathfrak{d}(B) \cap B^{\uparrow} }$. 
This exactly means that $x \in B^{\vee}$, which shows that $B$ has a $\cc$-sup. 
So the inclusion $\subseteq$ in Equation~\eqref{eq:inner2} holds. 
For the reverse inclusion, let $x \in B^{\vee}$, for some subset $B \subseteq A$ with a $\cc$-sup. 
Then $B$ has a sup $y \in \cc(B)$, so $x \sim y$. 
Thus, $x \in { \cc(B) }$, which yields $x \in { \cc(B) \cap B^{\uparrow} }$. 
So $x \in { \cc^{\uparrow}(A) }$, as required. 
\end{proof}

Let $P$ be a qoset, and let $A \subseteq K$ be subsets of $P$. 
Then $A$ \textit{sup-generates} $K$, or $K$ is \textit{sup-generated} by $A$, if
\[
x \in { (\downarrow\!\! x \cap A)^{\vee} },
\]
for all $x \in K$. 
An equivalent condition is that, whenever $x \in  K$, $y \in P$ with $x \not\leqslant y$, there exists some $a \in A$ with $a \leqslant x$ and $a \not\leqslant y$. 

\begin{corollary}\label{coro:supgene}
Let $P$ be a qoset and $A \subseteq P$. 
Then $A$ sup-generates $\mathfrak{d}^{\uparrow}(A)$, and the following conditions are equivalent:
\begin{enumerate}
  \item\label{coro:supgene1} $A$ sup-generates $P$;
  \item\label{coro:supgene2} $P = { \mathfrak{d}^{\uparrow}(A) }$;
  \item\label{coro:supgene3} the collection $\{ P \setminus { \uparrow\!\! a } : a \in A \}$ separates points.
\end{enumerate}
\end{corollary}

\begin{proof}
Since $\mathfrak{d}$ separates points and a subset $B$ has a $\mathfrak{d}$-sup if and only if it has a sup, we have 
\[
{ \mathfrak{d}^{\uparrow}(A) } = { \bigcup_{B \subseteq A} B^{\vee} },
\] 
for all $A \subseteq E$, where the union runs over the subsets $B \subseteq A$ with a sup. 
The fact that $A$ sup-generates $\mathfrak{d}^{\uparrow}(A)$ now easily follows from the definition. 

\eqref{coro:supgene1} $\Rightarrow$ \eqref{coro:supgene2}. 
Let $x \in P$. 
Since $A$ sup-generates $P$, we have $x \in { B^{\vee} }$ with $B := { \downarrow\!\! x \cap A }$, so $x \in { \mathfrak{d}^{\uparrow}(A) }$. 
This shows that $P = { \mathfrak{d}^{\uparrow}(A) }$. 

\eqref{coro:supgene2} $\Rightarrow$ \eqref{coro:supgene1}. 
From the fact that $A$ sup-generates $\mathfrak{d}^{\uparrow}(A)$ and the assumption $P = { \mathfrak{d}^{\uparrow}(A) }$, we obtain that $A$ sup-generates $P$. 

\eqref{coro:supgene3} $\Leftrightarrow$ \eqref{coro:supgene1}. 
Saying that $A$ sup-generates $P$ is equivalent to 
\[
\forall x, y \in  P \mbox { with } x \not\leqslant y, \exists a \in A : a \leqslant x \mbox{ and } a \not\leqslant y, 
\]
which is tatamount to the assertion that the collection $\{ P \setminus { \uparrow\!\! a } \}_{a \in A}$ separates points. 
\end{proof}

\begin{example}[Example~\ref{ex:ranzato0} continued]\label{ex:ranzato1}
Let $P$ be a qoset. 
Given a subset $B$, we have ${ \Min(B^{\uparrow}) } = { \mathfrak{q}(B) \cap B^{\uparrow} }$, where $\mathfrak{q}$ is the preclosure operator defined by 
\[
{ \mathfrak{q}(B) } = { \bigcap_{y \in B^{\uparrow}} P \setminus \Uparrow\!\! y }, 
\]
and $\Uparrow\!\! y$ denotes the subset $\{ x \in P : y < x \}$. 
Thus, the preclosure operator $\mathfrak{q}^{\uparrow}$ defined by Equation~\eqref{eq:pup} in Example~\ref{ex:ranzato0} is indeed the convolution of $\mathfrak{q}$ with $\uparrow\!\! \cdot$. 
Moreover, it is easily seen that $\mathfrak{q}$ separates points if and only if $P$ is a chain. 
Dual assertions hold for the preclosure operator $\mathfrak{p}$ defined by 
\[
{ \mathfrak{p}(B) } = { \bigcap_{y \in B^{\downarrow}} P \setminus \Downarrow\!\! y }, 
\]
where $\Downarrow\!\! y$ denotes the subset $\{ x \in P : x < y \}$. 
\end{example}

The following result notably provides us with sufficient conditions for $\cc^{\uparrow}$ to be finitary.

\begin{theorem}\label{thm:convex}
Let $(E, \leqslant, \cc)$ be an enriched qoset. 
Then $(\cc^{\uparrow})^{\circ} = (\cc^{\circ})^{\uparrow}$. 
Assume moreover that one of the following conditions is satisfied: 
\begin{enumerate}
  \item\label{thm:convex1} every principal order-ideal of $E$ is finite, or
  \item\label{thm:convex2} $\cc$ is finitary. 
\end{enumerate}
Then $\cc^{\uparrow}$ is finitary and 
\begin{equation}\label{eq:finitarybis}
{ \cc^{\uparrow}(A) } = { \bigcup_{F \subseteq A} \cc(F) \cap F^{\uparrow} },
\end{equation}
for all $A \subseteq E$, where the union runs over the finite subsets $F \subseteq A$. 
Moreover, if $\cc$ separates points, then 
\begin{equation}\label{eq:separating}
{ \cc^{\uparrow}(A) } = { \bigcup_{F \subseteq A} F^{\vee} },
\end{equation}
for all $A \subseteq E$, where the union runs over the finite subsets $F \subseteq A$ with a $\cc$-sup. 
\end{theorem}

\begin{proof}
Note first that it is equivalent to show that $\cc^{\uparrow}$ is finitary or to prove Equation~\eqref{eq:finitarybis}, thanks to Equation~\eqref{eq:inner} of Theorem~\ref{thm:charac}. 

Case \eqref{thm:convex1}. 
Let $x \in { \cc^{\uparrow}(A) }$, and let $F$ be the finite subset ${ \downarrow\!\! x } \cap A$ of $A$. 
By Lemma~\ref{lem:charac}, $x \in { \cc^{\uparrow}(F) }$. 
This proves that $\cc^{\uparrow}$ is finitary. 

Case \eqref{thm:convex2}. 
Let $x \in \cc^{\uparrow}(A)$. 
Then by Theorem~\ref{thm:charac}, $x \in \cc(B) \cap B^{\uparrow}$, for some $B \subseteq A$. 
Since $\cc$ is finitary, there is some finite subset $F$ of $B$ with $x \in \cc(F)$. 
From $F \subseteq B$ we get ${ B^{\uparrow} } \subseteq { F^{\uparrow} }$, hence $x \in { \cc(F) \cap F^{\uparrow} }$. 
This shows that Equation~\eqref{eq:finitarybis} holds, and that $\cc^{\uparrow}$ is finitary. 

Now, let us show that $(\cc^{\uparrow})^{\circ} = (\cc^{\circ})^{\uparrow}$. 
Take $\sss := \cc^{\circ}$. 
Then $\sss^{\uparrow} \leqslant \cc^{\uparrow}$, so $(\sss^{\uparrow})^{\circ} \leqslant (\cc^{\uparrow})^{\circ}$. 
From the first part of the proof, we know that $\sss^{\uparrow}$ is finitary, which yields $(\cc^{\circ})^{\uparrow} = \sss^{\uparrow} \leqslant (\cc^{\uparrow})^{\circ}$. 
To prove the converse inequality $(\cc^{\uparrow})^{\circ} \leqslant (\cc^{\circ})^{\uparrow}$, let $A \subseteq E$ and $x \in (\cc^{\uparrow})^{\circ}(A)$. 
Then $x \in \cc^{\uparrow}(F_0)$, for some finite subset $F_0 \subseteq A$. 
Thus, there is some subset $F_1$ of $F_0$ with $x \in { \cc(F_1) \cap F_1^{\uparrow} }$. 
Since $F_1$ is finite, we have $\cc(F_1) = \cc^{\circ}(F_1)$. 
So $x \in (\cc^{\circ})^{\uparrow}(A)$, as required. 

To conclude the proof, assume that $\cc$ separates points and $\cc^{\uparrow}$ is finitary, and let $x \in \cc^{\uparrow}(A)$. 
Let $F_0$ be a finite subset of $A$ with $x \in \cc^{\uparrow}(F_0)$. 
By Equation~\eqref{eq:inner2} of Theorem~\ref{thm:charac}, there is some (necessarily finite) subset $F_1$ of $F_0$ with $x \in F_1^{\vee}$. 
This shows that Equation~\eqref{eq:separating} holds. 
\end{proof}

\begin{remark}\label{rk:notation}
Since $(\cc^{\uparrow})^{\circ} = (\cc^{\circ})^{\uparrow}$, we shall denote this preclosure operator by $\cc^{\uparrow\circ}$. 
\end{remark}

\begin{remark}
The fact that $\cc$ finitary implies $\cc^{\uparrow}$ finitary can also be seen as a direct consequence of Theorem~\ref{thm:finitecopoints}. 
\end{remark}

\begin{example}[Example~\ref{ex:infconvex1} continued]\label{ex:infconvex2}
Let $(P, \leqslant)$ be a qoset. 
Let $\mathrsfs{U}$ denote the collection of upper, inf-closed subsets of $P$. 
Since $\mathrsfs{U}$ is stable under arbitrary intersections and directed unions, the closure operator $\langle \cdot \rangle_{\mathrsfs{U}}$ is finitary. 
Moreover,  
\[
{ \langle A \rangle_{\mathrsfs{U}}^{\downarrow} } = { \langle A \rangle_{\mathrsfs{H}} },
\]
for all $A \subseteq P$, where $\mathrsfs{H}$ denotes the collection of inf-closed subsets. 

To see why this holds, let us show first the inclusion ${ \langle A \rangle_{\mathrsfs{H}} } \subseteq { \langle A \rangle_{\mathrsfs{U}}^{\downarrow} }$. 
It suffices to show that $\langle A \rangle_{\mathrsfs{U}}^{\downarrow}$ is inf-closed. 
So let $F \subseteq { \langle A \rangle_{\mathrsfs{U}}^{\downarrow} }$ be a finite subset, let $B \subseteq P$, and take $U := { \langle A \cap B \rangle_{\mathrsfs{U}} } \in { \mathrsfs{U} }$. 
If $F \subseteq U$, then $F^{\wedge} \subseteq U$, since $U$ is inf-closed. 
Otherwise, there is some $f \in { F \setminus U }$. 
This yields $f \in { \downarrow\!\! (A \setminus B) }$, so $F^{\wedge} \subseteq { \downarrow\!\! f } \subseteq { \downarrow\!\! (A \setminus B) }$. 
This proves that 
\[
F^{\wedge} \subseteq { \bigcap_{B \subseteq P} { \langle A \cap B \rangle_{\mathrsfs{U}} } \cup { \downarrow\!\! (A \setminus B) } } = { \langle A \rangle_{\mathrsfs{U}}^{\downarrow} }.
\] 
So $\langle A \rangle_{\mathrsfs{U}}^{\downarrow}$ is inf-closed, as required. 

For the reverse inclusion, let $F \subseteq A$ be a finite subset with a $\langle \cdot \rangle_{\mathrsfs{U}}$-inf (i.e.\, with a $\langle \cdot \rangle_{\mathrsfs{U}}$-sup for the dual quasiorder $\geqslant$). 
Then $F \subseteq { \langle A \rangle_{\mathrsfs{H}} }$, so $F^{\wedge} \subseteq { \langle A \rangle_{\mathrsfs{H}} }$, since $\langle A \rangle_{\mathrsfs{H}}$ is inf-closed. 
Now, $\cc := { \langle \cdot \rangle_{\mathrsfs{U}} }$ is finitary, as already noted, and dually separates points, so we can apply Equation~\eqref{eq:separating} of Theorem~\ref{thm:convex} (taken dually) to obtain the desired inclusion. 
\end{example}

\section{Copoints, compact points, and extreme points}\label{sec:copoints}

In this section, we introduce the concepts of \textit{preinductive} preclosure operator and \textit{copoint} of an element. 
Given an enriched qoset $(E, \leqslant, \cc)$, we also define the notions of \textit{$\cc$-compact point} and \textit{$\cc$-extreme point}, and give various characterizations. 
In particular, we characterize $\cc$-extreme points when $\cc$ is the convolution of a preclosure operator with a preinductive closure operator (resp.\ with $\uparrow\!\! \cdot$). 

\subsection{Copoints and preinductivity}

Let $(E, \cc)$ be a preclosure space. 
Given $x \in E$, a maximal element of the set 
\begin{equation}\label{eq:cop}
\{ V \subseteq E : V \mbox{ $\cc$-closed and } V \not\ni x \},
\end{equation}
ordered by inclusion, is called a \textit{$\cc$-copoint of $x$}, and the set of $\cc$-copoints of $x$ is denoted by $\Cop_{\cc}(x)$. 
A \textit{$\cc$-copoint} is a subset which is the $\cc$-copoint of some $x \in E$. 
If $V \in \Cop_{\cc}(x)$, we say that $x$ is a \textit{$\cc$-attaching point} of $V$. 

Given $x \in E$, $\cc$ is called \textit{preinductive at $x$} if the set defined by \eqref{eq:cop} is preinductive (as per the definition given in Section~\ref{sub:max}). 
We call $\cc$ \textit{preinductive} if it is preinductive at every $x \in E$. 
Note that, if $\cc$ is preinductive, then every $x \notin \overline{\cc}(\emptyset)$ has a $\cc$-copoint, i.e.\ $\Cop_{\cc}(x)$ is nonempty. 

\begin{example}[Example~\ref{ex:clogene} continued]\label{ex:clogene2}
Let $\mathrsfs{V}$ be a collection of subsets of a set $E$. 
If $\mathrsfs{V}$ is stable under directed unions, then the closure operator $\langle \cdot \rangle_{\mathrsfs{V}}$ is preinductive. 
Moreover, every $\langle \cdot \rangle_{\mathrsfs{V}}$-copoint belongs to $\mathrsfs{V}$. 
In this situation, we better speak of \textit{$\mathrsfs{V}$-copoints}. 
\end{example}

\begin{proposition}\label{prop:finitaryispreinduc}
Let $(E, \cc)$ be a preclosure space. 
If $\cc$ is finitary, then $\cc$ is preinductive. 
\end{proposition}

\begin{proof}
By Lemma~\ref{lem:fin}, the collection of $\cc$-closed subsets is stable under directed unions. 
Thus, for all $x \in E$, the set defined by \eqref{eq:cop} is itself stable under directed unions, hence preinductive by Lemma~\ref{lem:preinductive}. 
\end{proof}

\begin{proposition}\label{prop:preinduc}
Let $(E, \cc)$ be a preclosure space. 
If $\cc$ is preinductive, then every $\cc$-closed subset is an intersection of $\cc$-copoints. 
\end{proposition}

\begin{proof}
Let $V$ be a $\cc$-closed subset and $x \in { \langle V \rangle_{\mathrsfs{V}_{\cc}} }$, where $\mathrsfs{V}_{\cc}$ denotes the collection of $\cc$-copoints. 
If $x \notin V$, there exists some $W \in \Cop_{\cc}(x)$ with $W \supseteq V$, since $\cc$ is preinductive. 
From $x \in { \langle V \rangle_{\mathrsfs{V}_{\cc}} }$, we get $x \in W$, a contradiction. 
So $x \in V$. 
This proves that ${ V } = { \langle V \rangle_{\mathrsfs{V}_{\cc}} }$, i.e.\ $V$ is an intersection of $\cc$-copoints. 
%
\end{proof}



%

\begin{proposition}\label{prop:ccpreinduc}
Let $(E, \cc)$ be a preclosure space. 
If $\cc$ is topological, then every $x \in E$ has at most one $\cc$-copoint, and the following conditions are equivalent:
\begin{enumerate}
  \item\label{prop:ccpreinduc1} $\cc$ is preinductive;
  \item\label{prop:ccpreinduc2} every $x \in E$ has a unique $\cc$-copoint;
  \item\label{prop:ccpreinduc3} $\cc$ is an Alexandrov closure operator. 
\end{enumerate}
\end{proposition}

\begin{proof}
Let $x \in E$, let $V$ be a $\cc$-copoint of $x$, and let $F$ be $\cc$-closed with $x \notin F$. 
Take $W := { V \cup F }$. 
By assumption, $\cc$ is topological, so $W$ is $\cc$-closed, and $x \notin { W }$. 
Thus, $W$ is a $\cc$-copoint of $x$. 
By maximality, $W = V$. 
So $V \supseteq F$, for all $\cc$-closed subsets $F$ with $x \notin F$. 
In particular, $x$ has at most one $\cc$-copoint. 

\eqref{prop:ccpreinduc1} $\Rightarrow$ \eqref{prop:ccpreinduc2}. 
Let $x \in E$. 
Since $\cc$ is preinductive and untied, $x$ has at least one $\cc$-copoint. 
So $x$ has a unique $\cc$-copoint. 


\eqref{prop:ccpreinduc2} $\Rightarrow$ \eqref{prop:ccpreinduc3}. 
We show that a union of $\cc$-closed subsets is $\cc$-closed. 
Combined with Proposition~\ref{prop:distr}, it is then an easy task to conclude that $\cc$ is Alexandrov closure operator. 
So let $(F_j)_{j \in J}$ be a collection of $\cc$-closed subsets, and take $F := { \bigcup_{j \in J} F_j }$. 
Let $x \in \cc(F)$, and let $V$ be the unique $\cc$-copoint of $x$. 
Suppose that $x \notin F$. 
Then $x \notin F_j$, for all $j \in J$. 
By the introductory part of the proof, $V \supseteq F_j$, for all $j \in J$. 
This yields $F \subseteq V$, so $\cc(F) \subseteq V$. 
Now, $x \in \cc(F)$, so $x \in V$, a contradiction. 
This proves that $x \in F$. 
So $\cc(F) = F$, i.e.\ $F$ is $\cc$-closed. 


\eqref{prop:ccpreinduc3} $\Rightarrow$ \eqref{prop:ccpreinduc1}. 
By assumption, there is some quasiorder $\leqslant$ such that ${ \cc(A) } = { \downarrow\!\! A }$, for all $A \subseteq E$. 
If $x \in E$, then $E \setminus { \uparrow\!\! x }$ is a lower subset, i.e.\ is $\cc$-closed. 
Moreover, if $F$ is a $\cc$-closed subset with $x \notin F$, then $F$ is a lower subset, so $F \subseteq { E \setminus { \uparrow\!\! x } }$. 
This proves that $E \setminus { \uparrow\!\! x }$ is the unique $\cc$-copoint of $x$, and that $\cc$ is preinductive. 
\end{proof}

\subsection{Compact points}

Let $(E, \leqslant, \cc)$ be an enriched qoset. 
If $x \in A \subseteq E$, then $x$ is a \textit{$\cc$-compact point} of $A$ if $x \notin \cc(A \setminus { \uparrow\!\! x })$. 
We denote by $\cp_{\cc} A$ the set of $\cc$-compact points of $A$. 

The following result gives a characterization of compact points. 

\begin{proposition}\label{prop:compacts}
Let $(E, \leqslant, \cc)$ be an enriched qoset, let $A \subseteq E$ and $x \in A$.
Consider the following assertions: 
\begin{enumerate}
  \item\label{prop:compacts1} $x$ is a $\cc$-compact point of $A$;
  \item\label{prop:compacts2} $x \in \cc(B)$ implies $x \in { \downarrow\!\! B }$, for all $B \subseteq A$;
  \item\label{prop:compacts4} $A \setminus { \uparrow\!\! x }$ is $\cc$-closed.
\end{enumerate}
Then \eqref{prop:compacts1} $\Leftrightarrow$ \eqref{prop:compacts2} $\Leftarrow$ \eqref{prop:compacts4}. 
Moreover, if $A$ is $\cc$-closed, then all conditions are equivalent. 
\end{proposition}

\begin{proof}
\eqref{prop:compacts4} $\Rightarrow$ \eqref{prop:compacts1}. 
If $x \in { \cc(A \setminus { \uparrow\!\! x }) }$, then $x \in { A \setminus { \uparrow\!\! x } }$ since $A \setminus { \uparrow\!\! x }$ is $\cc$-closed, a contradiction. 
So $x \notin { \cc(A \setminus { \uparrow\!\! x }) }$, i.e.\ $x$ is a $\cc$-compact point of $A$. 

\eqref{prop:compacts1} $\Rightarrow$ \eqref{prop:compacts2}. 
Suppose that $x \in \cc(B)$ for some $B \subseteq A$. 
If $x \notin { \downarrow\!\! B }$, then $B \subseteq { A \setminus { \uparrow\!\! x } }$, so $x \in { \cc(A \setminus { \uparrow\!\! x }) }$, a contradiction. 
Thus, $x \in { \downarrow\!\! B }$, as required. 

\eqref{prop:compacts2} $\Rightarrow$ \eqref{prop:compacts1}. 
Suppose that $x \in { \cc(A \setminus { \uparrow\!\! x }) }$. 
Then $x \in { \downarrow\!\! (A \setminus { \uparrow\!\! x }) }$. 
So there exists some $y \in { A \setminus { \uparrow\!\! x } }$ with $x \leqslant y$, a contradiction. 
Thus, $x \notin { \cc(A \setminus { \uparrow\!\! x }) }$, i.e.\ $x$ is a $\cc$-compact point of $A$. 

\eqref{prop:compacts1} $\Rightarrow$ \eqref{prop:compacts4} if $A$ is $\cc$-closed. 
We have ${ \cc(A \setminus { \uparrow\!\! x }) } \subseteq { \cc(A) } = { A }$. 
Suppose that there is some $y \in { \cc(A \setminus { \uparrow\!\! x }) }$ with $x \leqslant y$. 
Then $x \in { \downarrow\!\! \cc(A \setminus { \uparrow\!\! x }) } = { \cc(A \setminus { \uparrow\!\! x }) }$, a contradiction. 
We have proved that ${ \cc(A \setminus { \uparrow\!\! x }) } \subseteq { A \setminus { \uparrow\!\! x } }$, so $A \setminus { \uparrow\!\! x }$ is $\cc$-closed. 
\end{proof}

\begin{corollary}\label{coro:compacts}
Let $(E, \leqslant, \cc)$ be an enriched qoset, and let $x \in E$. 
Consider the following assertions: 
\begin{enumerate}
  \item\label{coro:compacts1} $x$ is a $\cc$-compact point of $E$;
  \item\label{coro:compacts2} $x \in \cc(B)$ implies $x \in { \downarrow\!\! B }$, for all $B \subseteq E$;
  \item\label{coro:compacts3} $E \setminus { \uparrow x }$ is $\cc$-closed;
  \item\label{coro:compacts5} $E \setminus { \uparrow x }$ is the unique $\cc$-copoint of $x$;
  \item\label{coro:compacts6} $x$ has a unique $\cc$-copoint. 
\end{enumerate}    
Then \eqref{coro:compacts1} $\Leftrightarrow$ \eqref{coro:compacts2} $\Leftrightarrow$ \eqref{coro:compacts3} $\Leftrightarrow$ \eqref{coro:compacts5} $\Rightarrow$ \eqref{coro:compacts6}. 
Moreover, if $\cc$ is preinductive and separates points, then all conditions are equivalent.  
\end{corollary}

\begin{proof}
\eqref{coro:compacts1} $\Leftrightarrow$ \eqref{coro:compacts2} $\Leftrightarrow$ \eqref{coro:compacts3} is already clear from Proposition~\ref{prop:compacts}. 

\eqref{coro:compacts5} $\Rightarrow$ \eqref{coro:compacts3} and \eqref{coro:compacts5} $\Rightarrow$ \eqref{coro:compacts6} are obvious. 

\eqref{coro:compacts3} $\Rightarrow$ \eqref{coro:compacts5}. 
We need to prove both existence and uniqueness. 
For existence, we prove that $E \setminus { \uparrow\!\! x }$ is a $\cc$-copoint of $x$. 
It is a $\cc$-closed subset by \eqref{coro:compacts3}, which does not contain $x$. 
Let us show that it is a maximal subset with this property. 
So let ${ V } \supseteq { E \setminus { \uparrow\!\! x } }$ be $\cc$-closed with $x \notin V$. 
Then ${ \uparrow\!\! x } \subseteq { E \setminus V }$, so $V \subseteq { E \setminus { \uparrow\!\! x } }$. 
Thus, ${ V } = { E \setminus { \uparrow\!\! x } }$, as required. 
To prove uniqueness, let $W$ be a $\cc$-copoint of $x$. 
Then again $W \subseteq { E \setminus { \uparrow\!\! x } }$. 
By maximality of $W$, we get ${ W } = { E \setminus { \uparrow\!\! x } }$. 
So $x$ has a unique $\cc$-copoint. 

\eqref{coro:compacts6} $\Rightarrow$ \eqref{coro:compacts3} if $\cc$ is preinductive and separates points. 
Let $V_x$ be the unique $\cc$-copoint of $x$. 
Then $x \in { E \setminus V_x }$ and $E \setminus V_x$ is an upper subset, so ${ \uparrow\!\! x } \subseteq { E \setminus V_x }$. 
We show ${ \uparrow\!\! x } = { E \setminus V_x }$. 
So let $y \in { E \setminus V_x }$. 
If $y \notin { \uparrow\!\! x }$, then $x \notin { \downarrow\!\! y }$. 
By assumption, $\cc$ separates points, so $\downarrow\!\! y$ is $\cc$-closed, hence there is some $\cc$-copoint $W$ of $x$ containing $\downarrow\!\! y$.
By uniqueness of $V_x$, we have $W = V_x$, so $y \in V_x$, a contradiction. 
Thus, $y \in { \uparrow\!\! x }$. 
This shows that ${ \uparrow\!\! x } = { E \setminus V_x }$, so ${ E \setminus { \uparrow\!\! x } }$ is $\cc$-closed. 
\end{proof}

\begin{remark}
The previous corollary, with Condition~\eqref{coro:compacts3}, shows in particular that the $\cc$-compact points of $E$ coincide with the $\overline{\cc}$-compact points of $E$. 
\end{remark}

\begin{example}[Example~\ref{ex:infconvex2} continued]\label{ex:infconvex3}
Let $P$ be a qoset. 
An element $x \in P$ is \textit{prime} if, whenever $x \geqslant f_0$ for some finite subset $F$ with an inf $f_0$, we have $x \geqslant f$ for some $f \in F$. 
Prime elements are exactly the $\langle \cdot \rangle_{\mathrsfs{U}}$-compact points of $(P, \geqslant)$, where $\mathrsfs{U}$ denotes the collection of upper, inf-closed subsets. 
Note that, in the set $\mathbb{N}$ of natural numbers, partially ordered by dual divisibility, prime elements in the previous sense coincide with prime powers (where $1$ is not considered as a prime power). 
\end{example}

\subsection{Extreme points}

Let $(E, \leqslant, \cc)$ be an enriched qoset. 
Recall that we write $[x]$ for the equivalence class of $x \in E$ with respect to the relation $\sim$, i.e.\ 
\[
{ y \in [x] } \Leftrightarrow { y \sim x } \Leftrightarrow { y \leqslant x \mbox{ and } x \leqslant y },  
\]
for all $x, y \in E$, and that we denote by $[A]$ the subset $\bigcup_{a \in A} [a]$, for all $A \subseteq E$. 

An element $x$ of a subset $A$ of $E$ is a \textit{$\cc$-extreme point} of $A$ if $x \notin \cc(A \setminus [x])$. 
We denote by $\ex_{\cc} A$ the set of $\cc$-extreme points of $A$. 
Note that, if no specific quasiorder is at stake, a preclosure space $(E, \cc)$ is often considered as the enriched qoset $(E, =, \cc)$.  

\begin{example}
Let $f : E \to E$ be any self-map on a set $E$. 
Let $\sss_{f}$ be the preclosure operator on $E$ defined by
\[
\sss_{f}(A) = A \cup f^{-1}(A),
\] 
for all $A \subseteq E$. 
Then $x \in A$ is an $\sss_{f}$-extreme point of $A$ if and only if $f(x) \notin A$ or $f(x) = x$, 
and $x$ is an $\sss_{f}$-extreme point of $E$ if and only if $x$ is a fixed point of $f$. 
\end{example}

\begin{example}
Let $f : E \to E'$ be a map between topological spaces $E$ and $E'$. 
We write $\overline{A}$ for the topological closure of $A \subseteq E$ in $E$. 
Let $\cc_f$ be the preclosure operator on $E$ defined by 
\[
\cc_f(A) = \bigcup_{{ f^{-1}(F') } \subseteq { A }} \overline{f^{-1}(F')}, 
\]
where $F'$ runs over the closed subsets of $E'$ such that ${ f^{-1}(F') } \subseteq { A }$. 
Then $x$ is a $\cc_f$-extreme point of $E$ if and only if $f$ is continuous at $x \in E$.  
\end{example}

\begin{corollary}\label{coro:extremes}
Let $(E, \leqslant, \cc)$ be an enriched qoset, and let $A \subseteq E$. 
Then 
\begin{align}\label{eq:excp}
{ \textstyle{\ex_{\cc}} A } = { { \textstyle{\cp_{\cc}} A } \cap { \Max A } }. 
\end{align}
If $x \in A$, consider the following assertions: 
\begin{enumerate}
  \item\label{coro:extremes1} $x$ is a $\cc$-extreme point of $A$;
  \item\label{coro:extremes2} $x \in \cc(B)$ implies $x \in [B]$, for all $B \subseteq A$;
  \item\label{coro:extremes4} $A \setminus [x]$ is $\cc$-closed. 
\end{enumerate}
Then \eqref{coro:extremes1} $\Leftrightarrow$ \eqref{coro:extremes2} $\Leftarrow$ \eqref{coro:extremes4}. 
Moreover, if $A$ is $\cc$-closed, then all conditions are equivalent. 
\end{corollary}

\begin{proof}
The second part of the corollary follows from Proposition~\ref{prop:compacts}, since extreme points in $(E, \leqslant, \cc)$ coincide with compact points in $(E, \sim, \cc)$. 
We still have to prove Equation~\eqref{eq:excp}. 
Let $x \in { { \cp_{\cc} A } \cap { \Max A } }$, and suppose that $x \in \cc(B)$ for some $B \subseteq A$. 
Since $x \in \cp_{\cc} A$, we have $x \in { \downarrow\!\! B}$. 
So $x \leqslant b$, for some $b \in B$. 
Now, $x \in \Max A$, so $x \sim b$. 
Thus, $x \in [B]$. 
This proves that $x \in \ex_{\cc} A$. 

Conversely, let $x \in \ex_{\cc} A$. 
Then obviously $x \in { \cp_{\cc} A }$. 
Let us show that $x \in  \Max A$. 
So let $a \in A$ with $x \leqslant a$. 
Then $x \in \cc(a)$, and using $x \in \ex_{\cc} A$ we obtain $x \sim a$, as required. 
\end{proof}

\begin{proposition}\label{prop:cpconvol}
Let $(E, \leqslant, \cc)$ be an enriched qoset and $\sss$ be a closure operator on $E$ that absorbs $\sim$. 
Let $A \subseteq E$ and $x \in E$. 
Then the following conditions are equivalent:
\begin{enumerate}
  \item\label{prop:cpconvol1} $x \in { \ex_{\cc * \sss} A }$;
  \item\label{prop:cpconvol3} $x \in { \ex_{\cc} (A \setminus V) }$, for some $\sss$-closed subset $V$ of $E$.  
\end{enumerate}
Moreover, if $\sss$ is preinductive, then one can choose $V \in \Cop_{\sss}(x)$ in \eqref{prop:cpconvol3}. 
\end{proposition}

\begin{proof}
We write $A_x$ as a shorthand for $A \setminus { [x] }$. 

\eqref{prop:cpconvol1} $\Rightarrow$ \eqref{prop:cpconvol3}. 
We have $x \in A$ and $x \notin { \cc * \sss(A_x) }$. 
Using the definition of the convolution product, there is some subset $B$ of $E$ with $x \notin { \cc(A_x \setminus B) }$ and $x \notin { \sss(A_x \cap B) }$. 
Take $V := \sss(A_x \cap B)$, which is $\sss$-closed since $\sss$ is idempotent. 
Then $x \in { A \setminus V }$. 
If $x \in { \cc(A_x \setminus V) }$, then $x \in { \cc(A_x \setminus (A_x \cap B)) } = { \cc(A_x \setminus B) }$, a contradiction. 
So $x \notin { \cc(A_x \setminus V) }$. 
This proves that $x \in { \ex_{\cc} (A \setminus V) }$. 
Moreover, if $\sss$ is preinductive, there exists some $W \in \Cop_{\sss}(x)$ with $W \supseteq V$, and it is easily seen that $x \in { \ex_{\cc} (A \setminus W) }$. 

\eqref{prop:cpconvol3} $\Rightarrow$ \eqref{prop:cpconvol1}. 
We have $x \in { A \setminus V}$ and $x \notin { \ex(A_x \setminus V) }$. 
If $x \in { \sss(A_x \cap V) }$, then $x \in { \sss(V) } = { V }$, a contradiction. 
So $x \notin { \sss(A_x \cap V) }$. 
This shows that $x \notin { \cc * \sss(A_x) }$. 
So $x \in { \ex_{\cc * \sss} A }$. 
\end{proof}

\begin{proposition}\label{prop:extrcharac}
Let $(E, \leqslant, \cc)$ be an enriched qoset. 
Let $A \subseteq E$ and $x \in A$. 
Consider the following assertions:
\begin{enumerate}
  \item\label{prop:extrcharac1} $x$ is a $\cc^{\uparrow}$-extreme point of $A$;
  \item\label{prop:extrcharac2a} $x \in { \ex_{\cc}(A \setminus U) }$, for some upper subset $U \subseteq E$; 
  \item\label{prop:extrcharac2b} $x \in { \ex_{\cc}(A \cap { \downarrow\!\! x }) }$;
  \item\label{prop:extrcharac3} $x \in { \cc(B) \cap B^{\uparrow} }$ implies $x \in [B]$, for all $B \subseteq A$;
  \item\label{prop:extrcharac4} $x \in B^{\vee}$ implies $x \in [B]$, for all $B \subseteq A$ with a $\cc$-sup;
  \item\label{prop:extrcharac5} $x \in { \cc(F) \cap F^{\uparrow} }$ implies $x \in [F]$, for all finite subsets $F \subseteq A$;
  \item\label{prop:extrcharac6} $x \in F^{\vee}$ implies $x \in [F]$, for all finite subsets $F \subseteq A$ with a $\cc$-sup.
\end{enumerate}
Then \eqref{prop:extrcharac1} $\Leftrightarrow$ \eqref{prop:extrcharac2a} $\Leftrightarrow$ \eqref{prop:extrcharac2b} $\Leftrightarrow$ \eqref{prop:extrcharac3} $\Rightarrow$ \eqref{prop:extrcharac4} $\Rightarrow$ \eqref{prop:extrcharac6} and \eqref{prop:extrcharac3} $\Rightarrow$ \eqref{prop:extrcharac5} $\Rightarrow$ \eqref{prop:extrcharac6}. 
Moreover,
\begin{itemize}
  \item if $\cc$ separates points, then \eqref{prop:extrcharac4} $\Rightarrow$ \eqref{prop:extrcharac3} and \eqref{prop:extrcharac6} $\Rightarrow$ \eqref{prop:extrcharac5};
  \item if $\cc$ is finitary, then \eqref{prop:extrcharac5} $\Rightarrow$ \eqref{prop:extrcharac3};
  \item if $\cc$ separates points and is finitary, then all conditions are equivalent.
\end{itemize}
\end{proposition}

\begin{proof}
\eqref{prop:extrcharac1} $\Leftrightarrow$ \eqref{prop:extrcharac2a} is a direct application of Proposition~\ref{prop:cpconvol} to the case of extreme points, with $\sss : A \to { \uparrow\!\! A }$. 

\eqref{prop:extrcharac2a} $\Rightarrow$ \eqref{prop:extrcharac2b}. 
Since $x \notin U$ and $U$ is an upper set, ${ \downarrow\!\! x } \subseteq { E \setminus U }$. 
So if $x \in \cc(A \cap { \downarrow\!\! x } \setminus { \uparrow\!\! x })$, then $x \in \cc(A \setminus U \setminus { \uparrow\!\! x })$, which contradicts $x \in { \ex_{\cc}(A \setminus U) }$. 
This proves that $x \in { \ex_{\cc}(A \cap { \downarrow\!\! x }) }$. 

\eqref{prop:extrcharac2b} $\Rightarrow$ \eqref{prop:extrcharac2a} is obvious with $U := { E \setminus { \downarrow\!\! x } }$. 

\eqref{prop:extrcharac1} $\Leftrightarrow$ \eqref{prop:extrcharac3} follows from the characterization of extreme points given by Corollary~\ref{coro:extremes} and Equation~\eqref{eq:inner} of Theorem~\ref{thm:charac}. 

\eqref{prop:extrcharac3} $\Rightarrow$ \eqref{prop:extrcharac4} and \eqref{prop:extrcharac5} $\Rightarrow$ \eqref{prop:extrcharac6} are clear with the inclusion $B^{\vee} \subseteq { \cc(B) \cap B^{\uparrow} }$, which holds for all subsets $B \subseteq A$ with a $\cc$-sup, as per the end of the proof of Theorem~\ref{thm:charac}. 

\eqref{prop:extrcharac4} $\Rightarrow$ \eqref{prop:extrcharac6} and \eqref{prop:extrcharac3} $\Rightarrow$ \eqref{prop:extrcharac5} are obvious. 

\eqref{prop:extrcharac4} $\Rightarrow$ \eqref{prop:extrcharac3} if $\cc$ separates points. 
Let $B \subseteq A$ with $x \in { \cc(B) \cap B^{\uparrow} }$. 
By Theorem~\ref{thm:charac}, $x \in { \cc^{\uparrow}(B) }$. 
Since $\cc$ separates points, $x \in { B_1^{\vee} }$, for some subset $B_1 \subseteq B$ with a $\cc$-sup, using again Theorem~\ref{thm:charac}. 
By \eqref{prop:extrcharac4}, we have $x \in [B_1]$, so $x \in [B]$. 

\eqref{prop:extrcharac6} $\Rightarrow$ \eqref{prop:extrcharac5} if $\cc$ separates points. 
The proof is similar to the lines just above. 

\eqref{prop:extrcharac5} $\Rightarrow$ \eqref{prop:extrcharac3} if $\cc$ is finitary. 
Let $B \subseteq A$ with $x \in { \cc(B) \cap B^{\uparrow} }$. 
Since $\cc$ is finitary, there is some finite subset $F \subseteq B$ with $x \in \cc(F)$. 
Moreover, $B^{\uparrow} \subseteq F^{\uparrow}$, so $x \in { \cc(F) \cap F^{\uparrow} }$. 
By \eqref{prop:extrcharac5}, we have $x \in [F]$, so $x \in [B]$.  

\eqref{prop:extrcharac6} $\Rightarrow$ \eqref{prop:extrcharac3} if $\cc$ separates points and is finitary. 
This is straightforward, for we have already proved that, in this case, \eqref{prop:extrcharac6} $\Rightarrow$ \eqref{prop:extrcharac5} and  \eqref{prop:extrcharac5} $\Rightarrow$ \eqref{prop:extrcharac3}. 
\end{proof}

\begin{example}[Example~\ref{ex:infconvex3} continued]\label{ex:infconvex4}
Let $P$ be a qoset. 
An element $x \in P$ is \textit{irreducible} if, whenever $x \in F^{\wedge}$ for some finite subset $F$, we have $x \in [F]$. 
Irreducible elements are exactly the $\langle \cdot \rangle_{\mathrsfs{H}}$-extreme points of $P$, where $\mathrsfs{H}$ denotes the collection of inf-closed subsets of $P$. 
Indeed, we already know that $\langle \cdot \rangle_{\mathrsfs{H}} = \langle \cdot \rangle_{\mathrsfs{U}}^{\downarrow}$ (see Example~\ref{ex:infconvex2}), where $\mathrsfs{U}$ denotes the collection of upper, inf-closed subsets. 
Since $\cc := { \langle \cdot \rangle_{\mathrsfs{U}} }$ is finitary and dually separates points, all conditions of the previous proposition are equivalent. 
In particular, Condition~\eqref{prop:extrcharac6} (taken dually) amounts to the definition of an irreducible element. 
\end{example}


%


\subsection{Galois embedding of enriched qosets}

Let $(E, \leqslant, \cc)$ and $(E', \leqslant, \cc')$ be enriched qosets. 
A \textit{Galois embedding} from $E$ to $E'$ is the data of two order-preserving, continuous maps $i : E \to E'$ and $s : E' \to E$ such that:
\begin{itemize}
  \item $s(i(x)) = x$, for all $x \in E$;
  \item $i(s(x')) \leqslant x'$, for all $x' \in E'$. 
\end{itemize}
Then, necessarily, $i$ is injective and $s$ is surjective. 
Moreover, the pair $(i, s)$ forms a \textit{Galois connection}, in the sense that 
\[
{ i(x) \leqslant x' } \Leftrightarrow { x \leqslant s(x') },
\]
for all $x \in E$, $x' \in E'$. 

\begin{proposition}
Let $(E, \leqslant, \cc)$, $(E', \leqslant, \cc')$ be enriched qosets, with a Galois embedding $(i, s)$ from $E$ to $E'$. 
Then 
\[
{ x \in { \textstyle{\cp_{\cc}} (A) } } \Leftrightarrow { i(x) \in { \textstyle{\cp_{\cc'}} (s^{-1}(A)) } },
\]
and
\[
{ x \in { \textstyle{\ex_{\cc^{\uparrow}}} (A) } } \Leftrightarrow { i(x) \in { \textstyle{\ex_{\cc'^{\uparrow}}} (s^{-1}(A)) } },
\]
for all $x \in E$ and $A \subseteq E$. 
\end{proposition}

\begin{proof}
Let $x \in E$ and $A \subseteq E$. 

Assume that $x \in { \cp_{\cc} (A) }$. 
Suppose that $i(x) \in \cc'(A')$, for some $A' \subseteq s^{-1}(A)$. 
Then $x = s(i(x)) \in s(\cc'(A')) \subseteq \cc(B)$, where $B := s(A')$, thanks to the continuity of $s$. 
Since $B \subseteq A$ and $x \in { \cp_{\cc} (A) }$, we get $x \in { \downarrow\!\! B }$. 
So there is some $y' \in A'$ with $x \leqslant s(y')$. 
Thus, $i(x) \leqslant i(s(y')) \leqslant y'$. 
This shows that $i(x) \in { \downarrow\!\! A' }$. 
So $i(x) \in { \cp_{\cc'} (s^{-1}(A)) }$. 

Conversely, assume that $i(x) \in { \cp_{\cc'} (s^{-1}(A)) }$. 
Suppose that $x \in \cc(B)$, for some $B \subseteq A$. 
Then $i(x) \in i(\cc(B)) \subseteq \cc'(i(B))$, thanks to the continuity of $i$. 
Since $i(B) \subseteq s^{-1}(A)$ and $i(x) \in { \cp_{\cc'} (s^{-1}(A)) }$, we get $i(x) \in { \downarrow\!\! i(B) }$. 
So there is some $y \in B$ with $i(x) \leqslant i(y)$. 
This implies $x \leqslant y$, so $x \in { \downarrow\!\! B }$. 
So $x \in { \cp_{\cc} (A) }$. 

Now, assume that $x \in { \ex_{\cc^{\uparrow}} (A) }$. 
Suppose that $i(x) \in { { \cc'(A') } \cap { A'^{\uparrow} } }$, for some $A' \subseteq s^{-1}(A)$. 
Then $x = s(i(x)) \in s(\cc'(A')) \subseteq \cc(B)$, where $B := s(A')$, thanks to the continuity of $s$. 
Moreover, $x = s(i(x)) \in { B^{\uparrow} }$. 
So $x \in { { \cc(B) } \cap { B^{\uparrow} } }$. 
Since $B \subseteq A$ and $x \in { \ex_{\cc^{\uparrow}} (A) }$, we get $x \in { [B] }$. 
So there is some $y' \in A'$ with $x \sim s(y')$. 
Thus, $i(x) \sim i(s(y')) \leqslant y'$. 
Recall that $i(x) \in { A'^{\uparrow} }$, so $y' \leqslant i(x)$. 
This shows that $i(x) \sim y'$, so $i(x) \in [A']$. 
So $i(x) \in { \ex_{\cc'^{\uparrow}} (s^{-1}(A)) }$. 

Conversely, assume that $i(x) \in { \ex_{\cc'^{\uparrow}} (s^{-1}(A)) }$. 
Let $B \subseteq A$ such that $x \in { { \cc(B) } \cap { B^{\uparrow} } }$. 
Then $i(x) \in i(\cc(B)) \subseteq \cc'(i(B))$, thanks to the continuity of $i$. 
Moreover, $i(x) \in i(B)^{\uparrow}$. 
Since $i(B) \subseteq s^{-1}(A)$ and $i(x) \in { \ex_{\cc'^{\uparrow}} (s^{-1}(A)) }$, we get $i(x) \in { [i(B)] }$. 
So there is some $y \in B$ with $i(x) \sim i(y)$. 
This implies $x \sim y$, so $x \in [B]$. 
So $x \in { \ex_{\cc^{\uparrow}} (A) }$. 
\end{proof}


\section{Extreme points in qosets}\label{sec:extr}

In this section, we examine some notions of extremality for elements of a qoset, and highlight links between them and their underlying (pre)closure operators. 
We do not attempt to review here all notions of extremality developed in the poset literature; rather, we  emphasize a new notion of extremality, namely that of \textit{relatively-maximal element}. 
We shall see that every relatively-maximal element is irreducible, while the converse does not hold in general. 
Relatively-maximal elements will reveal their importance in Section~\ref{sec:km} with our Representation Theorem I (Theorem~\ref{thm:rep1}), which has the notable consequence that every continuous directed-complete poset is inf-generated by the subset of its relatively-maximal elements. 

\begin{proposition}\label{prop:max}
Let $P$ be a qoset, and let $x \in P$. 
Then the following conditions are equivalent:
\begin{enumerate}
  \item\label{prop:max1} $x$ is maximal in $P$;
  \item\label{prop:max2} $x \in \ex_{\downarrow \cdot} P$. 
\end{enumerate}
\end{proposition}

\begin{proof}
We have $x \in { \downarrow\!\! (P \setminus [x]) }$ if and only if $x \leqslant y$ for some $y \not\sim x$, if and only if $x$ is not maximal in $P$. 
\end{proof}

\begin{proposition}\label{prop:irr}
Let $P$ be a qoset, and let $x \in P$. 
Then the following conditions are equivalent:
\begin{enumerate}
  \item\label{prop:irr1} $x$ is irreducible;
  \item\label{prop:irr2} $x \in \ex_{\langle \cdot \rangle_{\mathrsfs{H}}} P$, 
\end{enumerate}
where $\mathrsfs{H}$ denotes the collection of inf-closed subsets of $P$. 
\end{proposition}

\begin{proof}
See Example~\ref{ex:infconvex4}. 
\end{proof}

An element $x$ of a qoset $P$ is \textit{relatively-maximal} if there is some filter $T$ of $P$ such that $x \in { \Max (P \setminus T) }$. 
We denote by $\rMax P$ the set of relatively-maximal elements of $P$. 
We will see with Corollary~\ref{coro:irr} that every relatively-maximal element is irreducible. 
Following Jiang and Xu \cite[Definition~2.1]{Jiang08}, an order-ideal of a qoset $P$ is \textit{maximal relative to a filter} if there exists a filter $T$ of $P$ such that this order-ideal is maximal among the order-ideals that do not intersect $T$. 

\begin{proposition}\label{prop:spe}
Let $P$ be a qoset, and let $x \in P$. 
Then the following conditions are equivalent:
\begin{enumerate}
  \item\label{prop:spe1} $x$ is relatively-maximal;
  \item\label{prop:spe3} $\downarrow\!\! x$ is a maximal order-ideal relative to a filter; 
  \item\label{prop:spe2} $x \in \ex_{\mathfrak{t}^{\downarrow}} P$, 
\end{enumerate}
where $\mathfrak{t} := { \langle \cdot \rangle_{\mathrsfs{T}} }$ and $\mathrsfs{T}$ denotes the collection of filters of $P$. 
\end{proposition}

\begin{proof}
\eqref{prop:spe1} $\Leftrightarrow$ \eqref{prop:spe2} is a consequence of Proposition~\ref{prop:cpconvol}. 

\eqref{prop:spe1} $\Rightarrow$ \eqref{prop:spe3}.  
Since $x$ is relatively-maximal, we have $x \in { \Max (P \setminus T) }$, for some filter $T$. 
Suppose that $\downarrow\!\! x$ is not maximal relatively to $T$. 
Then there is some order-ideal $I$ that strictly contains $\downarrow\!\! x$ and such that ${ I \cap T } = { \emptyset }$. 
Thus, there is some $i \in I$ with $i \not\leqslant x$. 
Since $I$ is directed, there is some $i' \in I$ with $x \leqslant i'$ and $i \leqslant i'$. 
This implies that $x < i'$. 
Since $x$ is maximal in $P \setminus T$, this yields $i' \in { I \cap T }$, a contradiction. 
This shows that $\downarrow\!\! x$ is a maximal order-ideal relatively to the filter $T$. 

\eqref{prop:spe3} $\Rightarrow$ \eqref{prop:spe1}.  
Assume that $\downarrow\!\! x$ is a maximal order-ideal relative to some filter $T$. 
Then $x \in { P \setminus T }$. 
Let us show that $x \in { \Max (P \setminus T) }$. 
So let $y \in { P \setminus T }$ with $x \leqslant y$. 
If ${ \downarrow\!\! y } \cap { T }$ is nonempty, there is some $t \leqslant y$ with $t \in T$; then $y \in T$ since $T$ is an upper set, a contradiction. 
So ${ \downarrow\!\! y } \cap { T }$ is empty. 
Thus, ${ \downarrow\!\! x } = { \downarrow\!\! y }$, by maximality of $\downarrow\!\! x$. 
This means that $x \sim y$. 
So $x \in { \Max (P \setminus T) }$, which proves that $x$ is relatively-maximal. 
\end{proof}

We call an element $x$ of a qoset $P$ \textit{strongly irreducible} if $\Uparrow\!\! x$ is a filter. 
This terminology will be justified later with Corollary~\ref{coro:irr}. 
Note that strongly irreducible elements were called irreducible in \cite[Definition~I-3.5]{Gierz03} and used as a generalization to posets of irreducible elements defined on a semilattice. 

\begin{proposition}\label{prop:strirr}
Let $P$ be a qoset, and let $x \in P$. 
Then the following conditions are equivalent:
\begin{enumerate}
  \item\label{prop:strirr1} $x$ is strongly irreducible;
  \item\label{prop:strirr3} ${ x \in \Max (F^{\downarrow}) } \Rightarrow { x \in [F] }$, for all finite subsets $F$ of $P$; 
  \item\label{prop:strirr2} $x \in { \ex_{\mathfrak{p}^{\downarrow\circ}} P }$, 
\end{enumerate}
where $\mathfrak{p}$ is defined as per Example~\ref{ex:ranzato1}. 
\end{proposition}

\begin{proof}
Recall that $\mathfrak{p}^{\downarrow\circ}$ is the preclosure operator $(\mathfrak{p}^{\circ})^{\downarrow} = (\mathfrak{p}^{\downarrow})^{\circ}$, see Theorem~\ref{thm:convex} and Remark~\ref{rk:notation} (taken dually). 

\eqref{prop:strirr1} $\Rightarrow$ \eqref{prop:strirr3}. 
Let $F$ be a finite subset of $P$ with $x \in \Max (F^{\downarrow})$, and suppose that $x \notin [F]$. 
Then $F \subseteq { \Uparrow\!\! x }$. 
This latter subset is filtered by \eqref{prop:strirr1}, so $F \subseteq { \uparrow\!\! y } \subseteq { \Uparrow\!\! x }$, for some $y \in P$. 
This implies that $y \in F^{\downarrow}$ and $x < y$, which contradicts $x \in \Max (F^{\downarrow})$. 
Thus, $x \in [F]$, as required. 

\eqref{prop:strirr3} $\Rightarrow$ \eqref{prop:strirr1}. 
If $\Uparrow\!\! x$ is empty, then $x \in { \Max P }$, which can be written as $x \in { \Max (F^{\downarrow}) }$ with $F := \emptyset$. 
This yields $x \in [\emptyset]$, a contradiction. 
So $\Uparrow\!\! x$ is nonempty. 
Let $F$ be a finite subset of $\Uparrow\!\! x$. 
If $x \in { \Max F^{\downarrow} }$, then $x \in [F]$, a contradiction. 
So $x \notin { \Max F^{\downarrow} }$, i.e.\ there exists some $y \in { F^{\downarrow} }$ with $y > x$. 
This shows that $\Uparrow\!\! x$ is a filter, so $x$ is strongly irreducible.

\eqref{prop:strirr3} $\Leftrightarrow$ \eqref{prop:strirr2}. 
Combine Proposition~\ref{prop:extrcharac} (taken dually) and the fact that ${ \Max (F^{\downarrow}) } = { { \mathfrak{p}(F) } \cap { F^{\downarrow} } } = { { \mathfrak{p}^{\circ}(F) } \cap { F^{\downarrow} } }$, for all finite subsets $F$. 
\end{proof}


A \textit{semilattice} (resp.\ \textit{semilattice with a top}) is a poset in which every nonempty finite subset (resp.\  every finite subset) admits an inf. 
A qoset is \textit{Riesz} if, whenever $F \leqslant F'$ for finite subsets $F$, $F'$, there is some $x$ such that $F \leqslant x \leqslant F'$. 
Note that this is a self-dual property, in the sense that a qoset $(P, \leqslant)$ is Riesz if and only if its dual $(P, \geqslant)$ is Riesz. 
Moreover, every semilattice with a top is a Riesz poset. 

\begin{theorem}\label{thm:closures}
Let $P$ be a qoset. 
With the notations of the propositions above, we have
\begin{align}\label{eq:closures}
{ \langle A \rangle_{\mathrsfs{H}} } = { \langle A \rangle_{\mathrsfs{T}}^{\downarrow\circ} } \subseteq { \langle A \rangle_{\mathrsfs{T}}^{\downarrow} } \subseteq { \mathfrak{p}^{\downarrow}(A) },
\end{align}
for all subsets $A$ of $P$. 
Moreover, 
\begin{enumerate}
  \item\label{thm:closures0} if $P$ is Riesz, then $\mathfrak{p}^{\downarrow\circ}$ is idempotent, and ${ \langle \cdot \rangle_{\mathrsfs{H}} } = { \langle \cdot \rangle_{\mathrsfs{T}}^{\downarrow\circ} } = { \mathfrak{p}^{\downarrow\circ} }$;
  \item\label{thm:closures1} if $P$ is a semilattice, or if every principal filter of $P$ is finite, then $\langle \cdot \rangle_{\mathrsfs{T}}^{\downarrow}$ is finitary, and ${ \langle \cdot \rangle_{\mathrsfs{H}} } = { \langle \cdot \rangle_{\mathrsfs{T}}^{\downarrow} }$;
  \item\label{thm:closures2} if $P$ is a semilattice with a top, then $\langle \cdot \rangle_{\mathrsfs{T}}^{\downarrow}$ is finitary, and ${ \langle \cdot \rangle_{\mathrsfs{H}} } = { \langle \cdot \rangle_{\mathrsfs{T}}^{\downarrow} } = { \mathfrak{p}^{\downarrow\circ} }$. 
\end{enumerate}
\end{theorem}

\begin{proof}
We write $\mathrsfs{U}$ for the collection of upper, inf-closed subsets. 
Obviously, every filter is an upper, inf-closed subset, i.e.\ $\mathrsfs{T} \subseteq \mathrsfs{U}$. 
Thus, ${ \langle \cdot \rangle_{\mathrsfs{U}}^{\downarrow} } \leqslant { \langle \cdot \rangle_{\mathrsfs{T}}^{\downarrow} }$. 
Moreover, we know that ${ \langle \cdot \rangle_{\mathrsfs{H}} } = { \langle \cdot \rangle_{\mathrsfs{U}}^{\downarrow} }$, see Example~\ref{ex:infconvex2}. 
So ${ \langle \cdot \rangle_{\mathrsfs{H}} } \leqslant { \langle \cdot \rangle_{\mathrsfs{T}}^{\downarrow} }$. 
Since $\langle \cdot \rangle_{\mathrsfs{H}}$ is finitary, we deduce that 
\[
{ \langle \cdot \rangle_{\mathrsfs{H}} } \leqslant { \langle \cdot \rangle_{\mathrsfs{T}}^{\downarrow\circ} } \leqslant { \langle \cdot \rangle_{\mathrsfs{T}}^{\downarrow} }. 
\]

Let us show that ${ \langle \cdot \rangle_{\mathrsfs{T}}^{\downarrow\circ} } \leqslant { \langle \cdot \rangle_{\mathrsfs{H}} }$. 
It suffices to prove that ${ \langle F \rangle_{\mathrsfs{T}}^{\downarrow} } \subseteq { \langle F \rangle_{\mathrsfs{H}} }$, for all finite subsets $F$. 
So let $F$ be a finite subset and $x \in { \langle F \rangle_{\mathrsfs{T}}^{\downarrow} }$. 
Since $\mathrsfs{T}$ dually separates points, we have $x \in F_1^{\wedge}$, for some finite subset $F_1 \subseteq F$, by Equation~\eqref{eq:separating} of Theorem~\ref{thm:convex} (taken dually). 
The subset $\langle F_1 \rangle_{\mathrsfs{H}}$ is inf-closed, so it contains $F_1^{\wedge}$. 
This yields $x \in { \langle F \rangle_{\mathrsfs{H}} }$, as required. 

To finish the proof of Assertion~\eqref{eq:closures}, we now show that ${ \langle \cdot \rangle_{\mathrsfs{T}} } \leqslant { \mathfrak{p} }$, which will imply ${ \langle \cdot \rangle_{\mathrsfs{T}}^{\downarrow} } \leqslant { \mathfrak{p}^{\downarrow} }$. 
So let $x \in { \langle A \rangle_{\mathrsfs{T}} }$. 
If $y \in A^{\downarrow}$, then $A \subseteq { \uparrow\!\! y }$. 
Since $\uparrow\!\! y$ is a filter containing $A$, we have $x \in { \uparrow\!\! y }$. 
Thus, $x \in { P \setminus { \Downarrow\!\! y } }$. 
This shows that $x \in { \bigcap_{y \in A^{\downarrow}} { P \setminus { \Downarrow\!\! y } } } = { \mathfrak{p}(A) }$, as required. 


\eqref{thm:closures0}. 
To prove the assertion, it suffices to show that ${ \mathfrak{p}^{\downarrow}(F) } \subseteq { \langle F \rangle_{\mathrsfs{T}}^{\downarrow} }$, for all finite subsets $F$. 
So let $F$ be a finite subset and $x \in { \Max F^{\downarrow} }$. 
Let $\ell$ be a lower bound of $F$. 
Then $\{x, \ell \} \leqslant F$, so by the Riesz property there is some $y$ such that $\{ x, \ell \} \leqslant y \leqslant F$. 
By maximality of $x$, we deduce that $x \sim y$, which implies $\ell \leqslant x$. 
This proves that $x \in F^{\wedge}$. 
By Theorem~\ref{thm:convex}, $x \in { \langle F \rangle_{\mathrsfs{T}}^{\downarrow} }$. 
With Equation~\eqref{eq:pdown} of Example~\ref{ex:ranzato0} given for $\mathfrak{p}^{\downarrow}$, we deduce that ${ \mathfrak{p}^{\downarrow}(F) } \subseteq { \langle F \rangle_{\mathrsfs{T}}^{\downarrow} }$, as required. 

\eqref{thm:closures1}.
If $P$ is a semilattice, then $\langle \cdot \rangle_{\mathrsfs{T}}$ is obviously finitary, so $\langle \cdot \rangle_{\mathrsfs{T}}^{\downarrow}$ is finitary by Theorem~\ref{thm:convex}. 
Now, assume that all principal filters of $P$ are finite, and let $x \in { \langle A \rangle_{\mathrsfs{T}}^{\downarrow} }$. 
By Theorem~\ref{thm:charac}, there is some $B \subseteq A$ with $x \in { { \langle B \rangle_{\mathrsfs{T}} } \cap B^{\downarrow} }$. 
Then $B \subseteq { \uparrow\!\! x }$, and the latter set is finite by assumption, so $B$ is finite. 
This shows again that $\langle \cdot \rangle_{\mathrsfs{T}}^{\downarrow}$ is finitary. 
Thus, in both cases, $\langle \cdot \rangle_{\mathrsfs{T}}^{\downarrow}$ and $\langle \cdot \rangle_{\mathrsfs{H}}$ coincide. 

\eqref{thm:closures2}.
As a semilattice with a top, $P$ is a Riesz poset, and we have ${ \mathrsfs{T} } = { \mathrsfs{U} }$. 
Thus, ${ \langle \cdot \rangle_{\mathrsfs{T}}^{\downarrow} } = { \langle \cdot \rangle_{\mathrsfs{U}}^{\downarrow} } = { \langle \cdot \rangle_{\mathrsfs{H}} } = { \mathfrak{p}^{\downarrow\circ} }$, and $\langle \cdot \rangle_{\mathrsfs{T}}^{\downarrow}$ is finitary. 
\end{proof}

\begin{remark}\label{rk:sc1}
Another sufficient condition for $\langle \cdot \rangle_{\mathrsfs{T}}^{\downarrow}$ to be finitary (and for $\langle \cdot \rangle_{\mathrsfs{T}}^{\downarrow}$ and $\langle \cdot \rangle_{\mathrsfs{H}}$ to coincide) is that every element of $P$ has finitely many $\mathrsfs{T}$-copoints, see Theorem~\ref{thm:finitecopoints}. 
\end{remark}

The following result uses Theorem~\ref{thm:closures} to clarify the relation between irreducible, relatively-maximal, and strongly irreducible elements. 

\begin{corollary}\label{coro:irr}
Let $P$ be a qoset, and let $x \in P$. 
Consider the following assertions: 
\begin{enumerate}
	\item\label{irr4} $x$ is strongly irreducible in $P$; 
	\item\label{irr5bis} $x$ is relatively-maximal in $P$; 
	\item\label{irr3} $x$ is irreducible in $P$. 
\end{enumerate}
Then 
\eqref{irr4} $\Rightarrow$ \eqref{irr5bis} $\Rightarrow$ \eqref{irr3}. 
Moreover, 
\begin{itemize}
  \item if $P$ is Riesz, then \eqref{irr4} $\Leftrightarrow$ \eqref{irr5bis} $\Leftrightarrow$ \eqref{irr3}; 
  \item if $P$ is a semilattice, or if every principal filter of $P$ is finite, then \eqref{irr4} $\Rightarrow$ \eqref{irr5bis} $\Leftrightarrow$ \eqref{irr3}. 
\end{itemize}
\end{corollary}

\begin{proof}
\eqref{irr4} $\Rightarrow$ \eqref{irr5bis}. 
Let $T$ denote the filter $\Uparrow\!\! x$. 
Then $x$ is easily seen to be maximal in $P \setminus T$, so $x$ is relatively-maximal. 

\eqref{irr5bis} $\Rightarrow$ \eqref{irr3} is a direct consequence of the fact that ${ \langle A \rangle_{\mathrsfs{H}} } \subseteq { \langle A \rangle_{\mathrsfs{T}}^{\downarrow} }$, for all subsets $A$. 


\eqref{irr3} $\Rightarrow$ \eqref{irr4} if $P$ is Riesz. 
This follows directly from the fact that $\langle \cdot \rangle_{\mathrsfs{H}}$ and $\mathfrak{p}^{\downarrow\circ}$ coincide. 

\eqref{irr3} $\Rightarrow$ \eqref{irr5bis} if $P$ is a semilattice, or if all principal filters of $P$ are finite. 
This follows directly from the fact that $\langle \cdot \rangle_{\mathrsfs{H}}$ and $\langle \cdot \rangle_{\mathrsfs{T}}^{\downarrow}$ coincide. 
%
\end{proof}

\begin{remark}
The latter result shows in particular that, if $(P, \leqslant)$ is a semilattice with a top, then irreducibility, relative maximality, and strong irreducibility are equivalent properties.  
\end{remark}

\begin{remark}[Remark~\ref{rk:sc1} continued]\label{rk:sc2}
If every element of $P$ has finitely many $\mathrsfs{T}$-copoints, then irreducible elements and relatively-maximal elements coincide. 
\end{remark}

\begin{remark}\label{rk:complspe}
As a parallel to the previous notions of irreducibility, relative maximality, and strong irreducibility, we may consider the following ones. 
In a qoset $P$, an element $x$ is called
\begin{enumerate}
  \item\label{rk:complirr4} \textit{strongly completely irreducible} if $\Uparrow\!\! x$ is a principal filter; 
  \item\label{rk:complirr5} \textit{completely relatively-maximal} if $x \in { \Max(P \setminus { \uparrow\!\! y }) }$, for some $y$; 
  \item\label{rk:complirr6} \textit{completely irreducible} if ${ x \in B^{\wedge} } \Rightarrow { x \in [B] }$, for all subsets $B$. 
\end{enumerate}
It is easy to show that \eqref{rk:complirr4} $\Rightarrow$ \eqref{rk:complirr5} $\Rightarrow$ \eqref{rk:complirr6}. 
It is also possible to characterize each of these notions along the same lines as above. 
For instance, the following conditions are equivalent:
\begin{enumerate}
  \item\label{rk:complirr1} $x$ is strongly completely irreducible;
  \item\label{rk:complirr3} ${ x \in \Max (B^{\downarrow}) } \Rightarrow { x \in [B] }$, for all subsets $B$ of $P$; 
  \item\label{rk:complirr2} $x \in { \ex_{\mathfrak{p}^{\downarrow}} P }$, 
\end{enumerate}
where $\mathfrak{p}$ is defined as per Example~\ref{ex:ranzato1}. 
Indeed, \eqref{rk:complirr1} $\Leftrightarrow$ \eqref{rk:complirr3} was proved by Ranzato \cite[Theorem~6.1]{Ranzato02}; and to prove \eqref{rk:complirr3} $\Leftrightarrow$ \eqref{rk:complirr2}, one can combine Proposition~\ref{prop:extrcharac} (taken dually) and the fact that ${ \Max (B^{\downarrow}) } = { { \mathfrak{p}(B) } \cap { B^{\downarrow} } }$, for all subsets $B$. 
\end{remark}

\begin{example}\label{ex:finiteposet}
Consider the finite poset $P$ of Figure~\ref{fig:finiteposet}, where $P$ is the four element set $\{ x_1, x_2, y_1, y_2 \}$ endowed with the  relations $y_1 < x_1$, $y_1 < x_2$, $y_2 < x_1$, $y_2 < x_2$. 
Note that $P$ is not a Riesz poset. 
Moreover, $y_1, y_2$ are irreducible, but not strongly irreducible, since 
${ \Uparrow\!\! y_1 } = { \Uparrow\!\! y_2 } = { \{ x_1, x_2 \} }$ is not a filter.  
\end{example}

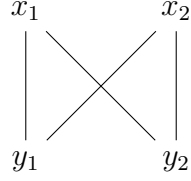
\begin{figure}
	\begin{center}
		\begin{tikzpicture}
			\node (y1) at (0,0) {$y_1$};
			\node (x1) at (0,2) {$x_1$};
			\node (y2) at (2,0) {$y_2$};
			\node (x2) at (2,2) {$x_2$};
			\draw (x1) -- (y1);
			\draw (x2) -- (y1);
			\draw (x2) -- (y2);
			\draw (x1) -- (y2);
		\end{tikzpicture}
	\end{center}
	\caption{Hasse diagram of the poset of Example~\ref{ex:finiteposet}. }
	\label{fig:finiteposet}
\end{figure}

\begin{example}\label{ex:hdP}
Consider the countably infinite poset $P$ of Figure~\ref{fig:hdP}, where $P$ is the set $\{ x_n \}_{n \geqslant 1} \cup \{ y_n \}_{n \geqslant 1}$ endowed with the following relations: 
\begin{align*}
y_1 &< \{ x_1, x_2, \ldots \} \\
y_n &< \{ x_1, \ldots, x_n \}
\end{align*}
for all $n \geqslant 1$, and no other relations hold. 
Note that $P$ is not Riesz, nor a semilattice. 
In this poset, every element is irreducible. 
In particular, $y_1$ is irreducible. 
However, $y_1$ is not relatively-maximal. 
We also see that $y_1$ is an inf of relatively-maximal elements ($y_1 = \bigwedge_{n \geqslant 1} x_n$). 
The element $y_2$ is relatively-maximal, since it is a maximal element in ${ P \setminus \uparrow\!\! y_1 } = { \{ y_2, y_3, \ldots \} }$. 
However, $y_2$ is not strongly irreducible as ${ \Uparrow\!\! y_2 } = { \{ x_1, x_2 \} }$ is not filtered. 
\end{example}

\begin{figure}
	\begin{center}
		\begin{tikzpicture}
			\node (y1) at (0,0) {$y_1$};
			\node (x1) at (0,2) {$x_1$};
			\node (y2) at (2,0) {$y_2$};
			\node (x2) at (2,2) {$x_2$};
			\node (y3) at (4,0) {$y_3$};
			\node (x3) at (4,2) {$x_3$};	
			\node (y4) at (6,0) {$y_4$};
			\node (x4) at (6,2) {$x_4$};	
			\node (x5) at (8,2) {$ $};
			\draw (y1) -- (x1);
			\draw (x2) -- (y1);
			\draw (y2) -- (x1);
			\draw (y2) -- (x2);
			\draw (x3) -- (y1);
			\draw (y3) -- (x1);
			\draw (y3) -- (x2);
			\draw (y3) -- (x3);
			\draw (x4) -- (y1);
			\draw (y4) -- (x1);
			\draw (y4) -- (x2);
			\draw (y4) -- (x3);
			\draw (y4) -- (x4);
			\draw[dashed] (x5) -- (y1);
		\end{tikzpicture}
	\end{center}
	\caption{Hasse diagram of the poset of Example~\ref{ex:hdP}. }
	\label{fig:hdP}
\end{figure}
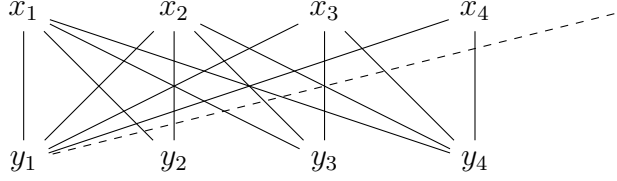

\section{Krein--Milman property and representation theorems}\label{sec:km}

This section focuses on the \textit{Krein--Milman property}: we show how this property can be transferred from a collection of subsets $\mathrsfs{Q}$ to another collection of subsets, thanks to the use of the convolution product. 
We apply this result to a variety of examples, where $\mathrsfs{Q}$ is typically the collection of \textit{preinductive} subsets of a qoset, as defined in Section~\ref{sec:posets}. 
The section culminates in two new representation theorems, which respectively generalize the Hofmann--Lawson theorem and the Birkhoff--Frink theorem, both recalled in the Introduction. 


\subsection{The Krein--Milman property}

Given a preclosure operator $\cc$ and a collection of subsets $\mathrsfs{V}$, we write $\cc_{\mathrsfs{V}}$ as a shorthand for $\cc * { \langle \cdot \rangle_{\mathrsfs{V}} }$. 
This notation happens to be consistent with the equality $\langle \cdot \rangle_{\mathrsfs{V}} = { \langle \cdot \rangle } * { \langle \cdot \rangle_{\mathrsfs{V}} }$, where $\langle \cdot \rangle$ is the closure operator defined in Proposition~\ref{prop:properties}. 


\begin{example}
Let $(E, \cc)$ be a closure space, quasi-ordered with $x \leqslant y$ if $x \in \cc(y)$, for all $x, y \in E$. 
Let $\mathrsfs{V}$ is a collection of $\cc$-open subsets containing all subsets of the form $E \setminus { \downarrow\!\! x }$, for $x \in E$. 
Then 
\[
{ x \in \cc_{\mathrsfs{V}}(A) } \Leftrightarrow { x \in \cc({ \downarrow\!\! x } \cap { A }) }, 
\]
for all $A \subseteq E$. 
\end{example}

\begin{example}
Let $E$ is a topological space, $\cc$ be the related topological closure operator on $E$, and $\mathrsfs{Q}$ be the collection of compact saturated subsets of $E$, i.e.\ compact subsets that are intersections of open sets. Then the $\cc_{\mathrsfs{Q}}$-closed subsets are exactly the closed subsets with respect to the \textit{patch topology} on $E$, see Goubault-Larrecq \cite[Definition~9.1.26]{Goubault13}. 
\end{example}

\begin{lemma}\label{lem:eqv}
Let $(E, \cc)$ be a preclosure space and $\mathrsfs{V}$ be a collection of subsets of $E$. 
Then the following equality holds:
\begin{equation}\label{eq:cv}
\cc_{\mathrsfs{V}}(A) = \bigcap_{V \in \mathrsfs{V}} V \cup \cc(A \setminus V),
\end{equation}
for all subsets $A$ of $E$. 
\end{lemma}

\begin{proof}
Let $\sss$ be the preclosure operator $A \mapsto { \bigcap_{V \in \mathrsfs{V}} V \cup \cc(A \setminus V) }$, and let $A \subseteq E$. 
If $x \in \cc_{\mathrsfs{V}}(A)$ and $V \in \mathrsfs{V}$, then $x \in { \langle A \cap V \rangle_{\mathrsfs{V}} \cup \cc(A \setminus V) } \subseteq { \langle V \rangle_{\mathrsfs{V}} \cup \cc(A \setminus V) } = { V \cup \cc(A \setminus V) }$. 
This proves that $x \in \sss(A)$. 

Now let $x \in \sss(A)$, and let $B \subseteq E$. 
If $x \notin { \langle A \cap B \rangle_{\mathrsfs{V}} }$, there exists some $V \in \mathrsfs{V}$ such that ${ V } \supseteq { A \cap B }$ and $x \notin V$. 
Then $x \in { \cc(A \setminus V) } \subseteq { \cc(A \setminus B) }$. 
This proves that $x \in \cc_{\mathrsfs{V}}(A)$. 
\end{proof}

\begin{definition}\label{def:kmp}
Let $(E, \leqslant, \cc)$ be an enriched qoset. 
A subset $K$ of $E$ (resp.\ a collection $\mathrsfs{K}$ of subsets of $E$) \textit{has the Krein--Milman property} with respect to $\cc$ (the \textit{$\cc$-KMp}, for short) if 
\begin{align}\label{eq:km}
\cc(K) = \cc(\textstyle{\ex_{\cc}} K)
\end{align}
(resp.\ for all $K \in \mathrsfs{K}$). 
\end{definition}

\begin{remark}
We may more generally define the Krein--Milman property with respect to a \textit{pair} $(\cc, \ttt)$ and replace Equation~\eqref{eq:km} by $\ttt(K) = \ttt(\textstyle{\ex_{\cc}} K)$. 
This would encompass the classical theorem proved by Krein and Milman \cite{Krein40} for compact convex subsets in locally-convex Hausdorff topological vector spaces, where both the convex hull operator and the closed convex hull operator are at play. 
Moreover, the transfer of the Krein--Milman property given by Theorem~\ref{thm:transfer} below would work equally well. 
However, we shall stick to Definition~\ref{def:kmp} in the frame of this paper, as the broader definition would remain unused. 
\end{remark}

Recall also that, if no specific quasiorder is at stake, a preclosure space $(E, \cc)$ is often considered as the enriched qoset $(E, =, \cc)$. 


\begin{example}
Let $f : E \to E'$ be a map between sets $E$ and $E'$. 
Then $f^{-1} f : A \mapsto f^{-1}(f(A))$ is a convexity operator on $E$. 
Moreover, a subset $K$ of $E$ has the $f^{-1} f$-KMp if and only if the restriction $f_{|K}$ of $f$ to $K$ is injective. 
\end{example}

\begin{example}\label{ex:polytopes1}
Given a finitary closure operator $\cc$ on a set $E$, a \textit{$\cc$-polytope} is a subset of the form $\cc(F)$, for some finite subset $F$ of $E$. 
Various conditions are known to be equivalent to the collection of $\cc$-polytopes having the $\cc$-KMp; this holds if and only if every $\cc$-copoint has a unique $\cc$-attaching point, see Edelman and Jamison \cite{Edelman85} and Poncet \cite[Theorem~7.2]{Poncet12c}. 
\end{example}

\begin{example}\label{ex:maximal}
The preinductive subsets of a qoset $P$ as defined in Section~\ref{sec:posets} are exactly the subsets that have the $\downarrow\!\! \cdot$-KMp, where $\downarrow\!\! \cdot$ denotes the Alexandrov closure operator. 
\end{example}

\begin{example}[Example~\ref{ex:ranzato1} continued]\label{ex:ranzato2}
Ranzato showed that a poset $P$ has the $\overline{\mathfrak{p}^{\downarrow}}$-KMp if $P$ is an algebraic sup-semilattice \cite[Remark~6.7]{Ranzato02} or if it is Noetherian \cite[Theorem~6.8]{Ranzato02}. 
\end{example}

As announced, the following result, though not difficult to prove, is one of the stepping stones of this paper. 
If $\mathrsfs{K}$, $\mathrsfs{V}$ are two collections of subsets of a set $E$, we write $\mathrsfs{K} \doublesetminus \mathrsfs{V}$ for the collection
\[
{ \mathrsfs{K} \doublesetminus \mathrsfs{V} } := { \{ K \setminus V : K \in \mathrsfs{K}, V \in \mathrsfs{V} \} }. 
\]

\begin{theorem}[Transfer of the Krein--Milman property]\label{thm:transfer}
Let $(E, \leqslant, \cc)$ be an enriched qoset and $\sss$ be a closure operator on $E$ that absorbs $\sim$. 
Let $\mathrsfs{K}$ be a collection of subsets of $E$. 
Assume that $\mathrsfs{K} \doublesetminus \mathrsfs{V}$ has the $\cc$-KMp, 
for some collection of subsets $\mathrsfs{V}$ that generates $\sss$, i.e.\ ${ \sss } = { \langle \cdot \rangle_{\mathrsfs{V}} }$. 
Then $\mathrsfs{K}$ has the $\cc * \sss$-KMp. 
\end{theorem}

\begin{proof}
Take $\ttt := { \cc * \sss } = { \cc_{\mathrsfs{V}} }$. 
Let $K \in \mathrsfs{K}$, $x \in { \ttt(K) }$, and $V \in \mathrsfs{V}$ such that $V \not\ni x$. 
Then $x \in { \cc(K \setminus V) }$ by Lemma~\ref{lem:eqv}. 
Since $\mathrsfs{K} \doublesetminus \mathrsfs{V}$ has the $\cc$-KMp, we deduce that $x \in { \cc(\ex_{\cc} (K \setminus V)) }$. 
By Proposition~\ref{prop:cpconvol}, ${ \ex_{\cc} (K \setminus V) } \subseteq { (\ex_{\ttt} K) \setminus V }$, so that $x \in \cc((\ex_{\ttt} K) \setminus V)$. 
Using again Lemma~\ref{lem:eqv}, we obtain $x \in { \ttt(\ex_{\ttt} K) }$. 
This shows that $\mathrsfs{K}$ has the $\ttt$-KMp, as required. 
\end{proof}

\begin{example}[Example~\ref{ex:polytopes1} continued]\label{ex:polytopes2}
Let $(E, \leqslant, \cc)$ be an enriched qoset, where $\cc$ is a finitary closure operator. 
If the collection of $\cc$-polytopes has the $\cc$-KMp, then the collection of $\cc^{\uparrow}$-polytopes has the $\cc^{\uparrow}$-KMp. 
Indeed, let $\mathrsfs{K}$ be the collection of $\cc^{\uparrow}$-polytopes, and $\mathrsfs{U}$ be the collection of upper subsets. 
Let $L \in { \mathrsfs{K} \doublesetminus \mathrsfs{U} }$. 
Then $L$ can be written as $\cc^{\uparrow}(F) \setminus U$, for some finite subset $F$ and some upper subset $U$. 
We have ${ F \setminus U } \subseteq { L } \subseteq { \cc(F \setminus U) } =: P$, and this latter subset is a $\cc$-polytope. 
By assumption, $P$ has the $\cc$-KMp. 
Thus, ${ \cc(L) } = { P } = { \cc(\ex_{\cc} P) }$. 
Moreover, $L \subseteq P$ implies ${ \ex_{\cc} P } \subseteq { \ex_{\cc} L }$, so that $\cc(L) = \cc(\ex_{\cc} L)$. 
This shows that $\mathrsfs{K} \doublesetminus \mathrsfs{U}$ has the $\cc$-KMp. 
From Theorem~\ref{thm:transfer}, we deduce that $\mathrsfs{K}$ has the $\cc^{\uparrow}$-KMp, as required. 
\end{example}

Recall that a subset $A$ of a qoset is called \textit{order-saturated} if $[A] = A$. 

\begin{example}[Example~\ref{ex:maximal} continued]\label{ex:maximal2}
Let $P$ be a qoset and $\mathrsfs{V}$ be a collection of order-saturated subsets of $P$. 
Consider the enriched qoset $(P, \sim, \downarrow\!\!\cdot)$ equipped with $\sss := \langle \cdot \rangle_{\mathrsfs{V}})$. 
Let $\mathrsfs{K}_{\mathrsfs{V}}$ be the collection made of the subsets $K$ of $P$ such that $K \setminus V$ be preinductive for all $V \in \mathrsfs{V}$. 
Then $\mathrsfs{K}_{\mathrsfs{V}}$ contains all finite subsets of $P$, and more generally all subsets satisfying the ascending chain condition (see Lemma~\ref{lem:preinductive}). 
Moreover, $\mathrsfs{K}_{\mathrsfs{V}} \doublesetminus \mathrsfs{V}$ has the $\downarrow\!\!\cdot$-KMp. 
So $\mathrsfs{K}_{\mathrsfs{V}}$ has the $\langle \cdot \rangle_{\mathrsfs{V}}^{\downarrow}$-KMp, by Theorem~\ref{thm:transfer}. 
In particular, if $P$ is Noetherian, then $P$ has the $\langle \cdot \rangle_{\mathrsfs{V}}^{\downarrow}$-KMp. 
\end{example}


\begin{example}[Example~\ref{ex:maximal2} continued]
Let $P$ be a qoset equipped with an upper semiclosed topology, and $\mathrsfs{V}$ be a collection of order-saturated, open subsets.
Let $\mathrsfs{K}$ be the collection of compact subsets of $P$. 
Then every element of $\mathrsfs{K} \doublesetminus \mathrsfs{V}$ is compact, hence preinductive by Lemma~\ref{lem:wallace}. 
Thus, $\mathrsfs{K}$ has the $\langle \cdot \rangle_{\mathrsfs{V}}^{\downarrow}$-KMp. 
\end{example}

\begin{example}[Example~\ref{ex:maximal2} continued]
Let $P$ be a monotone convergence space and $\mathrsfs{V}$ be a collection of open subsets. 
Every $V \in \mathrsfs{V}$ is an upper subset, hence is order-saturated. 
Let $\mathrsfs{F}$ be the collection of closed subsets of $P$. 
Then every element of $\mathrsfs{F} \doublesetminus \mathrsfs{V}$ is closed, hence preinductive by Lemma~\ref{lem:mcs}. 
Thus, $\mathrsfs{F}$ has the $\langle \cdot \rangle_{\mathrsfs{V}}^{\downarrow}$-KMp. 
\end{example}

\begin{example}[Example~\ref{ex:maximal2} continued]\label{ex:scc}
Let $P$ be a qoset equipped with a topology, and $\mathrsfs{V}$ be a collection of open upper subsets.
Let $\mathrsfs{S}$ be the collection of strongly chain-complete subsets of $P$. 
Then every element of $\mathrsfs{S} \doublesetminus \mathrsfs{V}$ is easily seen to be strongly chain-complete, hence preinductive. 
Thus, $\mathrsfs{S}$ has the $\langle \cdot \rangle_{\mathrsfs{V}}^{\downarrow}$-KMp. 
\end{example}

\subsection{Representation theorems}\label{subsec:repthm}

Let $P$ be a qoset, and let $A \subseteq K$ be subsets of $P$. 
Recall from Section~\ref{sec:specialcase} 
that $A$ \textit{sup-generates} (resp.\ \textit{inf-generates}) $K$, or $K$ is \textit{sup-generated} (resp.\ \textit{inf-generated}) by $A$, if
\[
x \in { (\downarrow\!\! x \cap A)^{\vee} },
\]
(resp.\ $x \in { (\uparrow\!\! x \cap A)^{\wedge} }$) for all $x \in K$. 
An equivalent condition is that, whenever $x \in  K$, $y \in P$ with $x \not\leqslant y$, there exists some $a \in A$ with $a \leqslant x$ and $a \not\leqslant y$ \cite[Proposition~I-3.9]{Gierz03}. 
Informally, we speak of a representation of points of $K$ by $A$. 
Note that, in this case, every $x \in K$ is also a sup (resp.\ an inf) of $\downarrow\!\! x \cap A$ (resp.\ of $\uparrow\!\! x \cap A$) in the qoset $K$. 


\begin{example}[Example~\ref{ex:maximal2} continued]\label{ex:maximal3}
Let $P$ be a qoset and $\mathrsfs{V}$ be a collection of upper subsets of $P$ that dually separates points.
Let $\mathrsfs{K}_{\mathrsfs{V}}$ be the collection made of the subsets $K$ of $P$ such that $K \setminus V$ be preinductive for all $V \in \mathrsfs{V}$. 
From Example~\ref{ex:maximal2}, $\mathrsfs{K}_{\mathrsfs{V}}$ has the $\langle \cdot \rangle_{\mathrsfs{V}}^{\downarrow}$-KMp. 
Thus, if $K \in \mathrsfs{K}_{\mathrsfs{V}}$ and $x \in K$, then $x \in \langle \ex K \rangle_{\mathrsfs{V}}^{\downarrow}$, where $\ex K$ denotes the set of extreme points of $K$ with respect to $\langle \cdot \rangle_{\mathrsfs{V}}^{\downarrow}$.
By assumption, the collection $\mathrsfs{V}$ dually separates points.  
Thus, by Theorem~\ref{thm:charac} (taken dually), $x \in B^{\wedge}$, for some subset $B$ of $\ex K$ with an inf, which implies that $x \in (\uparrow\!\! x \cap \ex K)^\wedge$. 
This shows that $K$ is inf-generated by the set of its extreme points with respect to $\langle \cdot \rangle_{\mathrsfs{V}}^{\downarrow}$, for all $K \in \mathrsfs{K}_{\mathrsfs{V}}$. 
In particular, if $P$ is Noetherian, then $P$ is inf-generated by $\ex_{\langle \cdot \rangle_{\mathrsfs{V}}^{\downarrow}} P$. 
\end{example}

It was proved by Hofmann and Lawson \cite[Proposition~2.7]{HofmannLawson76} that irreducible elements form an inf-generating subset of every continuous lattice; see also \cite[Corollary~I-3.10]{Gierz03} for a generalization to continuous directed-complete semilattices. 
The following theorem, applied to the case of a continuous directed-complete poset equipped with the Scott or Lawson topology, generalizes \cite[Corollary~I-3.10]{Gierz03} further, as it does not require the ambient poset to be a semilattice. 

A topology on a qoset is \textit{lower semiclosed} if each principal order-ideal is a closed subset. 
We say that a qoset equipped with a topology \textit{has small open filters} if, whenever $x \in U$ for some open upper subset $U$, there exists an open filter $V$ such that $x \in { V } \subseteq { U }$. 

\begin{theorem}[Representation Theorem I]\label{thm:rep1}
Let $P$ be a qoset equipped with a lower semiclosed topology. 
Assume that $P$ has small open filters, and let  $S$ be a strongly chain-complete subset of $P$. 
Then $S$ is inf-generated by the set of its relatively-maximal elements. 
\end{theorem}

\begin{proof}
Let $\mathrsfs{V}$ denote the collection of open filters of $P$. 
Then $\mathrsfs{V}$ dually separates points, thanks to the assumptions made on the topology of $P$. 
Besides, $S \setminus V$ is strongly chain-complete, hence preinductive, for all $V \in \mathrsfs{V}$. 
By Example~\ref{ex:maximal3}, $S$ is inf-generated by the set $\ex S$ of its extreme points with respect to $\langle \cdot \rangle_{\mathrsfs{V}}^{\downarrow}$. 
To finish the proof, it suffices to note that $\ex S$ is included in $\rMax S$, the set of relatively-maximal elements of $S$, by Proposition~\ref{prop:spe}.  
\end{proof}




The preceding point raises the question about what posets with a given topology have small open filters. 
We refer to \cite{Gierz03} for the notions used in the following proposition and not already defined above. 

\begin{proposition}\label{prop:ulc}
Let $P$ be a poset equipped with a topology. 
Assume that one of the following conditions is satisfied:
\begin{enumerate}
  \item\label{prop:ulc1} $P$ is a continuous directed-complete poset equipped with the Scott or Lawson topology, or 
  \item\label{prop:ulc2} $P$ is a semitopological semilattice such that, for all $x \in U$ open, there exists $z \in U$ with $x \in { (\uparrow\!\! z)^{\circ} }$, or 
  \item\label{prop:ulc3} $P$ is a locally-compact Hausdorff topological semilattice with small semilattices. 
\end{enumerate}
Then the topology is lower semiclosed and $P$ has small open filters. 
\end{proposition}

\begin{proof}
Case~\eqref{prop:ulc1}. 
Use \cite[Proposition I-3.3]{Gierz03} and \cite[Proposition~III-1.6(i)]{Gierz03}. 

Case~\eqref{prop:ulc2}. 
Adapt the technique of the proof of \cite[Proposition I-3.3]{Gierz03}, using the fact that the topological interior of an upper set is again an upper set in a semitopological semilattice. 

Case~\eqref{prop:ulc3}. 
For $x \in U$, use local compactness to pick a compact neighborhood $K$ of $x$ contained in $U$. 
Use the assumption of small semilattices to find a semilattice neighborhood $V$ of $x$ contained in $K$. 
Then $V \subseteq K$ is a compact semilattice, so it has a least element $z$ by \cite[Proposition~VI-1.13(v)]{Gierz03}.  
We thus have \eqref{prop:ulc2} and hence the desired property.
\end{proof}

We say that a qoset equipped with a topology \textit{has small open principal filters} if, whenever $x \in U$ for some open upper subset $U$, there exists an open principal filter $V$ such that $x \in { V } \subseteq { U }$. 
Recall from Remark~\ref{rk:complspe} that an element $x$ of a qoset $P$ is \textit{completely relatively-maximal} if $x \in { \Max(P \setminus { \uparrow\!\! y }) }$, for some $y \in P$. 

\begin{theorem}[Representation Theorem II]\label{thm:rep2}
Let $P$ be a qoset equipped with a lower semiclosed topology. 
Assume that $P$ has small open principal filters, and let $S$ be a strongly chain-complete subset of $P$. 
Then $S$ is inf-generated by the set of its completely relatively-maximal elements. 
\end{theorem}

\begin{proof}
The proof is very similar to that of Theorem~\ref{thm:rep1}. 
\end{proof}

\begin{example}[Compactly generated posets]
Let $P$ be a poset. 
An element $y$ of $P$ is \textit{compact} if $y \ll y$, where $\ll$ denotes the usual way-below relation, as defined in Section~\ref{sec:posets}. 
As noted with \cite[Remark~I-4.24]{Gierz03}, $y \in P$ is compact if and only if $\uparrow\!\! y$ is a Scott-open filter. 
The poset $P$ is \textit{compactly generated} if $P$ is sup-generated by the set of its compact elements. 
Now, if $P$ is a compactly generated directed-complete poset, equipped with the Scott or Lawson topology, then the topology is lower semiclosed, and $P$ has small open principal filters. 
Moreover, $P$ is itself strongly chain-complete. 
Thus, $P$ is inf-generated by the set of its completely relatively-maximal elements (hence by the set of its completely irreducible elements as well). 
This generalizes the well-known representation theorem for algebraic lattices due to Birkhoff and Frink \cite[Theorem~7]{BirkhoffFrink48}; see also  Ern\'e \cite[Proposition~7]{Erne87b}. 
\end{example}

\section{Local operators and representation of kit points}\label{sec:local}


In this final section, we delve deeper into the properties of preclosure operators to establish our third and last representation theorem. 
For this purpose, we introduce the notion of \textit{kit point}, which requires that every copoint admits a compact attaching point. 
Given an enriched qoset $(E, \leqslant, \cc)$ with adequate properties, we show that the set of kit points is $\cc^{\uparrow}$-closed and has the Krein--Milman property with respect to $\cc^{\uparrow}$, and that every kit point is sup-generated by an antichain of compact points, unique up to order-equivalence, and finite if $\cc$ is finitary. 
In passing, we explore the interplay between kitness, compactness, continuity, distributivity of $\cc$, and the Carath\'eodory number of $\cc$. 
As an additional warranty that our framework is meaningful enough, we prove that $\cc$-compact points agree with $\cc^{\uparrow}$-extreme points if $\cc$ is distributive, which reduces to the known coincidence between prime elements and irreducible elements in a distributive semilattice. 



\subsection{Locality, kit points and Carath\'eodory number}\label{subsec:local}

Let $(E, \cc)$ be a preclosure operator. 
Given $x \in E$, we say that $\cc$ is \textit{local at $x$} if the set
\[
\{ V \subseteq E : V \mbox{ $\cc$-closed and } V \not\ni x \},
\]
ordered by inclusion, is local, i.e.\ if $\cc$ is preinductive at $x$ and $x$ has finitely many $\cc$-copoints. 
We call $\cc$ \textit{local} if it is local at every $x \in E$. 

\begin{example}
If $E$ is a finite set, then every closure operator on $E$ is local. 
\end{example}

\begin{example}
Let $P$ be a qoset. 
The Alexandrov closure operator $\downarrow\!\! \cdot$ (resp.\ its dual $\uparrow\!\! \cdot$) is local, and every $x \in P$ has a \textit{unique} copoint, given by $P \setminus { \uparrow\!\! x }$ (resp.\ given by $P \setminus { \downarrow\!\! x }$). See Proposition~\ref{prop:ccpreinduc}. 
\end{example}

\begin{theorem}\label{thm:finitecopoints}
Let $(E, \cc, \sss)$ be a bi-preclosure space. 
If $\cc$ is local and idempotent and $\sss$ is finitary, then $\cc * \sss$ is finitary. 
\end{theorem}

\begin{proof}
Let $A \subseteq E$ and $x \in { \cc * \sss(A) }$. 
Let $V_1, \ldots, V_k$ be the $\cc$-copoints of $x$, and let $i \in \{ 1, \ldots, k \}$. 
We have $x \notin V_i$, by definition of $\cc$-copoints. 
Now, $x \in { \cc * \sss(A) }$, so $x \in { \cc(A \cap V_i) }$ or $x \in { \sss(A \setminus V_i) }$. 
If $x \in { \cc(A \cap V_i) }$, then $x \in { \cc(V_i) } = { V_i }$, a contradiction. 
So $x \in \sss(A \setminus V_i)$; since $\sss$ is finitary, there exists some finite subset $F_i$ of $A \setminus V_i$ such that $x \in \sss(F_i)$. 
Take $F := { \bigcup_{i = 1}^k F_i } \subseteq { A }$. 
We now show that $x \in { \cc * \sss(F) }$. 

So let $B \subseteq E$ such that $x \not\in \cc(F \cap B)$. 
Let $W$ be a $\cc$-copoint of $x$ containing $\cc(F \cap B)$. 
There is some $i_0 \in \{ 1, \ldots, k \}$ such that $W = V_{i_0} \supseteq \cc(F \cap B)$. 
We have $x \in \sss(F_{i_0})$, and ${ F_{i_0} } \subseteq { F \setminus V_{i_0} } \subseteq { F \setminus B }$, so $x \in \sss(F \setminus B)$. 
This proves that $x \in { \cc * \sss(F) }$. 
So $\cc * \sss$ is finitary. 
\end{proof}

\begin{remark}\label{rk:carat}
The \textit{Carath\'eodory number} $\car(\cc)$ of a finitary preclosure operator $\cc$ is defined as the least integer $d$ such that, whenever $x \in \cc(A)$, there exists some finite subset $F$ of $A$ with $|F| \leqslant d$ and $x \in \cc(F)$. 
From the previous proof, it is then easy to see that 
\[
\car(\cc * \sss) \leqslant \sup_{x \in E} | \textstyle{\Cop_{\cc}(x)} | \car(\sss), 
\] 
if $\cc$ is local and idempotent and $\sss$ is finitary. 
\end{remark}

Let $(E, \leqslant, \cc)$ be an enriched qoset. 
We call $x \in E$ a \textit{$\cc$-kit point} of $E$ if every $\cc$-copoint of $x$ has a $\cc$-compact $\cc$-attaching point. 
The subset made of $\cc$-kit points is denoted by $\kit_{\cc} E$. 
Obviously, 
\[
{ \textstyle{\cp_{\cc}} E } \subseteq { \textstyle{\kit_{\cc}} E },
\]
i.e., every $\cc$-compact point is a $\cc$-kit point. 
We say that $\cc$ is \textit{kitted} if every $x \in E$ is a $\cc$-kit point of $E$, equivalently if ${ \kit_{\cc} E } = E$. 

\begin{lemma}\label{lem:kit}
Let $(E, \leqslant, \cc)$ be an enriched qoset. 
Assume that $\cc$ is preinductive. 
Then $\kit_{\cc} E$ is $\cc^{\uparrow}$-closed. 
\end{lemma}

\begin{proof}
Let $x \in { \cc^{\uparrow}(\kit_{\cc} E) }$, and let $V$ be a $\cc$-copoint of $x$. 
By Theorem~\ref{thm:charac}, there is some $B \subseteq \kit_{\cc} E$ such that $x \in { \cc(B) \cap B^{\uparrow} }$. 
If $B \subseteq V$, then $\cc(B) \subseteq V$, so $x \in V$, a contradiction. 
Thus, there is some $b \in { B \setminus V }$. 
Since $x \in B^{\uparrow}$, we have $b \leqslant x$. 
Using the fact that $\cc$ is preinductive, there exists some $\cc$-copoint $W$ of $b$ containing $V$. 
Since $b \in { \kit_{\cc} E }$, there is some $\cc$-compact point $y$ such that $W \in \Cop_{\cc}(y)$. 
By Corollary~\ref{coro:compacts}, we have $W = { E \setminus { \uparrow\!\! y } }$. 
In particular, $y \leqslant b$. 
This yields $y \leqslant x$, so $x \notin W$. 
By maximality of $V$, we obtain $V = W \in \Cop_{\cc}(y)$. 
This shows that every $\cc$-copoint of $x$ has a $\cc$-compact $\cc$-attaching point, i.e.\ that $x \in { \kit_{\cc} E }$. 
So we have proved that ${ \cc^{\uparrow}(\kit_{\cc} E) } \subseteq { \kit_{\cc} E }$, i.e.\ $\kit_{\cc} E$ is $\cc^{\uparrow}$-closed. 
\end{proof}

The following result is inspired by Martinez \cite[Lemma~2.6]{Martinez72}. 

\begin{lemma}\label{lem:local}
Let $(E, \leqslant, \cc)$ be an enriched qoset. 
Assume that $\cc$ is finitary, and let $x \in E$. 
If $x$ is a $\cc$-kit point, then $\cc$ is local at $x$ (in particular, $| \Cop_{\cc}(x) | < \infty$). 
\end{lemma}

\begin{proof}
Since $\cc$ is finitary, we know by Proposition~\ref{prop:finitaryispreinduc} that $\cc$ is preinductive. 
So it remains to show that $x$ has finitely many $\cc$-copoints. 
By hypothesis, if $V \in \Cop_{\cc}(x)$, there is some $\cc$-compact point $y_V \in E$ with $V \in \Cop_{\cc}(y_V)$. 
By Corollary~\ref{coro:compacts}, $y_V$ being $\cc$-compact has at most one $\cc$-copoint, so $\Cop_{\cc}(y_V) = \{ V \}$. 
Take 
\[
{ A } := { \downarrow\!\! \{ y_V : V \in \textstyle{\Cop_{\cc}}(x) \} }, 
\]
and suppose that $x \notin { \overline{\cc}(A) }$. 
Let $V \in \Cop_{\cc}(x)$ containing $A$. 
Then $y_V \in { A } \subseteq { V }$, a contradiction. 
Thus, $x \in { \overline{\cc}(A) }$. 
Since $\cc$ is finitary, $\overline{\cc}$ is also finitary by Proposition~\ref{prop:progag}, so there exists some finite subset $F$ of $A$ such that $x \in \overline{\cc}(F)$. 
We write $F = \{ f_1, \ldots, f_k \}$ .
Let $V_1, \ldots, V_k \in \Cop_{\cc}(x)$ such that $f_i \leqslant y_{V_i}$, for all $i \in \{ 1, \ldots, k \}$. 
We show that every $\cc$-copoint of $x$ is one of $V_1, \ldots, V_k$. 
So let $W \in \Cop_{\cc}(x)$, and suppose that $y_{V_i} \in W$, for all $i \in \{ 1, \ldots, k \}$. 
As a $\cc$-closet subset, $W$ is a lower subset, so ${ \downarrow\!\! y_{V_i} } \subseteq { W }$, for all $i \in \{ 1, \ldots, k \}$. 
This implies that $F \subseteq W$, hence $x \in { \overline{\cc}(F) } \subseteq { W }$, a contradiction. 
Thus, there is some $j \in \{ 1, \ldots, k \}$ such that $y_{V_j} \notin W$. 
Using again that $\cc$ is preinductive, this yields $W \subseteq V_j$, because $\Cop_{\cc}(y_{V_j}) = \{ V_j \}$. 
Now, both $W$ and $V_j$ are $\cc$-copoints of $x$, so we obtain $W = V_j$, as required. 
\end{proof}


\begin{theorem}\label{thm:local}
Let $(E, \leqslant, \cc)$ be an enriched qoset. 
If $\cc$ is local and idempotent, then $\cc$ is finitary. 
Conversely, if $\cc$ is finitary and kitted, then $\cc$ is local, and its Carath\'eodory number satisfies 
\[
\car(\cc) = \sup_{x \in E} |\textstyle{\Cop_{\cc}(x)}|. 
\]
\end{theorem}

\begin{proof}
Assume that $\cc$ is local and idempotent. 
The closure operator $\langle \cdot \rangle$ defined in Proposition~\ref{prop:properties} is finitary, so $\cc = \cc * \langle \cdot \rangle$ is finitary by Theorem~\ref{thm:finitecopoints}. 

Conversely, assume that $\cc$ is finitary and kitted. 
Then $\cc$ is local, as a direct consequence of Lemma~\ref{lem:local}. 
Also, following the proof of Lemma~\ref{lem:local}, it is easily seen that, if $\car(\cc) \leqslant d$, then $\sup_{x \in E} |\Cop_{\cc}(x) | \leqslant d$. 
We deduce that 
\[
\car(\cc) = \sup_{x \in E} |\textstyle{\Cop_{\cc}(x)}|, 
\]
thanks to Remark~\ref{rk:carat}. 
\end{proof}


\subsection{Continuity and distributivity}

Let $(E, \leqslant, \cc)$ be an enriched qoset. 
Following Poncet \cite{Poncet22}, we define the \textit{way-below relation} $\ll_{\cc}$ associated to $\cc$ by $x \ll_{\cc} y$ if 
\[
{ y \in \cc(A) } \Rightarrow { x \in { \downarrow\!\! A } },
\]
for all $A \subseteq E$. 
Note that $x \in E$ is $\cc$-compact if and only if $x \ll_{\cc} x$. 
We write $\twoheaddownarrow_{\cc} x$ for the subset $\{ y \in E : y \ll_{\cc} x \}$. 

\begin{lemma}\label{lem:basic}
Let $(E, \leqslant, \cc)$ be an enriched qoset. Then 
\begin{enumerate}
	\item\label{lem:basic1} the way-below relation is transitive;
	\item\label{lem:basic2} $x \ll_{\cc} y$ implies $x \leqslant y$;
	\item\label{lem:basic3} $x \leqslant y \ll_{\cc} z$ implies $x \ll_{\cc} z$; 
	\item\label{lem:basic4} $x \ll_{\cc} y \leqslant z$ implies $x \ll_{\cc} z$; 
	\item\label{lem:basic5} ${ \twoheaddownarrow_{\cc} x } \subseteq { \downarrow\!\! x } \subseteq { \cc(x) }$. 
\end{enumerate}
\end{lemma}

\begin{proof}
\eqref{lem:basic1} can be obtained by combining \eqref{lem:basic2} and \eqref{lem:basic3}. 

\eqref{lem:basic2}. 
If $x \ll_{\cc} y$, then from $y \in \cc(y)$ we get $x \in { \downarrow\!\! y }$, i.e.\ $x \leqslant y$. 

\eqref{lem:basic3}. 
Suppose that $x \leqslant y \ll_{\cc} z$, and let $A$ be a subset such that $z \in \cc(A)$. 
	Then $y \in { \downarrow\!\! A }$, so $x \in { \downarrow\!\! y } \subseteq { \downarrow\!\! A }$. 
This proves that $x \ll_{\cc} z$. 

\eqref{lem:basic4}. 
Suppose that $x \ll_{\cc} y \leqslant z$, and let $A$ be a subset such that $z \in \cc(A)$. 
Then $y \in { \downarrow\!\! z } \subseteq { \downarrow\!\!\cc(A) } = { \cc(A) }$, so $x \in { \downarrow\!\! A }$. 
This proves that $x \ll_{\cc} z$. 

\eqref{lem:basic5}. 
The inclusion ${ \twoheaddownarrow_{\cc} x } \subseteq { \downarrow\!\! x }$ is a rephrasing of \eqref{lem:basic2}, and the inclusion ${ \downarrow\!\! x } \subseteq { \cc(x) }$ is a consequence of $\cc(x)$ being a lower subset. 
\end{proof}

Let $(E, \leqslant, \cc)$ be an enriched qoset. 
Then $\cc$ is called \textit{continuous} if
\[
x \in \cc(\twoheaddownarrow_{\cc} x),
\]
for all $x \in E$ (equivalently, if $\cc$ commutes with arbitrary intersections of lower subsets, see Ern\'e and Wilke \cite[Theorem~2.5]{Erne83}). 
Moreover, $\cc$ is called \textit{algebraic} if 
\[
x \in \cc(\downarrow\!\! ({ \twoheaddownarrow_{\cc} x } \cap { \textstyle{\cp_{\cc}} E })),
\]
for all $x \in E$, and \textit{distributive} if
\[
{ x \in \cc(A) } \Rightarrow { x \in \cc({ \downarrow\!\! x } \cap { \downarrow\!\! A }) },
\]
for all $x \in E$ and $A \subseteq E$; see \cite[Theorem~2.1]{Erne83} for a characterization of distributive closure operators. 


\begin{theorem}\label{thm:prime}
Let $(E, \leqslant, \cc)$ be an enriched qoset. 
Then
\[
{ \textstyle{\cp_{\cc}} E } \subseteq { K \cap { \textstyle{\ex_{\cc^{\uparrow}}} E } } \subseteq { \textstyle{\ex_{\cc^{\uparrow}}} K }, 
\]
where $K := { \kit_{\cc} E }$, and all three sets coincide if $\cc$ is idempotent and preinductive. 
Moreover, consider the following assertions:
\begin{enumerate}
  \item\label{thm:prime1} $\cc$ is idempotent, preinductive, and kitted;
  \item\label{thm:prime2} $\cc$ is algebraic;
  \item\label{thm:prime3} $\cc$ is continuous;
  \item\label{thm:prime5} $\cc$ is distributive;
  \item\label{thm:prime6} ${ \cp_{\cc} E } = { \ex_{\cc^{\uparrow}} E }$. 
\end{enumerate}
Then \eqref{thm:prime1} $\Rightarrow$ \eqref{thm:prime2} $\Rightarrow$ \eqref{thm:prime3} $\Rightarrow$ \eqref{thm:prime5} $\Rightarrow$ \eqref{thm:prime6}. 
\end{theorem}

\begin{proof}
${ \cp_{\cc} E } \subseteq { K \cap \ex_{\cc^{\uparrow}} E }$. 
Let $x \in { \cp_{\cc} E }$, and suppose that $x \in \cc^{\uparrow}(A)$, for some $A \subseteq E$. 
By Theorem~\ref{thm:charac}, there is some $B \subseteq A$ with $x \in { \cc(B) \cap B^{\uparrow} }$. 
Since $x \in \cp_{\cc} E$, we have $x \in { \downarrow\!\! B }$, so $x \leqslant b$, for some $b \in B$. 
Using the fact that $x \in B^{\uparrow}$, we have also $b \leqslant x$. 
So $x \sim b$, hence $x \in [A]$. 
This proves that $x \in { \ex_{\cc^{\uparrow}} E }$. 
Since every $\cc$-compact point is a $\cc$-kit point, we conclude that $x \in { K \cap \ex_{\cc^{\uparrow}} E }$. 

${ K \cap \ex_{\cc^{\uparrow}} E } \subseteq { \ex_{\cc^{\uparrow}} K }$. 
This inclusion is a direct consequence of the definitions. 

${ \ex_{\cc^{\uparrow}} K } \subseteq { \cp_{\cc} E }$ if $\cc$ is idempotent and preinductive. 
Let $x \in { \ex_{\cc^{\uparrow}} K }$. 
By Proposition~\ref{prop:cpconvol}, there is some $\cc$-copoint $V$ of $x$ such that $x \in { \Min (K \setminus V) }$. 
Since $x$ is a $\cc$-kit point, there is a $\cc$-compact point $y$ such that $V \in \Cop_{\cc}(y)$. 
From Corollary~\ref{coro:compacts}, $V$ is the unique $\cc$-copoint of $x$ and $V = { E \setminus { \uparrow\!\! y } }$. 
In particular, $y \leqslant x$. 
Since $x \in { \Min (K \setminus V) }$ and $y \in { K \setminus V }$, this yields $x \sim y$. 
So $x$ is itself $\cc$-compact, i.e.\ $x \in { \cp_{\cc} E }$, as required. 

\eqref{thm:prime1} $\Rightarrow$ \eqref{thm:prime2}. 
Let $C$ denote the set of $\cc$-compact points, i.e.\ $C := { \cp_{\cc} E }$. 
Let $x \in E$, and suppose that $x \notin { \cc(\twoheaddownarrow_{\cc} x \cap C) }$. 
Since $\cc$ is idempotent and preinductive, there exists some $\cc$-copoint $V$ of $x$ containing $\cc(\twoheaddownarrow_{\cc} x \cap C)$. 
By assumption, there is some $y \in C$ such that $V \in \Cop_{\cc}(y)$. 
By Corollary~\ref{coro:compacts}, $V = { E \setminus \uparrow\!\! y }$. 
Now, $x \notin V$, so $y \leqslant x$. 
This yields $y \ll_{\cc} y \leqslant x$, hence $y \in { \twoheaddownarrow_{\cc} x \cap C }$. 
So $y \in V$, a contradiction. 
This proves that $x \in { \cc(\twoheaddownarrow_{\cc} x \cap C) }$, for all $x \in E$. 
So $\cc$ is algebraic. 

\eqref{thm:prime2} $\Rightarrow$ \eqref{thm:prime3}. 
Let $x \in E$. 
Since $\cc$ is algebraic, $x \in \cc(\downarrow\!\! ({ \twoheaddownarrow_{\cc} x } \cap { \textstyle{\cp_{\cc}} E }))$. 
This implies $x \in \cc(\downarrow\!\! \twoheaddownarrow_{\cc} x)$. 
Since $\twoheaddownarrow_{\cc} x$ is a lower subset by Lemma~\ref{lem:basic}\eqref{lem:basic3}, we obtain $x \in \cc(\twoheaddownarrow_{\cc} x)$, for all $x \in E$. 
So $\cc$ is continuous. 

\eqref{thm:prime3} $\Rightarrow$ \eqref{thm:prime5}. 
Let $A \subseteq E$ and $x \in \cc(A)$. 
From the definition of $\ll_{\cc}$, we get ${ \twoheaddownarrow_{\cc} x } \subseteq { \downarrow\!\! A }$. 
We also have ${ \twoheaddownarrow_{\cc} x } \subseteq { \downarrow\!\! x }$ by Lemma~\ref{lem:basic}, so ${\twoheaddownarrow_{\cc} x } \subseteq { { \downarrow\!\! x } \cap { \downarrow\!\! A } }$. 
Since $\cc$ is continuous, we obtain $x \in { \cc(\twoheaddownarrow_{\cc} x) } \subseteq { \cc({ \downarrow\!\! x } \cap { \downarrow\!\! A }) }$, as required. 

\eqref{thm:prime5} $\Rightarrow$ \eqref{thm:prime6}. 
Let $x \in { \ex_{\cc^{\uparrow}} E }$, and suppose that $x \in \cc(A)$, for some $A \subseteq E$. 
Since $\cc$ is distributive, we have $x \in \cc({ \downarrow\!\! x } \cap { \downarrow\!\! A })$. 
Take $B := { { \downarrow\!\! x } \cap { \downarrow\!\! A } }$. 
Then $x \in { \cc(B) \cap B^{\uparrow}}$, so $x \in \cc^{\uparrow}(B)$ by Theorem~\ref{thm:charac}. 
Using the fact that $x \in { \ex_{\cc^{\uparrow}} E }$, this yields $x \in [B]$. 
So there is some $b \in B$ with $x \sim b$. 
In particular, $x \leqslant b \in { \downarrow\!\! A }$, so $x \in { \downarrow\!\! A }$. 
This proves that $x \in { \cp_{\cc} E }$. 
\end{proof}

\begin{proposition}\label{prop:kitted}
Let $(E, \leqslant, \cc)$ be an enriched qoset. 
Assume that 
\begin{itemize}
  \item ${ \cp_{\cc} E } = { \ex_{\cc^{\uparrow}} E }$, and
  \item $E \setminus V$ is preinductive in $(E, \geqslant)$, for all $\cc$-copoints $V$. 
\end{itemize}
Then $\cc$ is kitted. 
\end{proposition}

\begin{proof}
Let $x \in E$, and let $V \in \Cop_{\cc}(x)$. 
Since $E \setminus V$ is preinductive in $(E, \geqslant)$, there is some $y \in \Min (E \setminus V)$ with $y \leqslant x$. 
This implies that $V \subseteq { E \setminus { \uparrow\!\! y } }$. 
Let us show that $y \in { \ex_{\cc^{\uparrow}} E }$. 
To this end, let $B \subseteq E$ with $y \in { \cc(B) \cap B^{\uparrow} }$. 
If $B \subseteq V$, then $\cc(B) \subseteq V$, so $y \in V$, a contradiction. 
Thus, there is some $b \in { B \setminus V }$. 
Since $y \in B^{\uparrow}$, we have $b \leqslant y$. 
Using that $y \in  \Min (E \setminus V)$, we obtain $y \sim b$. 
This shows that $y \in [B]$. 
So $y \in { \ex_{\cc^{\uparrow}} E }$, by Proposition~\ref{prop:extrcharac}. 
By assumption, ${ \ex_{\cc^{\uparrow}} E } = { \cp_{\cc} E }$, so $y$ is $\cc$-compact, so that $E \setminus { \uparrow\!\! y }$ is $\cc$-closed by Corollary~\ref{coro:compacts}. 
From the maximality of $V$, we get $V = { E \setminus { \uparrow\!\! y } }$, and $V \in \Cop_{\cc}(y)$. 
This shows that $x$ is a $\cc$-kit point. 
So $\cc$ is kitted. 
\end{proof}

\begin{example}[Example~\ref{ex:infconvex4} continued]\label{ex:infconvex5}
Let $P$ be a qoset. 
Take $\cc := { \langle \cdot \rangle_{\mathrsfs{U}} }$, where $\mathrsfs{U}$ denotes the collection of upper, inf-closed subsets. 
As already seen, the set $\cp_{\cc} P$ (considered for the dual order $\geqslant$) agrees with the prime elements of $P$, and the set $\ex_{\cc^{\downarrow}} P$ agrees with the irreducible elements of $P$. 
Thus, the inclusion ${ \cp_{\cc} P } \subseteq { \ex_{\cc^{\downarrow}} P }$ amounts to the well-known fact that every prime element is irreducible. 
By the previous theorem, the converse holds if $\cc$ is distributive (in a dual sense). 
In the special case where $P$ is a lattice (a poset in which every pair $\{ x, y \}$ has a sup $x \vee y$ and an inf $x \wedge y$), it is easy to show that this notion of distributivity amounts to the usual property 
\[
{ x \vee (y \wedge z) } = { (x \vee y) \wedge (x \vee z) }, 
\]
for all $x, y, z \in P$. 
\end{example}

\subsection{Representation theorem for kit points}\label{sec:og}

In a qoset, an \textit{antichain} is a subset $A$ such that ${ a \leqslant a' } \Rightarrow { a = a' }$, for all $a, a' \in A$. 

\begin{theorem}[Representation Theorem III]\label{thm:rep3}
Let $(E, \leqslant, \cc)$ be an enriched qoset. 
Assume that $\cc$ is idempotent and preinductive. 
Then $K := \kit_{\cc} E$ is $\cc^{\uparrow}$-closed and has the $\cc^{\uparrow}$-KMp, i.e.\ 
\begin{align}\label{eq:kmkit}
{ K } = { \cc^{\uparrow}(\textstyle{\ex_{\cc^{\uparrow}}} K) }. 
\end{align}
Moreover, for every $x \in { \kit_{\cc} E }$, 
\begin{itemize}
  \item there exists an antichain $Y_x \subseteq { \cp_{\cc} E }$ such that 
$
x \in { { \cc(Y_x) } \cap { Y_x^{\uparrow} } }
$; 
  \item this antichain is unique up to order-equivalence; 
  \item if $\cc$ is finitary, every such antichain is finite. 
\end{itemize}
In addition, if $\cc$ separates points, then $\kit_{\cc} E$ is sup-generated by the subset of $\cc$-compact points, and $x \in Y_x^{\vee}$, for all $x \in { \kit_{\cc} E }$. 
\end{theorem}


\begin{proof}
The fact that $K := \kit_{\cc} E$ is $\cc^{\uparrow}$-closed is Lemma~\ref{lem:kit}. 
Recall that, by Theorem~\ref{thm:prime}, ${ \cp_{\cc} E } = { \ex_{\cc^{\uparrow}} K }$. 

Let $x \in K$, and let $\{ V_j \}_{j \in J}$ be the collection of distinct $\cc$-copoints of $x$. 
By definition of $K$, there is some $\cc$-compact point $y_j$ of $E$ such that $V_j \in \Cop_{\cc}(y_j)$. 
By Corollary~\ref{coro:compacts}, we have $V_j = { E \setminus { \uparrow\!\! y_j } }$. 
In particular, $y_j \leqslant x$, for all $j \in J$. 
Take $Y_x := \{ y_j \}_{j \in J}$. 
Then $x \in Y_x^{\uparrow}$. 
Let us show that $x \in \cc(Y_x)$. 
Suppose this is not the case. 
Since $\cc$ is idempotent and preinductive, there is some $j_0 \in J$ such that $V_{j_0}$ contains $Y_x$. 
This yields $y_{j_0} \in { Y_x } \subseteq { V_{j_0} }$, a contradiction. 
So $x \in { \cc(Y_x) \cap Y_x^{\uparrow} }$, hence $x \in \cc^{\uparrow}(Y_x)$ by Theorem~\ref{thm:charac}. 
Since $Y_x \subseteq { \cp_{\cc} E } = { \ex_{\cc^{\uparrow}} K }$, we also have $x \in { \cc^{\uparrow}(\ex_{\cc^{\uparrow}} K) }$ for all $x \in K$. 
So Assertion~\eqref{eq:kmkit} holds. 

Let us show that $Y_x$ is an antichain. 
To this end, suppose that $y_i \leqslant y_j$, for some $i, j \in J$. 
Then ${ V_i } = { E \setminus { \uparrow\!\! y_i } } \subseteq { E \setminus { \uparrow\!\! y_j } } = { V_j }$. 
Now, $V_i$ and $V_j$ are $\cc$-copoints of $x$, so $V_i = V_j$, which yields $i = j$, so $y_i = y_j$. 

Now we prove that the representing antichain is unique, up to order-equivalence. 
So let $Z_x$ be an antichain of $\cc$-compact points such that $x \in { { \cc(Z_x) } \cap { Z_x^{\uparrow} } }$, and let $z \in Z_x$. 
Then $z \in { \downarrow\!\! x } \subseteq { \downarrow\!\! \cc(Y_x) } = { \cc(Y_x) }$. 
Since $z$ is $\cc$-compact, there is some $y \in Y_x$ with $z \leqslant y$. 
Similarly, there is some $z' \in Z_x$ with $y \leqslant z'$. 
Using that $Z_x$ is an antichain, we get $z = z'$, so $z \sim y$. 
From that point, it is easy to deduce that there is a bijection $\sigma : Z_x \to Y_x$ such that $z \sim \sigma(z)$, for all $z \in Z_x$. 

Assume that $\cc$ is finitary. 
Then Lemma~\ref{lem:local} applies: the collection $\{ V_j \}_{j \in J}$ is finite, so the set $Y_x$ is finite. 

Assume that $\cc$ separates points. 
We have seen that $x \in { \cc(Y_x) \cap Y_x^{\uparrow} }$. 
By Proposition~\ref{prop:sep}, we deduce that $x \in { \mathfrak{d}(Y_x) \cap Y_x^{\uparrow} } = { Y_x^{\vee} }$. 
This shows that $x$ is a sup of $Y_x$, hence also a sup of ${ \downarrow\!\! x } \cap { \cp_{\cc} E }$, as required. 
\end{proof}


\begin{example}[Example~\ref{ex:infconvex5} continued]\label{ex:infconvex6}
We discuss here the representation by prime elements in Noetherian distributive lattices. 
Let $L$ be a distributive lattice, i.e.\ a semilattice in which every pair $\{ x, y \}$ has a sup $x \vee y$, and such that ${ x \vee (y \wedge z) } = { (x \vee y) \wedge (x \vee z) }$, for all $x, y, z \in L$. 
We assume that $L$ is Noetherian, in the sense that every order-ideal of $L$ is principal (see Remark~\ref{rk:noetherian} on Noetherian posets). 
In particular, $L$ has a top. 

Consider the finitary closure operator $\cc = \langle \cdot \rangle_{\mathrsfs{U}}$, where $\mathrsfs{U}$ denotes the collection of upper, inf-closed subsets. 
Note that $\cc$ dually separates points. 
Moreover, $L$ being distributive, the set of $\cc$-compact points of $(L, \geqslant)$, which agrees with the set of prime elements of $L$, coincides with $\ex_{\cc^{\downarrow}} L$. 
Also, $L$ being Noetherian, every subset of $L$ satisfies the ascending chain condition, so is preinductive by Lemma~\ref{lem:preinductive}. 
By Proposition~\ref{prop:kitted}, $\cc$ is kitted (for the dual quasiorder). 
Therefore, by the previous theorem, every element of $L$ is the inf of a finite antichain of prime elements, and this antichain is unique. 
\end{example}

\begin{example}[Example~\ref{ex:ring2} continued]\label{ex:ring3}
Let $R$ be a PID, quasi-ordered by the dual divisibility relation $\leqslant$. 
Then the quotient $R / {\sim}$ of $R$ by the order-equivalence relation $\sim$ (which identifies associate elements) is a distributive lattice. 
Moreover, $R / {\sim}$ is a Noetherian poset. 
By the previous example, every element of $R$ is thus an inf of a (possibly empty) finite antichain of prime elements of $(R, \leqslant)$, and this antichain is unique up to order-equivalence. 
If we admit that every prime element of $(R, \leqslant)$ is order-equivalent to the power of some prime element of $(R, \times)$, we obtain the classical result that $R$ is a unique factorization domain. 
\end{example}

\begin{acknowledgements}
I gratefully thank Prof.\ Jimmie D.\ Lawson; the personal communication \cite{Lawson11} he sent me at the time I was closing my PhD thesis is the essence of Section~\ref{subsec:repthm}, and was the starting point of this paper. 
\end{acknowledgements}

\bibliographystyle{plain}

\def\cprime{$'$} \def\cprime{$'$} \def\cprime{$'$} \def\cprime{$'$}
  \def\ocirc#1{\ifmmode\setbox0=\hbox{$#1$}\dimen0=\ht0 \advance\dimen0
  by1pt\rlap{\hbox to\wd0{\hss\raise\dimen0
  \hbox{\hskip.2em$\scriptscriptstyle\circ$}\hss}}#1\else {\accent"17 #1}\fi}
  \def\ocirc#1{\ifmmode\setbox0=\hbox{$#1$}\dimen0=\ht0 \advance\dimen0
  by1pt\rlap{\hbox to\wd0{\hss\raise\dimen0
  \hbox{\hskip.2em$\scriptscriptstyle\circ$}\hss}}#1\else {\accent"17 #1}\fi}

\end{document}